\documentclass[a4paper, 11pt]{article}
\usepackage[utf8]{inputenc}     
\usepackage[T1]{fontenc}
\usepackage{lmodern}
\usepackage{graphicx}
\usepackage[english]{babel}
\usepackage{subfigure}
\usepackage{amsmath, amsfonts, amssymb,amsthm} 
\usepackage{hyperref} 
\usepackage{textcomp}
\usepackage{mathrsfs}
\usepackage{multirow}
\usepackage{float}

\newtheorem{theorem}{Theorem}
\newtheorem{prop}{Proposition}
\newtheorem{lem}{Lemma}
\newtheorem{corollary}{Corollary}
\newtheorem{remark}{Remark}

\newcommand{\bydef}{\,\stackrel{\mbox{\tiny\textnormal{\raisebox{0ex}[0ex][0ex]{def}}}}{=}\,}
\newcommand{\eps}{\varepsilon}
\newcommand{\cC}{\mathcal{C}}

\begin{document}

\title{Global bifurcation diagram of steady states of systems of PDEs via rigorous numerics: a 3-component reaction-diffusion system}

\author{Maxime Breden \thanks{CMLA, ENS Cachan, 61 av. du Pr\'esident Wilson, 94235 Cachan, France ({\tt mbreden@ens-cachan.fr}).}
\and Jean-Philippe Lessard\thanks {D\'epartement de Math\'ematiques et de Statistique, Universit\'e Laval, 1045 avenue de la M\'edecine, Qu\'ebec, QC, G1V0A6, Canada ({\tt jean-philippe.lessard@mat.ulaval.ca}). This author was supported by NSERC.}
\and Matthieu Vanicat \thanks{CMLA, ENS Cachan, 61 av. du Pr\'esident Wilson, 94235 Cachan, France ({\tt mvanicat@ens-cachan.fr}).}}

\date{January 2013}

\maketitle

\vspace{-.5cm}

\begin{abstract}
In this paper, we use rigorous numerics to compute several global smooth branches of steady states for a system of three reaction-diffusion PDEs introduced by Iida et al. [J. Math. Biol., {\bf 53}, 617--641 (2006)] to study the effect of cross-diffusion in competitive interactions. An explicit and mathematically rigorous construction of a global bifurcation diagram is done, except in small neighborhoods of the bifurcations. The proposed method, even though influenced by the work of van den Berg et al. [Math. Comp., {\bf 79}, 1565--1584 (2010)], introduces new analytic estimates, a new {\em gluing-free} approach for the construction of global smooth branches and provides a detailed analysis of the choice of the parameters to be made in order to maximize the chances of performing successfully the computational proofs.
\end{abstract}

\begin{center}
{\bf \small Keywords} \\ \vspace{.05cm}
{ \small Equilibria of systems of PDEs $\cdot$ Bifurcation diagram $\cdot$ Contraction mapping $\cdot$ Rigorous numerics}
\end{center}

\begin{center}
{\bf \small Mathematics Subject Classification (2010): 65N15 $\cdot$ 37M20 $\cdot$ 35K55} 
\end{center}

\section{Introduction} \label{sec:introduction}
Establishing the existence of non constant bounded solutions to parameter dependent systems of reaction-diffusion PDEs is a classical problem in nonlinear analysis. Methods like singular perturbation theory \cite{MR1636690}, local bifurcation theory (see \cite{MR2154033} and the references therein), the local theory of Crandall and Rabinowitz \cite{MR0288640}, and the Leray-Schauder degree theory can be used to prove existence of such non constant solutions. However, it appears difficult in practice to use such results to answer specific questions about the solutions, e.g. determining the number of (or a lower bound on the number of) non constant co-existing steady states. If the parameter dependent system under study undergoes a bifurcation from a trivial solution, it also appears difficult to determine rigorously the behaviour of the solutions on the global bifurcating branches. The global bifurcation theorem of Rabinowitz \cite{MR0301587}, although powerful and general, can convey only partial information about the global behaviour of the branches. 

While the development of theoretical knowledge about existence of solutions of systems of PDEs is slow, meticulous and often hard to grasp for non experts, there exist several user friendly bifurcation softwares that can efficiently produce tremendous amount of bounded approximate solutions. However, with any numerical methods, there is the question of validity of the outputs. 

The goal of this paper is to propose, in the context of studying non constant steady states for systems of PDEs, a rigorous computational method in an attempt to fill the gap between the above mentioned theoretical and computational advances. More specifically, we compute rigorously a global bifurcation diagram of steady states for the 3-component reaction-diffusion system arising in population dynamics introduced in \cite{MR2251792}, and given by

\begin{equation}
\label{eq:mimura_system}
  \left\{
      \begin{aligned}
        & \partial_t x = d\Delta x + (r_1-a_1(x+y)-b_1z)x + \frac{1}{\varepsilon}\left(y\left(1-\frac{z}{N}\right)-x\frac{z}{N}\right),  \\
        & \partial_t y = (d+\beta N)\Delta y + (r_1-a_1(x+y)-b_1z)y - \frac{1}{\varepsilon}\left(y\left(1-\frac{z}{N}\right)-x\frac{z}{N}\right) , \\
        & \partial_t z = d\Delta z + (r_2-b_2(x+y)-a_2z)z ,
              \end{aligned}
    \right.
\end{equation}
where $x=x(\xi,t)$, $y=y(\xi,t)$ and $z=z(\xi,t)$ are defined for $t >0$, $\xi \in [0,1]$ with Neumann boundary conditions and with the following fixed numerical values $a_1=3$, $a_2=3$, $b_1=1$, $b_2=1$, $r_1=5$, $r_2=2$, $\beta =3$, $\varepsilon =0.01$ and $N=1$. We consider the diffusion $d$ as a free parameter. The above choice of fixed parameter values is chosen so that, when the parameter $d$ varies, Turing's instability (e.g. \cite{turing}) can be observed. More precisely, varying the diffusion $d$ can destabilize the constant steady state solution given by $(x,y,z)=(\frac{91}{64},\frac{13}{64},\frac{1}{8})$, which is stable for the corresponding finite dimensional ODE model without diffusion terms. Hence, by changing the diffusion $d$, interesting non trivial bounded stationary patterns can arise as a result of Turing's instability, which is one of the most important mechanisms of pattern formation. Note that considering a 3-component reaction-diffusion system allows having Turing instability with quadratic reaction terms, as opposed to the standard cubic case. A nice consequence of this choice of system is that the nonlinear terms are easier to estimate than in the cubic case.

The system of PDEs \eqref{eq:mimura_system} has been introduced in \cite{MR2251792} with the goal of studying the (theoretically harder to study) cross-diffusion model for the competitive interaction between two species
\begin{equation}
\label{eq:mimura_cross_diffusion_system}
  \left\{
      \begin{aligned}
        & \partial_t u_1(\xi,t) = \Delta [(d_1+\alpha u_2)u_1] + (r_1-a_1 u_1 -b_1 u_2) u_1 ,    \\
        & \partial_t u_2 (\xi,t) = \Delta [(d_2+\beta u_1)u_2]  + (r_2-b_2u_1-a_2 u_2)u_2,    \\
      \end{aligned}
    \right.
\end{equation}
where $u_1=u_1(\xi,t), u_2=u_2(\xi,t)$ are defined for $t >0$, $\xi \in [0,1]$, and satisfy Neumann boundary conditions. As already mentioned in \cite{MR2251792}, model \eqref{eq:mimura_cross_diffusion_system} falls into quasi-linear parabolic systems so that even the existence problem of solutions is not trivial and has been investigated by several authors (e.g. see \cite{MR1242579,MR1974423,MR1616969} and the references therein). It is shown in \cite{MR2251792} that the solutions of \eqref{eq:mimura_cross_diffusion_system} can be approximated by those of \eqref{eq:mimura_system} in a finite time interval if the solutions are bounded and provided $\varepsilon$ is small enough. It is also proved in \cite{MR2437576} that the steady states of the reaction-diffusion system \eqref{eq:mimura_system} approximates the steady states of the cross-diffusion system \eqref{eq:mimura_cross_diffusion_system} as $\varepsilon$ goes to $0$. Hence, developing a rigorous computational approach to prove existence of non constant steady states of \eqref{eq:mimura_system} seems to be an interesting problem, as it may shed some light on how to rigorously study non constant steady states of a cross-diffusion model (see Theorem~\ref{thm:cross_diffusion}). Here is a first result, whose proof can be found in Section~\ref{sec:proofs}.

\begin{theorem}[\bf Rigorous computation of a global bifurcation diagram] \label{thm:bif_diagram}
Except in small neighborhoods of the bifurcations, each point in Figure~\ref{japonais} represents exactly one steady state for \eqref{eq:mimura_system}, each curve on the diagram is smooth and between the apparent bifurcations (in black dots), there are no secondary bifurcations of steady states.
\end{theorem}

\begin{figure}[h]
\label{japonais}
\begin{center}
\includegraphics[width=10cm]{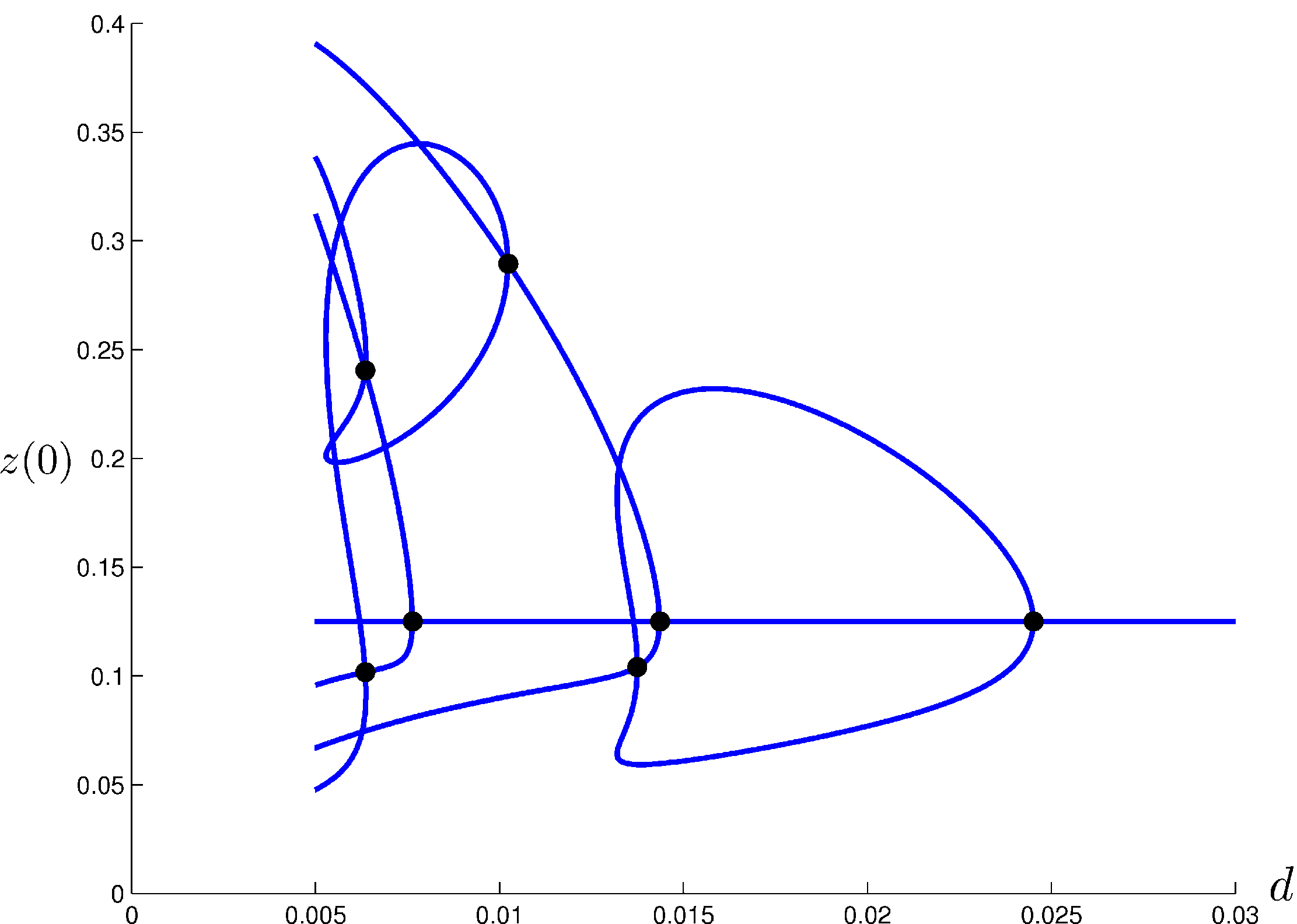}
\end{center}
\vspace{-.5cm}
\caption{\small The rigorously computed bifurcation diagram of Theorem~\ref{thm:bif_diagram}. Note that the apparent bifurcations which appear in black are not proved rigorously. The horizontal axis represents the diffusion parameter $d$ while the vertical axis represents the value $z(0)$ of the steady state $(x,y,z)$ of \eqref{eq:mimura_system}. The apparent intersections which are not denoted by black dots are not bifurcation points, e.g. see Figure~\ref{espace1} (c),(d) for a geometrical interpretation of apparently close solutions.
}
\end{figure}

Let us briefly introduce the ideas behind the rigorous method. First, the steady states $(x,y,z)$ defined on $[0,1]$ of \eqref{eq:mimura_system} with Neumann boundary conditions can be extended periodically on $[-1,1]$ and then expanded as Fourier series of the form $x(\xi) = \frac{x_0}{2} + \sum_{k \ge 1} x_k \cos ( k \pi \xi)$, $y(\xi) = \frac{y_0}{2} + \sum_{k \ge 1} y_k \cos ( k \pi \xi)$ and $z(\xi) = \frac{z_0}{2} + \sum_{k \ge 1} z_k \cos ( k \pi \xi)$. Plugging the expansions of $x$, $y$ and $z$ in \eqref{eq:mimura_system} and computing the Fourier coefficients of the resulting expansions yield the following infinite set of algebraic equations to be satisfied

\begin{eqnarray} \label{fourier}
      &&  \hspace{-.7cm} f_{n_x}(U) \bydef (r_1-d(\pi n)^2)x_n + \frac{1}{\varepsilon}y_n - a_1 [x^2]_n - a_1 [x\ast y]_n - (b_1+\frac{1}{\varepsilon N})[x\ast z]_n - \frac{1}{\varepsilon N}[y\ast z]_n =0, \nonumber \\ \nonumber
       &&  \hspace{-.7cm} f_{n_y}(U) \bydef ((r_1 - \frac{1}{\varepsilon} - (d+\beta N) (\pi n)^2 )y_n - a_1 [y^2]_n - a_1 [x\ast y]_n - (b_1-\frac{1}{\varepsilon N})[y\ast z]_n + \frac{1}{\varepsilon N}[x\ast z]_n=0, \\
       &&  \hspace{-.7cm} f_{n_z}(U) \bydef (r_2 - d (\pi n)^2 )z_n - a_2 [z^2]_n - b_2 [x\ast z]_n - b_2 [y\ast z]_n = 0,
\end{eqnarray}
where $n\ge 0$, $[\phi \ast \varphi]_n \bydef \frac{1}{2}\sum_{k\in\mathbb{Z}}\phi_{\vert k \vert} \varphi_{\vert n-k \vert}$ and where $\phi^2 = \phi \ast \phi$. Given $n \ge 0$, let $u_n \bydef (x_n,y_n,z_n)$ the $n$-th Fourier coefficients of $(x,y,z)$ and let $u \bydef (u_n)_{n \ge 0}$. Define $U=(d,u)$ and let $f_n(U) \bydef (f_{n_x}(U),f_{n_y}(U),f_{n_z}(U))$, where each component is defined by \eqref{fourier}. Finally define $f(U) = (f_n(U))_{n \ge 0}$. In order to compute steady states of  \eqref{eq:mimura_system}, we will be looking for solutions $U$ of $f(U)=0$ in a Banach space of fast decaying coefficients. Let us be more explicit about this. As in \cite{MR2630003}, we define weight functions ($q>1$)
\begin{equation}\label{e:weights}
  \omega^q_n = \left\{ 
  \begin{array}{ll}
  1, &  n =0; \\
  n^q, & n \geq 1, \\
  \end{array}
 \right.
\end{equation}
which are used to define, for $u=(u_n)_{n \ge 0}$ as defined above, the norm

\begin{equation}
\label{norme}
\left\Vert u \right\Vert_q =  \sup\limits_{n \in \mathbb{N}} \left\vert u_n \right\vert_{\infty} \omega_n^q,
\end{equation}
where $q>1$ is a decay rate and $ \left\vert u_n \right\vert_{\infty} = \max (\vert x_n \vert, \vert y_n \vert, \vert z_n \vert) $. Define
\begin{equation} \label{eq:norm_q}
\left\Vert U \right\Vert_q = \max \left\{ \frac{\vert d \vert}{\rho} , \left\Vert u \right\Vert_q \right\}
\end{equation}
where $\rho$ is a constant whose value will be chosen later, and
%
$\Omega_q = \left\lbrace U=(d,u) \ \vert \ \left\Vert U \right\Vert_q < \infty\right\rbrace$,
%
a Banach space with norm \eqref{eq:norm_q} of sequences decreasing to zero at least as fast as $n^{-q}$, as $n \rightarrow \infty$.

\begin{lem} \label{lemma:equivalence}
Fix a diffusion parameter $d$ and a decay rate $q >1$. Using the above construction, $U=(d,u) \in \Omega_{q}$ is a solution of $f(U)=0$ if and only if $(x,y,z)$ is a strong $C^2$-solution of the stationary Neumann problem of \eqref{eq:mimura_system}.
\end{lem}

\begin{proof}

Assume that $U=(d,u) \in \Omega_q$ is a solution of $f(U)=0$. Since $\Omega_{q}$ is a Banach algebra for $q > 1$ (see Section~\ref{Estimes}), then one can use a bootstrap argument (e.g. like the one in Section~\ref{sec:bootstrap}) with the fact that $\Omega_q \subset \Omega_{q_0}$ for any $q_0 \in (1,q]$ to get that $U \in \Omega_{q_0}$, for every $q_0>1$. By construction of $f$ given component-wise by \eqref{fourier}, $(x,y,z)$ defined by the Fourier coefficients $(u_n)_{n\ge 0}=(x_n,y_n,z_n)_{n\ge 0}$ is then a strong $C^2$-solution of the stationary Neumann problem of \eqref{eq:mimura_system}. Now, if  $(x,y,z)$ is a strong $C^2$-solution of the stationary Neumann problem of \eqref{eq:mimura_system}, it is in fact, by a bootstrap argument on the PDE, a $C^\infty$-solution. Hence, the Fourier coefficients of $(x,y,z)$ decrease faster than any algebraic decay, which implies that $U \in \Omega_q$, for all $q>1$.
\end{proof}

Hence, based on the result of Lemma~\ref{lemma:equivalence}, we focus our attention on finding $U \in \Omega_q$ such that $f(U)=0$, for a fixed $q >1$. To find the zeros of $f$, the idea is the following. Find an approximate solution $\overline U \in \Omega_q$ of $f=0$, which is done by applying Newton's method on a finite dimensional projection of $f$. Then construct a nonlinear operator $T:\Omega_q \rightarrow \Omega_q$ satisfying two properties. First, it is defined so that the zeros of $f$ are in one-to-one correspondence with the fixed points of $T$, that is $f(U)=0$ if and only if $T(U)=U$. Second, it is constructed as a Newton-like operator around the numerical approximation $\overline U$. The final and most involved step is to look for the existence of a set $B \subset \Omega_q$ centered at $\overline U$ which contains a genuine zero of the nonlinear operator $f$. The idea to perform such task is to find $B \subset \Omega_q$ such that $T:B \rightarrow B$ is a contraction, and to use the contraction mapping theorem to conclude about the existence of a unique fixed point of $T$ within $B$. The method used to find $B$ is based on the notion of the radii polynomials, which provide an efficient means of finding a set on which the contraction mapping can be applied \cite{MR2338393}. We refer to Section~\ref{sec:uniform_contraction} for the definition of the radii polynomials and to Section~\ref{I} for their explicit construction. The polynomials are used to find (if possible) an $r>0$ such that $T$ is a contraction on the closed ball $B(\overline U,r)$ of radius $r$ and centered at $\overline U$ in $\Omega_q$, for $\left\Vert\cdot\right\Vert_q$. With this norm, the closed ball is given by
\begin{equation} \label{eq:def_ball}
B(U,r)= U + \left[-\frac{r}{\rho},\frac{r}{\rho}\right] \times \prod\limits_{n\in\mathbb{N}}\left[-\frac{r}{\omega_n^q},\frac{r}{\omega_n^q}\right]^3.
\end{equation} 

The above scheme to enclose uniquely and locally zeros of $f$ can be extended to find smooth solution paths $\{\tilde U(s)\}_{s \in I}$ such that $f(\tilde U(s))=0$ for all $s \in I$ in some interval $I$. The idea is to construct radii polynomials defined in terms of both $r$ and $s$ and to apply the uniform contraction principle. With this construction, it is possible to prove existence of smooth global solution curves of $f=0$. See Section~\ref{sec:uniform_contraction} for more details.

Before proceeding further, it is worth mentioning that the method proposed in the present work is strongly influenced by the method based on the radii polynomials introduced in \cite{MR2338393} and the rigorous branch following method of \cite{MR2630003}. There are however some differences. First, prior to the present work, the method based on the radii polynomials has never been applied to systems of reaction-diffusion PDEs. Second, new convolution estimates are introduced in Section~\ref{Estimes} for a new range of decay rates, that is for $q \in (1,2)$. The importance of these new estimates is that they can improve the success rate of the proofs while reduce significantly the computational time. Indeed, it is demonstrated in Section~\ref{sec:q} that using $q < 2$ can greatly improve the efficiency of the rigorous method based on the radii polynomials. Also in Section~\ref{sec:q}, a detailed analysis of the optimal choice of the decay rate parameter $q>1$ is made, where the goal is to maximize the chances of performing successfully the computational proofs. Note that prior to the present work, the method based on the radii polynomials has only used decay rates $q \ge 2$. Third, the method here is slightly different from the approach of \cite{MR2630003} for the construction of global branches in the sense that it is a {\em gluing-free} approach. More precisely, it means that no extra work has to be made to glue together adjacent small pieces of smooth curves (see Theorem~\ref{th_rec}). This is due to the fact that here, a global $C^0$ piecewise linear numerical approximation of the curve is constructed, while in \cite{MR2630003}, the global representation of the curve is piecewise linear but not $C^0$.

Let us present a consequence of Theorem~\ref{thm:bif_diagram}. This result could be hard to prove with a purely analytic approach. Its proof is presented in Section~\ref{sec:proofs} and see Figure~\ref{fig10} for a representation.

\begin{corollary}[\bf Co-existence of non constant steady states] \label{corollary:co-existence}
Consider 3-component reaction-diffusion system of PDEs \eqref{eq:mimura_system} with the fixed diffusion parameter $d=0.006$. Then \eqref{eq:mimura_system} has at least eleven distinct co-existing steady states with ten of them being non constant.  
\end{corollary}
\begin{figure} [H]
\begin{center}
\includegraphics[width=8cm]{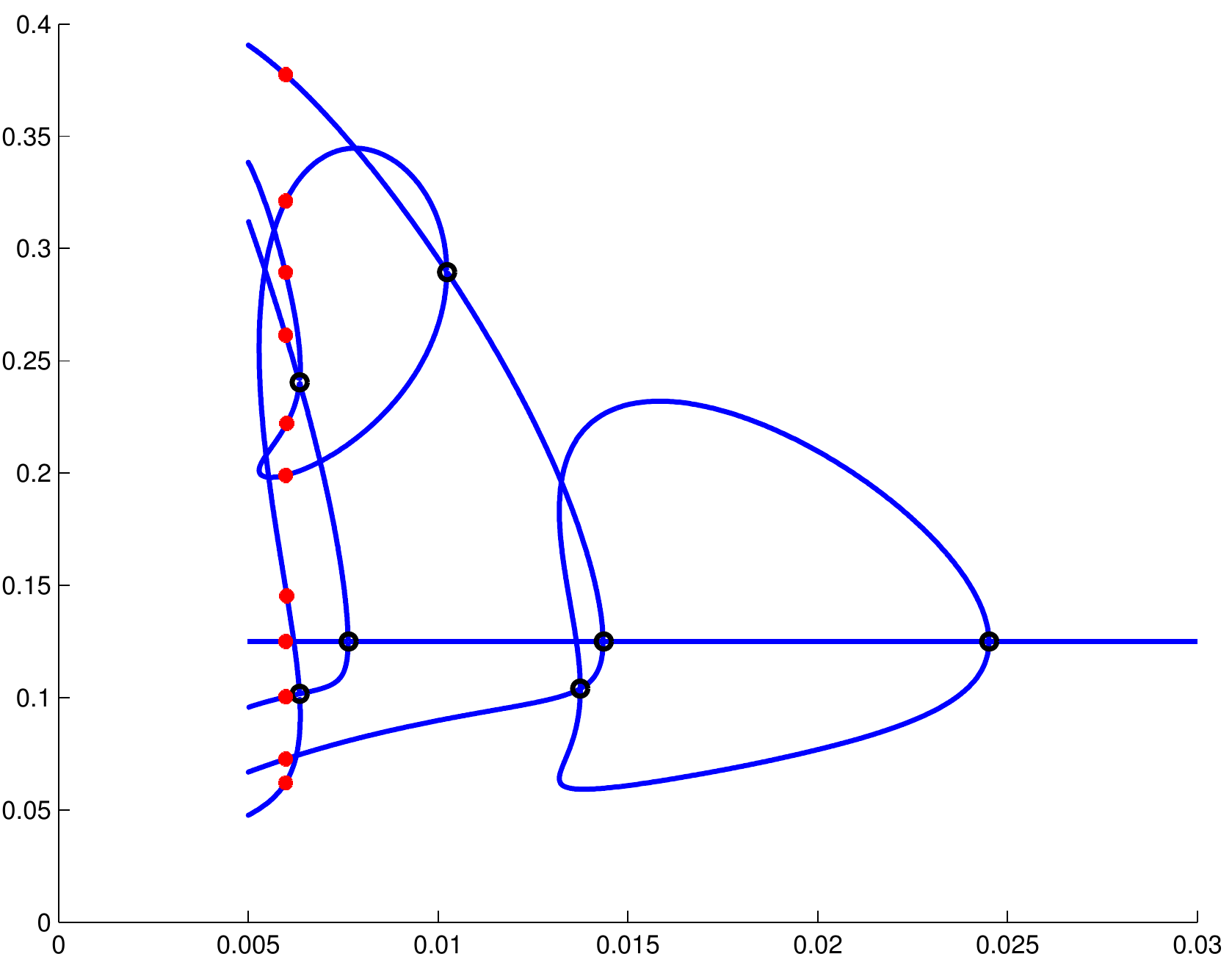}
\end{center}
\vspace{-.5cm}
\caption{\small Geometrical interpretation of Corollary~\ref{corollary:co-existence}. In red, eleven co-existing steady states of \eqref{eq:mimura_system} at the parameter value $d=0.006$. Ten of these solutions are non constant.}
\label{fig10}
\end{figure}
%
%

The following result may be a step toward rigorously studying steady states of the cross-diffusion model \eqref{eq:mimura_cross_diffusion_system}. Its proof is omitted since similar to the proof of Theorem~\ref{thm:bif_diagram}.

\begin{theorem}[\bf Rigorous computations of approximations for a cross-diffusion model] \label{thm:cross_diffusion}
Fix $d=0.02$. For each $\varepsilon \in \{10^{-2}, 10^{-3}, 10^{-4}, 10^{-5} \}$, there exist a non-trivial steady state solution $(x_\eps,y_\eps,z_\eps)$ of the reaction-diffusion system \eqref{eq:mimura_system}. See Figure~\ref{epsilon} for a representation.
\end{theorem}

\begin{figure} [H]
\begin{center}
\subfigure[Diagram of Theorem~\ref{thm:bif_diagram} with point ${\bf 1}$ (in red) corresponding to the starting point of the continuation done in the proof of Theorem~\ref{thm:cross_diffusion} as $\varepsilon \searrow 0$.]{\includegraphics[width=7cm]{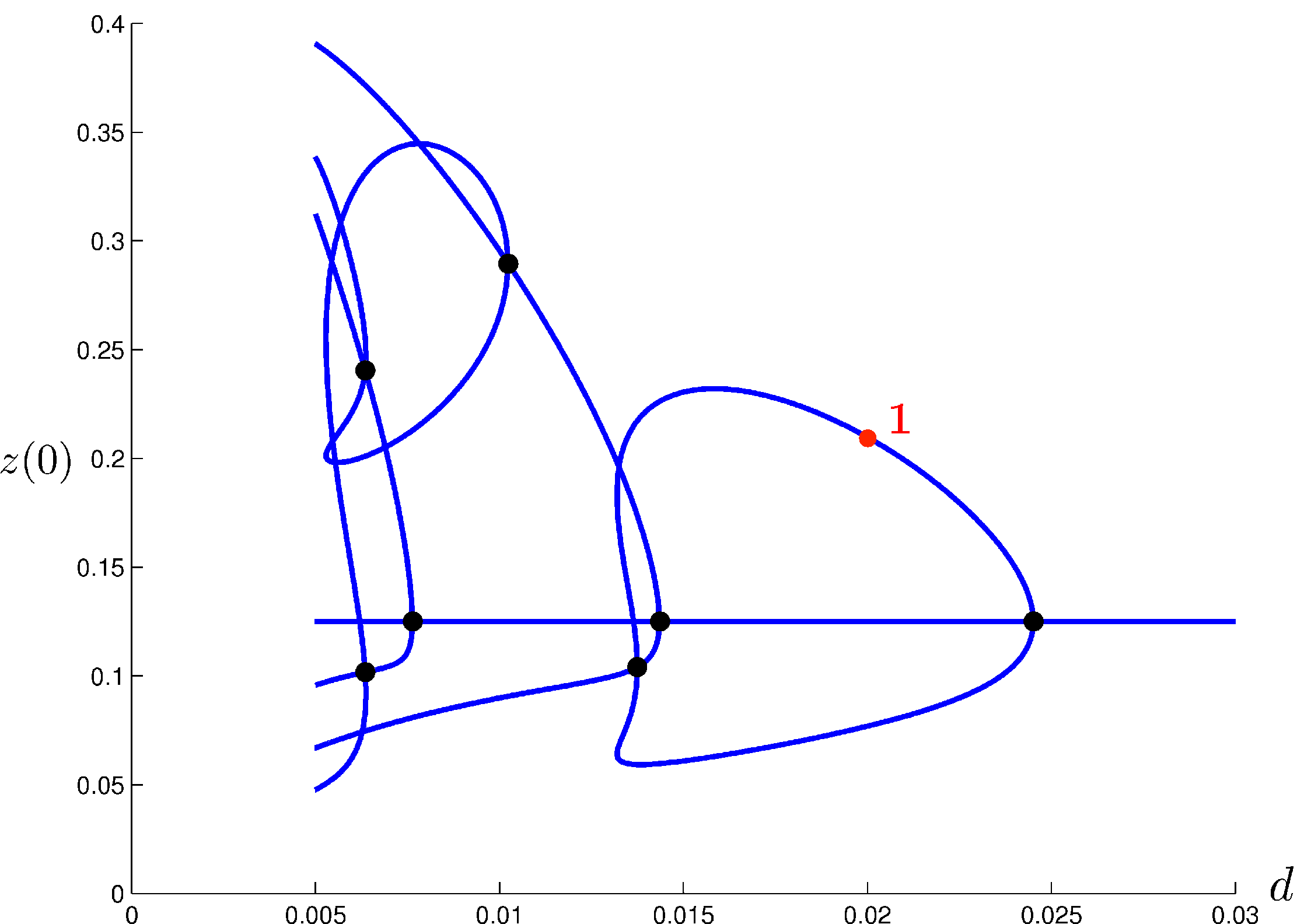}} 
\subfigure[$(u_1,u_2)$ at $\varepsilon=10^{-2}$ corresponding to the red point on the bifurcation diagram on the left. ]{\includegraphics[width=5cm]{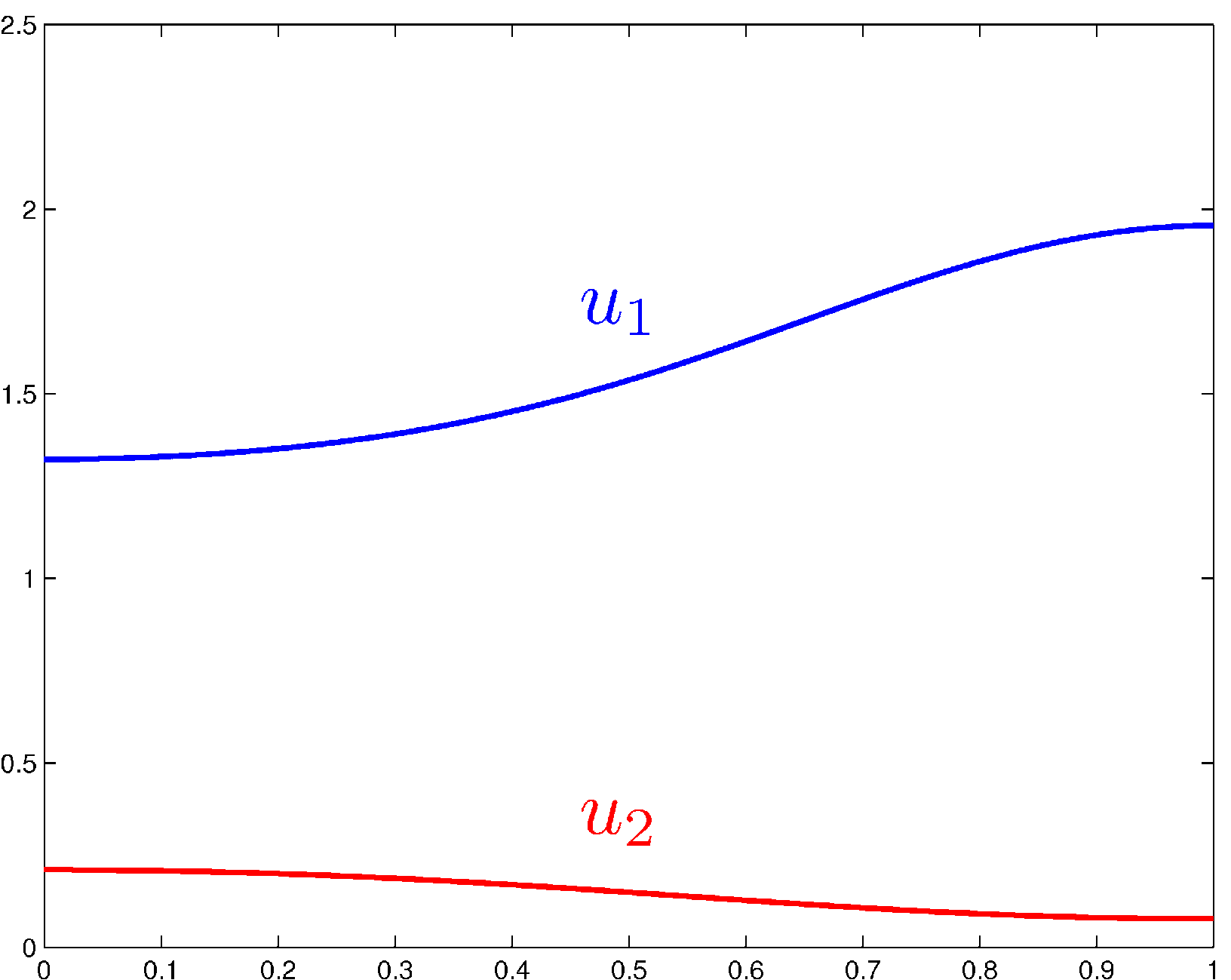}} \\
\subfigure[$(u_1,u_2)$ at $\varepsilon=10^{-3}$]{\includegraphics[width=4cm]{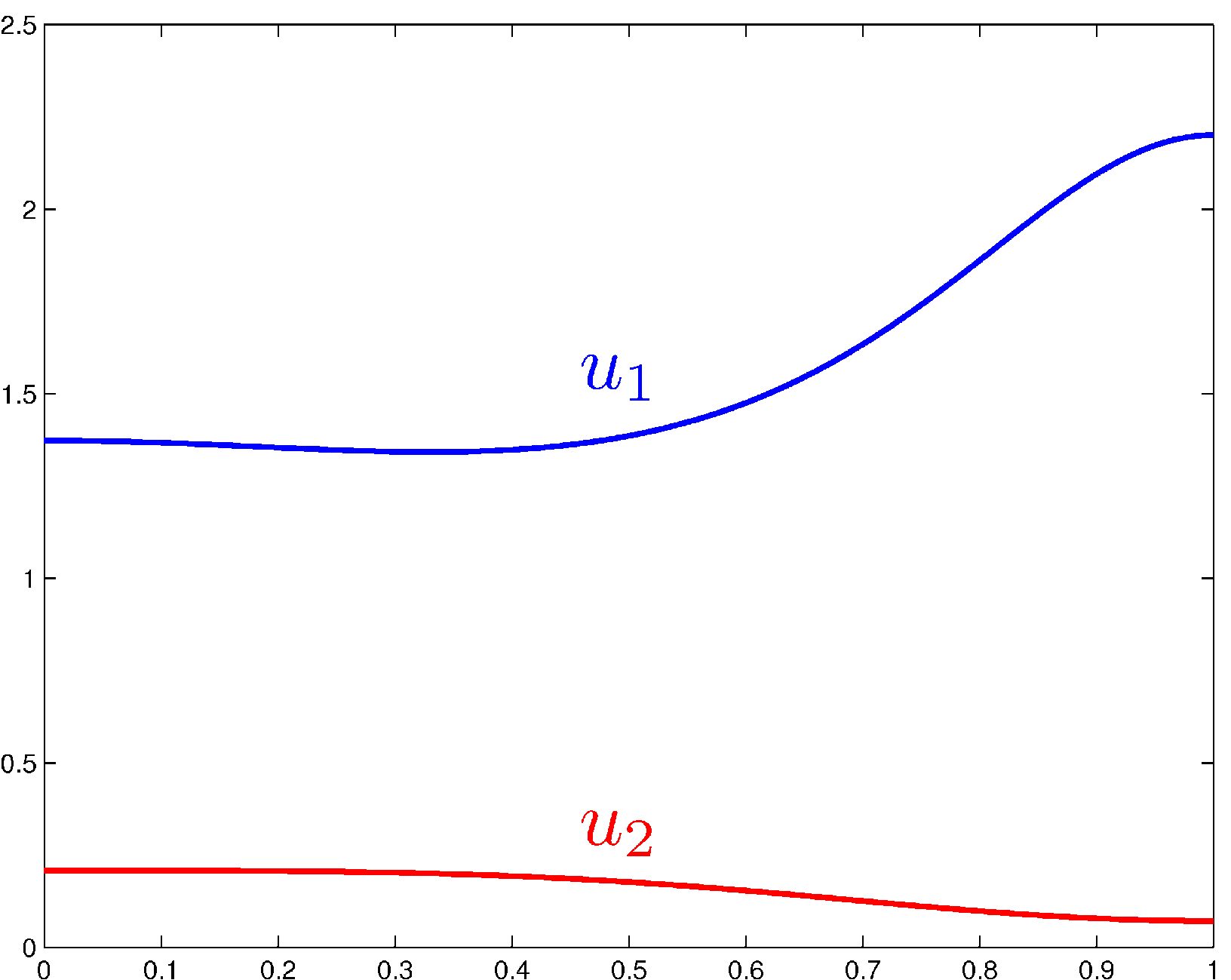}}
\subfigure[$(u_1,u_2)$ at $\varepsilon=10^{-4}$]{\includegraphics[width=4cm]{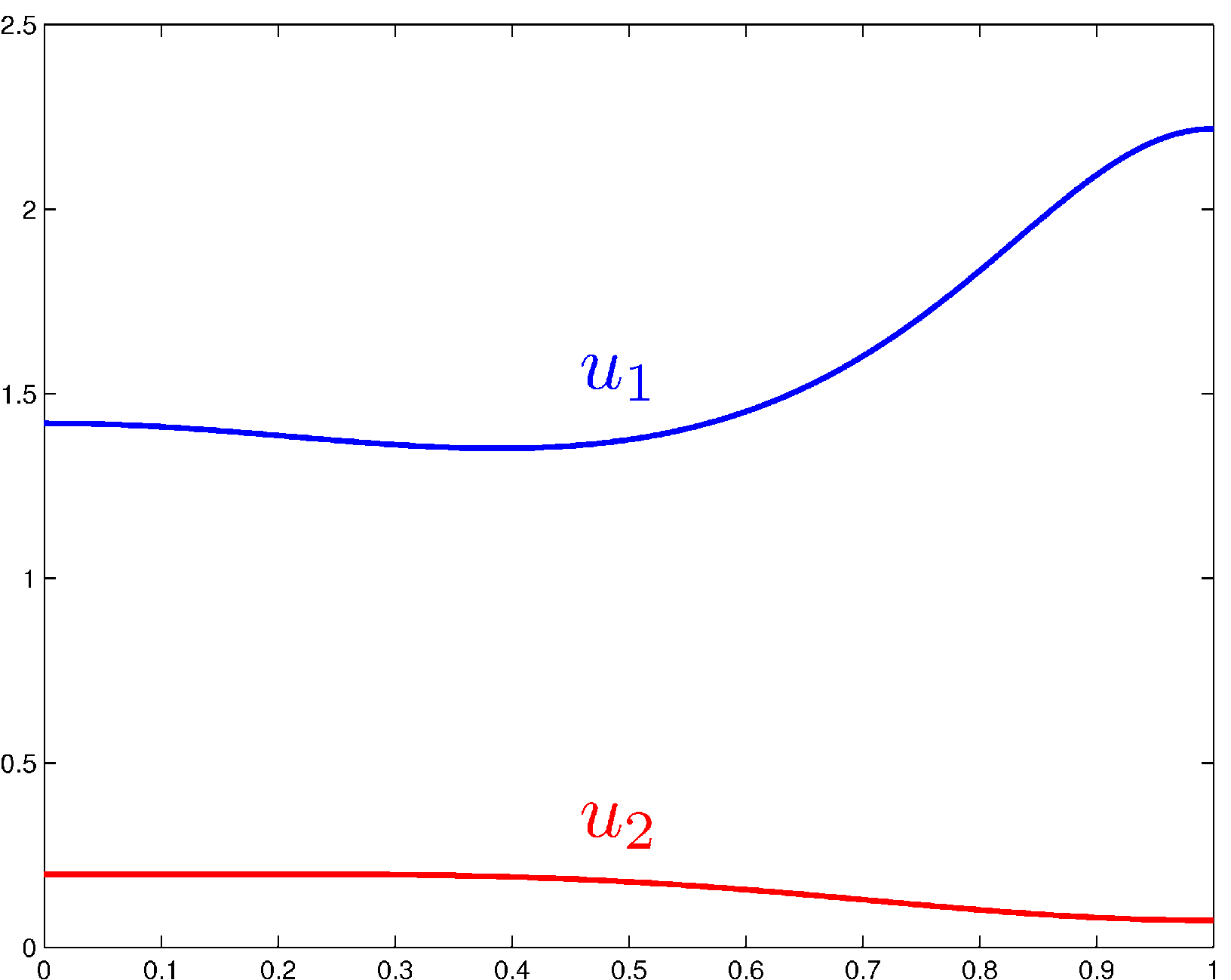}}
\subfigure[\label{10^-5} $(u_1,u_2)$ at $\varepsilon=10^{-5}$]{\includegraphics[width=4cm]{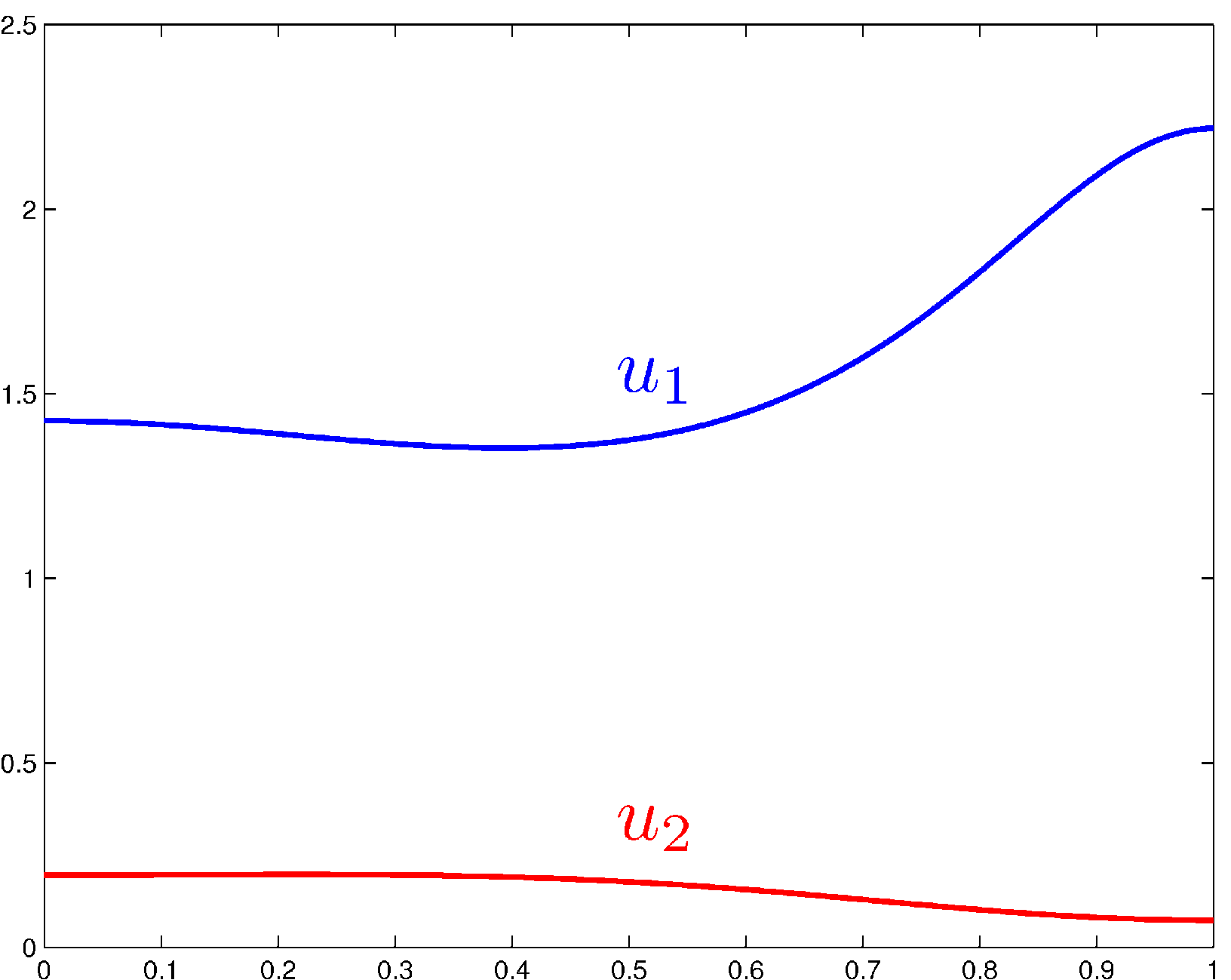}}
\end{center}
\vspace{-.5cm}
\caption{\small Corresponding to the rigorously computed steady states $(x_\eps,y_\eps,z_\eps)$ of \eqref{eq:mimura_system} from Theorem~\ref{thm:cross_diffusion}, we define $u_1(\eps) = x_\eps+y_\eps$ and $u_2(\eps) = z_\eps$. There is apparent convergence as $\varepsilon$ approaches zero.}
\label{epsilon}
\end{figure}

The solutions in Figure~\ref{epsilon} are rigorously computed using the method presented in this paper. In fact the computation uses a simpler version of the method because the proofs of existence are only done at discrete values of the parameters. As can be seen in Table~\ref{table:esp_convergence}, there is apparent convergence as $\varepsilon$ goes to zero. The $\left\Vert \cdot \right\Vert_{\infty}$ bounds there are rigorous because we control through $r=10^{-8}$ (the radius of the ball \eqref{eq:def_ball}) the error we have made with the numerical approximations. According to the work of \cite{MR2437576}, the solution on Figure~\ref{10^-5} given by $(u_1(\eps),u_2(\eps))=(x_\eps+y_\eps,z_\eps)$ when $\eps = 10^{-5}$ should be close to a solution of the cross-diffusion system \eqref{eq:mimura_cross_diffusion_system}. 
\begin{table}[h]
\begin{center}
\begin{tabular}{|c|c|c|}
\hline $\varepsilon_1$ & $\varepsilon_2$ & $\max\left(\left\Vert 
u_1(\eps_1) - u_1(\eps_2)
\right\Vert_{\infty}, \left\Vert 
u_2(\eps_1) - u_2(\eps_2)
\right\Vert_{\infty} \right)$ \\ 
\hline $10^{-2}$ & $10^{-3}$ & 0.2918 \\ 
\hline $10^{-3}$ & $10^{-4}$ & 0.0757 \\ 
\hline $10^{-4}$ & $10^{-5}$ & 0.0092 \\ 
\hline 
\end{tabular}  
\end{center}
\vspace{-.5cm}
\caption{\small Apparent convergence of the rigorously computed steady states of the 3-component reaction-diffusion system \eqref{eq:mimura_system} toward a steady state of the  cross-diffusion system \eqref{eq:mimura_cross_diffusion_system}.}
\label{table:esp_convergence}
\end{table}

The paper is organized as follows. In Section~\ref{method}, the general method is introduced. In Section~\ref{sec:applic}, the method is applied to the problem of computing rigorously steady states of the 3-component reaction-diffusion system \eqref{eq:mimura_system}, where the radii polynomials are explicitly constructed in this context. In Section~\ref{sec:parameter_opt}, a detailed analysis of the optimal choice of the parameters $\Delta_s$, $m$ and $q$ is made, with the goal of maximizing the chances of performing successfully the computational proofs. Finally, in Section~\ref{sec:proofs}, the proofs of Theorem~\ref{thm:bif_diagram} and Corollary~\ref{corollary:co-existence} are presented.

\section{Description of the general method}
\label{method}

In this section, we present the general method, leaving some technical details to Section~\ref{sec:applic}, where all the computations and estimates are presented explicitly to compute several global smooth branches of steady states of \eqref{eq:mimura_system}. The attention is now focused on describing a general method to prove existence and compute global smooth solution curves of $f(U)=0$ in a Banach space $\Omega_q$ of fast decaying Fourier coefficients. The method is based on the radii polynomials, first introduced in \cite{MR2338393}, and is strongly influenced by the rigorous branch following method of \cite{MR2630003}. The idea is to compute a set of numerical approximations $\{\overline U_0, \dots, \overline U_j \}$ of $f=0$ by considering a finite dimensional projection, to use the approximations to construct a global continuous curve of piecewise linear interpolations between the $\overline U_i$'s (see Figure~\ref{fig:piecewise_linear_approx}) and to apply the uniform contraction principle on {\em tubes} centered at each segment to conclude about the existence of a unique smooth solution curve of $f=0$ nearby the piecewise linear curve of approximations. The approximate curve is computed using pseudo-arclength continuation (e.g. \cite{MR910499}).

\subsection{Construction of a piecewise linear curve of approximations} \label{sec:pred-cor}

To construct a piecewise linear curve of approximations of $f=0$, we consider a finite dimensional projection $f^{[m]}$ of $f$ whose dimension depends on $m$ (see Section~\ref{sec:finite_dim_projection}). In what follows, $(\cdot)^{[m]}$ denotes considering this finite dimensional projection. Reversely, when we have some finite dimensional vector $U^{[m]}$, $U$ denotes the infinite vector obtained by completing $U^{[m]}$ with zeros. 

Suppose we have an approximate zero $\overline u_0^{[m]}$ of $f^{[m]}(d_0,\cdot)$ at $d_0$. Then, given $U_0^{[m]} \bydef  \left(d_0,u_0^{[m]}\right)$, we compute an approximate tangent vector $\dot U_0^{[m]}$, that is 
$Df^{[m]}\left(\overline U_0^{[m]}\right)\left(\dot U_0^{[m]}\right) \approx 0.$
Consider 
\begin{equation} \label{eq:predictor}
\hat U_0^{[m]}  \bydef  \overline U_0^{[m]} + \Delta_s \dot U_0^{[m]},
\end{equation}
 a predictor, where $\Delta_s$ is a parameter (whose value, representing roughly the arc length of curve we are covering in one step, is discussed in Section~\ref{opt}). Consider also the plane $\Pi$ whose equation is given by $E(U) \bydef \left(U-\hat U_0\right)\cdot \dot U_0=0$.  The pseudo-arclength operator is
\[
F:U\mapsto \begin{pmatrix}  E(U)\\f(U)  \end{pmatrix} ,
\]
and using Newton's method with initial point $\hat U_0^{[m]} $, we compute $U_1^{[m]}$ such that $F^{[m]}(U_1^{[m]}) \approx 0$. We refer to Figure~\ref{fig:pred_corr} for a geometrical representation of a pseudo-arclength continuation step. A next predictor-corrector step can be performed starting at $U_1^{[m]}$, and so on.

\begin{figure}[h]
\begin{center}
\includegraphics[width=7cm]{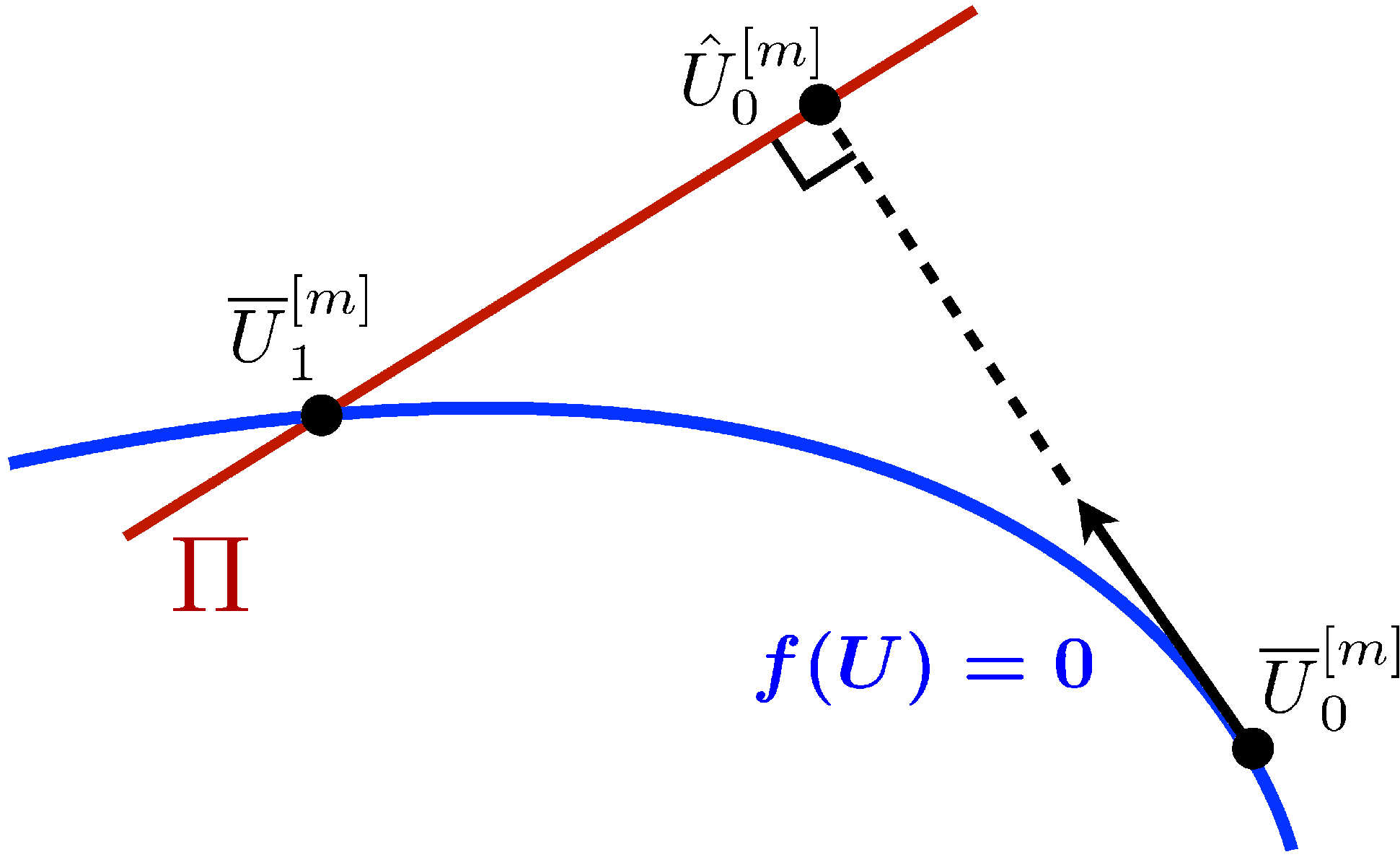}
\end{center}
\vspace{-.6cm}
\caption{\small A predictor-corrector step with pseudo-arclength continuation.}
\label{fig:pred_corr}
\end{figure}

Applying the predictor-corrector step $j$ times, we compute a set $\{\overline U_0,\dots, \overline U_j\}$ of approximations that defines a piecewise linear approximation curve (see Figure~\ref{fig:piecewise_linear_approx}). The next step is to show existence of a unique smooth solution curve $\cC$ of $f=0$ nearby the piecewise linear curve of approximations, as portrayed in Figure~\ref{fig:piecewise_linear_approx}. This task is twofold. First, one shows the existence of a unique portion of solution curve $\mathcal{C}^{(i)}$ in a small tube centered at the segment $[\overline  U_i,\overline  U_{i+1}]$. This is done in Theorem~\ref{th_local} by showing that a Newton-like operator $\tilde T$ is a uniform contraction on the tube. To verify the hypothesis of the uniform contraction principle, Theorem~\ref{th_pol} is introduced. This requires the construction of some bounds, which are presented in Section~\ref{sec:def_bornes}. In practice, verifying the hypothesis of Theorem~\ref{th_pol} is done via Lemma~\ref{nb_fini} by using the radii polynomials which are presented in Section~\ref{sec:def_pol}. From Lemma~\ref{nb_fini}, one sees that the strength of the radii polynomials is that they provide an efficient means (in the form of a finite number of polynomial inequalities to be checked rigorously on a computer using interval arithmetic) of finding a set on which $\tilde T$ is a uniform contraction. Second, one shows that each $\mathcal{C}^{(i)}$ is smooth, and that  
\[
\mathcal{C} \bydef \bigcup_{i=0}^{j-1} \mathcal{C}^{(i)}
\]
is a global smooth solution curve of $f=0$. In Section~\ref{sec:recollement}, we show how the smoothness of $\mathcal{C}^{(i)}$ can be proved by verifying the hypothesis of Theorem~\ref{th_reg}. Afterward, we show that if the hypotheses of Theorem~\ref{th_reg} and Theorem~\ref{th_rec} are satisfied, then $\mathcal{C}^{(i)}$ and $\mathcal{C}^{(i+1)}$ connect smoothly. The smoothness of the global solution curve $\mathcal{C}$ follows by construction.

\begin{figure}[h]
\begin{center}
\includegraphics[width=12cm]{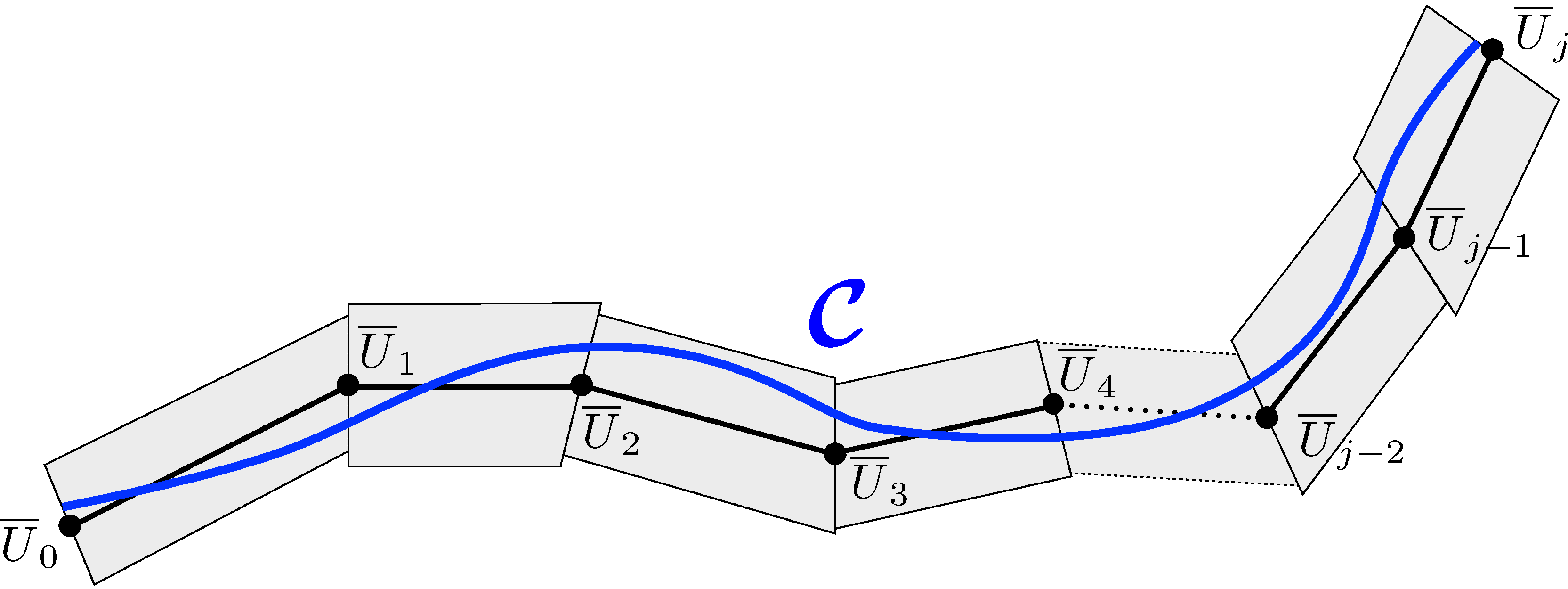}
\end{center}
\vspace{-.5cm}
\caption{\small Piecewise linear curve approximation (in black) constructed using pseudo-arclength continuation and existence of a global smooth solution curve $\mathcal C$ of $f=0$ (in blue) nearby the approximations }
\label{fig:piecewise_linear_approx}
\end{figure}

\subsection{Newton-like operator, uniform contraction and radii polynomials} \label{sec:uniform_contraction}

Let us define what is required to prove existence of some portion of smooth curve $\mathcal{C}^{(i)}$. Without loss of generality, let us introduce the idea to prove the existence of $\mathcal{C}^{(0)}$ that is the piece of curve close to the segment $[\overline U_0,\overline U_1]$ with two approximate tangent vectors $\dot U_0$ and $\dot U_1$ at those points given by the pseudo-arclength continuation algorithm. For any $s$ in $[0,1]$, we set
\[
\overline U_s  \bydef  (1-s) \overline U_0 + s \overline U_1 =\overline U_0 + s\Delta \overline U, \mbox{ where } \Delta \overline U  \bydef  \overline U_1 -\overline U_0
\]
and
\[
\dot U_s  \bydef  (1-s) \dot U_0 + s \dot U_1 =\dot U_0 + s\Delta\dot U, \mbox{ where } \Delta \dot U  \bydef  \dot U_1 -\dot U_0.
\]
Then we define, still for $s$ in $[0,1]$, the hyperplane $\Pi_s$ whose equation is given by
\begin{equation}
\label{Es}
E_s\left(U\right) \bydef \left(U-\overline U_s\right)\cdot \dot U_s,
\end{equation}
the function
\begin{equation}
\label{F_s}
F_s(U) \bydef 
\begin{pmatrix}
E_s(U)\\
f(U)
\end{pmatrix}
\end{equation}
and the Newton-like operator
\begin{equation}
\label{T}
T_s(U) \bydef U-JF_s(U),
\end{equation} 
where $J$ is an injective linear operator approximating the inverse of $DF_0\left(\overline U_0\right)$ (see Section~\ref{def_J} for an example of how to construct $J$ and check that it is injective). We now use the uniform contraction principle on $T_s$ to conclude about the existence of a curve of fixed points that corresponds, by injectivity of $J$, to a solution curve of $f(U)=0$. 

\begin{theorem}
\label{th_local}
If there exists some $r>0$ such that 
\[
\tilde T :
\left\{
\begin{aligned}
&\left[0,1\right] \times B(0,r) \longrightarrow B(0,r) \\
&(s,V) \longmapsto T_s\left(V+\overline U_s\right)-\overline U_s
\end{aligned}
\right.
\]
is a uniform contraction, then for every $s \in [0,1]$, there exists a unique $\tilde U(s) \in B\left(\overline U_s,r\right)$ such that $F_s\left(\tilde U(s)\right)=0$. Moreover, the function $s\mapsto \tilde U(s)$ is of class $C^{k}$ if $(s,V)\mapsto \tilde T(s,V)$ is of class $C^{k}$.
\end{theorem}

\begin{proof}
This is a direct application of the uniform contraction principle (e.g. see \cite{MR660633}). 
\end{proof}

It seems legitimate to expect $T_s$ to be a contraction on a small set containing the segment $[\overline  U_0,\overline  U_{1}]$ parameterized by $\overline U_s$ ($s \in [0,1]$) since $T_s$ is an approximate Newton operator at $\overline  U_0$. 

\subsubsection{Definition of some bounds} \label{sec:def_bornes}

To prove that $\tilde T$ is a uniform contraction on $B(0,r)$, we prove the existence of bounds $Y_d(s)$, $Y_{n}(s)$, $Z_d(r,s)$ and $Z_{n}(r,s)$ such that for every $n \in \mathbb{N}$ and $s \in[0,1]$,
\begin{equation}
\label{ineq3}
\left\vert \left[T_s\left(\overline U_s\right)-\overline U_s\right]_{d} \right\vert \leq Y_{d}(s) \mbox{  and  }  \left\vert \left[T_s\left(\overline U_s\right)-\overline U_s\right]_{n} \right\vert \leq Y_{n}(s),
\end{equation}
\begin{equation}
\label{ineq1}
 \sup \limits_{V,V'\in B(0,r)} \left\vert \left[DT_s\left(\overline U_s + V\right)\left(V'\right) \right]_{d}\right\vert \leq Z_{d}(r,s) \mbox{  and  }
 \sup\limits_{V,V'\in B(0,r)} \left\vert \left[DT_s\left(\overline U_s + V\right)\left(V'\right) \right]_{n}\right\vert \leq Z_{n}(r,s).
\end{equation}

The subscript $(\cdot)_d$ corresponds to the first entry $d$ of $U$ and the first entry $E_s$ of $F_s$. We set $Y=(Y_d,Y_0,\dots , Y_n, \dots)$, where $Y_n\in \mathbb{R}^3$ (same for $Z$). Absolute values and inequalities applied to vectors are considered component-wise. For the sake of simplicity of the presentation, we omit to write explicitly the dependence of those terms in $r$ and $s$ when we are not focusing on them. Let us now give some sufficient conditions on those bounds for $\tilde T$ to be a uniform contraction.
\begin{theorem}
\label{th_pol}
If $Y,Z$ verify (\ref{ineq3}) and (\ref{ineq1}) resp, and if there exists $r>0$ such that for all $s \in[0,1]$
\begin{equation}
\label{ineq4}
\left\Vert Y(s) + Z(r,s) \right\Vert_q < r,
\end{equation}
then $\tilde T$ is a uniform contraction on $ B(0,r)$, with contraction constant $\displaystyle{\kappa \bydef \max\limits_{s\in[0,1]}\frac{\left\Vert Z (r,s) \right\Vert_q }{r}<1}$.
\end{theorem}

\begin{proof}
For all $s\in[0,1]$, $ j \in \cup_{n \ge 0} \{ n_x,n_y,n_z \}$, $V,V'\in B\left(\overline U_s,r\right)$, the mean value theorem yields the existence of $W=\lambda V + (1-\lambda)V'$ for some $\lambda = \lambda(j) \in [0,1]$, such that
\[
\left[T_s(V)-T_s(V')\right]_j=\left[DT_s(W)(V-V')\right]_j.
\]
%
Then
\[
\left\vert \left[T_s(V)-T_s(V')\right]_{j}  \right\vert \leq \left\vert \left[ DT_s(W)\left(\frac{r(V-V')}{\left\Vert V-V'\right\Vert_q}\right)\right]_{j} \right\vert\frac{1}{r}\left\Vert V-V'\right\Vert_q \leq \frac{Z_{j}(r,s)}{r}\left\Vert V-V'\right\Vert_q.
\]
Similarly,
\[
\left\vert \left[T_s(V)-T_s(V')\right]_d \right\vert \leq \frac{Z_d(r,s)}{r}\left\Vert V-V'\right\Vert_q.
\]
Therefore, for all $s\in [0,1]$ and $V,V' \in B(0,r)$,
\[
\left\Vert \tilde T(s,V)-\tilde T(s,V')\right\Vert_q =  \left\Vert T_s\left(V+\overline U_s\right)-T_s\left(V'+ \overline U_s\right)\right\Vert_q \leq \frac{\left\Vert Z(r,s)\right\Vert_q}{r}\left\Vert V-V'\right\Vert_q  = \kappa \left\Vert V-V'\right\Vert_q
\]
and for all $s\in [0,1]$ and $V \in B(0,r)$,
\[
\left\Vert  \tilde T(s,V) \right\Vert_q = \left\Vert T_s(V) - \overline U_s \right\Vert_q  \leq  \left\Vert T_s(V) -  T_s\left(\overline U_s\right) \right\Vert_q + \left\Vert T_s\left(\overline U_s\right) - \overline U_s \right\Vert_q \leq \left\Vert Z(r,s)+Y(s) \right\Vert_q <r .
\]
\end{proof}

Suppose the bounds $Y$ and $Z$ verifying conditions (\ref{ineq3}) and (\ref{ineq1}) are computed. To be able to use Theorem~\ref{th_pol}, we need to check that those bounds also verify inequality (\ref{ineq4}). Note that for every $n \in \mathbb{N}$, $Y_{n}$ is a function of $s$ and $Z_{n}$ is a function of both $r$ and $s$. Besides, they can be constructed as polynomials in $r$ and $s$, and for $n$ greater than some $M$, we can choose
\[
Y_n=0 \quad \mbox{and} \quad Z_n=\hat Z_M\frac{\omega_M^q}{\omega_n^q},
\]
where $\hat Z_M$ is also a polynomial in $r$ and $s$. Those assertions are not explained in details here. We refer to Section~\ref{sec:Yn_n_ge_m} and Section~\ref{sec:Zn_n_ge_M} for explicit details. For the moment, let us only say that $Y_n$ can be taken to be $0$ for n large enough because $\overline U_s$ has only a finite number $m$ of non zero coefficients, and hence $\left[T_s\left(\overline U_s\right)-\overline U_s\right]_{n}=0$ for $n$ large enough. Let us now introduce the radii polynomials which allow us to verify inequality (\ref{ineq4}) using rigorous numerics.

\subsubsection{Radii polynomials} \label{sec:def_pol}
 
Let $M$ be a computational parameter. We refer to Section~\ref{Y} to determine how to choose its value. Define 
\[
P_d(r,s) \bydef Y_d(s) + Z_d(r,s) - \frac{r}{\rho},
\] 
for $0 \le n<M$,
\[
P_{n}(r,s) \bydef Y_{n}(s) + Z_{n}(r,s) - \frac{r}{\omega_n^q},
\]
and for $n=M$,
\[
P_{M}(r,s) \bydef \hat Z_{M}(r,s)  - \frac{r}{\omega_M^q}.
\]
Note that the term $\displaystyle{- \frac{r}{\omega_n^q}}$ has to be understood component-wise.

\begin{lem}
\label{nb_fini}
Suppose that there exists $r>0$ such that for all $s \in [0,1]$,
\[
P_d(r,s)<0 \quad \mbox{and} \quad P_n(r,s)<0 \quad {\rm for~ all~} 0 \le n \leq M.
\]
Then inequality~\eqref{ineq4} is satisfied and Theorem~\ref{th_pol} holds.
\end{lem}

\begin{proof}
By definition of the radii polynomials, we have that
\[
Y_d(s)+Z_d(r,s)<\frac{r}{\rho}\qquad \mbox{and} \qquad Y_{n}(s)+Z_n(r,s)<\frac{r}{\omega_n^q}
\]
for all $0 \le n \leq M$. Since $\displaystyle{P_{M}(r,s)=\hat Z_{M}(r,s)  - \frac{r}{\omega_M^q}<0}$, we also have that 
\[
Y_{n}(s)+Z_n(r,s)=\hat Z_M(r,s)\frac{\omega_M^q}{\omega_n^q}<\frac{r}{\omega_n^q},
\]
for all $n > M$. Therefore inequality~\eqref{ineq4} is satisfied. 
\end{proof}

We show in Section~\ref{opt} how to carefully chose $\Delta_s$ and $m$ to maximize the chance of finding an $r>0$ satisfying the hypotheses of Lemma~\ref{nb_fini}. Lemma~\ref{nb_fini} is useful for checking efficiently the hypotheses of Theorem~\ref{th_pol}, and therefore Theorem~\ref{th_local} by verifying a finite number of inequalities.

\subsection{Constructing a global smooth solution curve} \label{sec:recollement}

Suppose now that we found $r>0$ satisfying the radii polynomials inequalities of Lemma~\ref{nb_fini}. Hence, by Theorem~\ref{th_local}, there exists a smooth function $s\in[0,1]\mapsto\tilde U(s)$ whose image $\cC^{(0)}$ is the only solution curve to $f=0$ within in a small tube of radius $r$ centered at the segment $[\overline U_0,\overline U_1]$. More precisely, this tube is given by $\left\{ U\ \vert \ \exists s \in [0,1],\ U \in B\left(\overline U_s,r\right) \cap \Pi_s \right\}$. The following result is similar to Lemma~10 in \cite{MR2630003}.

\begin{theorem} \label{th_reg}
Suppose that
\begin{equation} \label{eq:th_reg}
- \Delta \overline U \cdot \dot U_0 +r \left(W_1^q\right)^{[m]} \cdot \left\vert \Delta \dot U \right\vert + \left\vert \Delta \overline U \cdot  \Delta \dot U \right\vert < 0,
\end{equation}
where
\[
\left(W_1^q\right)^{[m]}=\left(\frac{1}{\rho},\underbrace{\frac{1}{\omega_0^q},\dots,\frac{1}{\omega_0^q}}_{3  \ \ times},\dots ,\underbrace{\frac{1}{\omega_{m-1}^q},\dots,\frac{1}{\omega_{m-1}^q}}_{3 \ times}\right)^T,
\]
then $\cC^{(0)}$ is a smooth curve, that is, for all $s \in [0,1]$, $\displaystyle{\frac{d \tilde U}{ds}(s) \neq 0}$.
\end{theorem}

\begin{proof}

Considering the equality $E_s(\tilde U(s))=0$ and taking its derivative with respect to $s$,
\[
\frac{d\tilde U}{ds}(s)\cdot \dot U_s=- \frac{\partial E_s}{\partial s}(\tilde U(s)).
\]
Recalling the definition of $E_s$ from (\ref{Es}),
\begin{equation}
\label{dEds}
\frac{\partial E_s}{\partial s}(\tilde U(s))=- \Delta \overline U \cdot \dot U_0 + \left(\tilde U(s)- \overline U_s \right) \cdot \Delta \dot U -s \Delta \overline U \cdot  \Delta \dot U.
\end{equation}
Hence, for all $s \in [0,1]$,
\[
 \frac{\partial E_s}{\partial s}(\tilde U(s)) \leq - \Delta \overline U \cdot \dot U_0 +r \left(W_1^q\right)^{[m]} \cdot \left\vert \Delta \dot U \right\vert + s\left\vert \Delta \overline U \cdot  \Delta \dot U \right\vert,
\]
since $\left\Vert \tilde U(s)- \overline U_s \right\Vert_q < r $. Finally, using inequality \eqref{eq:th_reg}, for all $s \in [0,1]$, $\displaystyle{\frac{\partial E_s}{\partial s}(\tilde U(s))\neq 0}$. This proves that for all $s \in [0,1]$, $\displaystyle{\frac{d \tilde U}{ds}(s) \neq 0}$, yielding the smoothness of $\cC^{(0)}$. 
\end{proof}

In practice, the hypothesis of Theorem~\ref{th_reg} are checked rigorously using interval arithmetic. Remark that this hypothesis is very reasonable if $r$ is small enough. See Figure~\ref{regu} for a geometric representation of the important quantities involved in the hypothesis \eqref{eq:th_reg}. One can see there that if the length of the segment $[\overline U_0,\overline U_1]$ is small, that is if the vector $\Delta \overline U$ is small, then the vectors $\Delta \overline U$ and $\Delta \dot U$ should be close to be perpendicular and the vectors $\Delta \overline U$ and $\dot U_0$ should be close to be parallel. Hence, $\Delta \overline U \cdot \dot U_0$ should be close to the value $\| \Delta \overline U \| \| \dot U_0\|$ while the vector $\Delta \overline U \cdot \Delta \dot U$ should be close to $0$. Hence, for very small value of $r$, the value of $r \left(W_1^q\right)^{[m]} \cdot \left\vert \Delta \dot U \right\vert $ should be small, and then the chances of satisfying inequality \eqref{eq:th_reg} should be high.

\begin{figure} [H]
\begin{center}
\includegraphics[width=8cm]{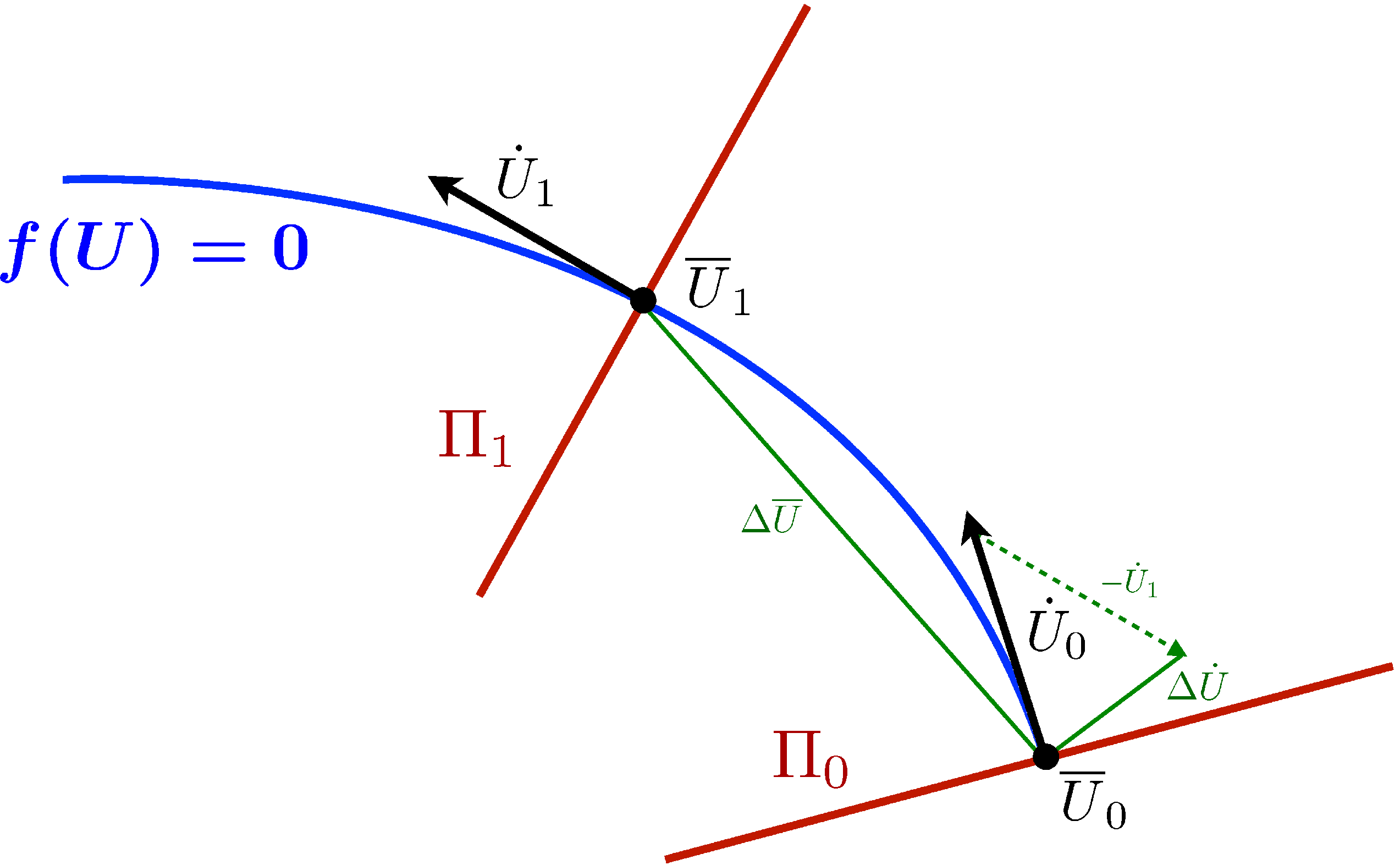}
\end{center}
\vspace{-.5cm}
\caption{\small The important quantities involved in the hypothesis \eqref{eq:th_reg} of Theorem~\ref{th_reg}.}
\label{regu}
\end{figure}

Assuming that two consecutive smooth curves have been computed, we can prove that they connect smoothly in one curve. Using the notation $(\cdot)^{(0)}$ (resp. $(\cdot)^{(1)}$) to refer to the first (resp. the second) portion of curve. The following result is similar to Proposition~8 in \cite{MR2630003} but two aspects are different. First the gluing of $\cC^{(0)}$ and $\cC^{(1)}$ comes for free from the $C^0$ representation of the union of the segments $[\overline U_0,\overline U_1]$ and $[\overline U_1,\overline U_2]$. Second, the proof that $\cC^{(0)} \cup \cC^{(1)} $ is smooth at the intersection $\cC^{(0)} \cap \cC^{(1)} $ is more detailed and calls upon the implicit function theorem.

\begin{theorem}
\label{th_rec}
Assume that the hypotheses \eqref{ineq4} of Theorem~\ref{th_pol} and \eqref{eq:th_reg} of Theorem~\ref{th_reg} are satisfied between $\overline U_0$ and $\overline U_1$, and also between $\overline U_1$ and $\overline U_2$. Then $\cC^{(0)} \cup \cC^{(1)} $ is a smooth curve.
\end{theorem}

\begin{proof}
Since Theorem~\ref{th_pol} is satisfied over the segment $[\overline U_0,\overline U_1]$ (resp. $[\overline U_1,\overline U_2]$), consider the radius $r^{(0)}$ (resp. $r^{(1)}$) satisfying \eqref{ineq4}. Without loss of generality, one can assume that $r^{(0)} \neq r^{(1)}$ by continuity of the radii polynomials. Theorem~\ref{th_pol} allows us to use Theorem~\ref{th_local} to get the existence of two smooth functions $\tilde U^{(0)}$ and $\tilde U^{(1)}$ whose images are $\cC^{(0)}$ and $\cC^{(1)} $.
First we prove that the two curves connect, that is $\tilde U^{(0)}(1)=\tilde U^{(1)}(0)$. Recalling \eqref{F_s}, we have that $F^{(0)}_1=F^{(1)}_0$, in particular $\overline U^{(0)}_1=\overline U_1=\overline U^{(1)}_0$ and $\Pi^{(0)}_1=\Pi^{(1)}_0$. Moreover, $\tilde U^{(0)}(1)$ is the only solution of $F^{(0)}_1=0$ in $B\left(\overline U_1,r^{(0)}\right)$ and $\tilde U^{(1)}(0)$ is the only solution of $F^{(1)}_0=0$ in $B\left(\overline U_1,r^{(1)}\right)$. Hence $\tilde U^{(0)}(1)$ and $\tilde U^{(1)}(0)$ are both solutions to $F^{(1)}_0=0$ and since one of the two balls must be included into the other (they have same center) the two solutions are equal. 

From now we assume without loss of generality that $r^{(0)}<r^{(1)}$ and we show that the connection between the two curves is smooth. For all  $s$ in $[0,1]$, the radii polynomials are negative at $r^{(0)}>0$. The fact that the radii polynomials are continuous in $s$ yields the existence of $\varepsilon_0>0$ such that those polynomials are still negative for $s$ in $[0,1+\varepsilon_0]$, and then that $\tilde T^{(0)}$ is still a uniform contraction for $[0,1+\varepsilon_0]$. As a result, $\tilde U^{(0)}$ can be extended into a smooth function defined on $[0,1+\varepsilon_0]$  such that  $f\left(\tilde U^{(0)}(s)\right)=0$, for all $s$ in $[0,1+\varepsilon_0]$. Hence, $\tilde U^{(0)}\left([1,1+\varepsilon_0]\right)\subset \cC^{(1)}$. We want to find $s \in [0,1]$  such that $\tilde U^{(0)}(1+\varepsilon) \in \Pi^{(1)}_s$ for all $\varepsilon \in [0,\varepsilon_0]$. To do that, set
\[
\varphi : (\varepsilon,s)\longmapsto \left(\tilde U^{(0)}(1+\varepsilon) - \overline U^{(1)}_s \right) \cdot \dot U^{(1)}_s.
\]
We have 
$\varphi(0,0)=\left(\tilde U^{(0)}(1) - \overline U_1 \right) \cdot \dot U_1=0$
and $\varphi(0,s)=E^{(1)}_s\left(\tilde U^{(0)}(1)\right)$, so
\[
\frac{\partial \varphi}{\partial s}(0,0)=\left.\frac{\partial E^{(1)}_s}{\partial s}\right\vert_{s=0}\left(\tilde U^{(0)}(1)\right).
\]
Since the hypothesis of Theorem~\ref{th_reg} is verified, we get according to (\ref{dEds}) that $\displaystyle{\frac{\partial \varphi}{\partial s}(0,0)<0}$. Hence the implicit function theorem holds and there exist $\varepsilon_1 \in (0, \varepsilon_0]$ and a smooth function $s: \varepsilon \in [0,\varepsilon_1] \mapsto s(\varepsilon)$, such that $s(0)=0$ and $\varphi(\varepsilon,s(\varepsilon))=0$.
Also,
\[
\frac{\partial \varphi}{\partial \varepsilon}(0,0)+\frac{\partial \varphi}{\partial s}(0,0)s'(0)=0,
\]
and since
\[
\frac{\partial \varphi}{\partial \varepsilon}(0,0) = \frac{d\tilde U}{ds}^{(0)}(1)\cdot \dot U_1 >0 \mbox{ and } \frac{\partial \varphi}{\partial s}(0,0)<0 \mbox{ according to (\ref{dEds}) and the hypotheses},
\]
we have that $s'(0)>0$.
Hence, there exists $\varepsilon_2 \in (0, \varepsilon_1]$ such that for all $\varepsilon \in [0,\varepsilon_2]$, $s(\varepsilon)\in[0,1]$ and $ F^{(1)}_{s(\varepsilon)}\left(\tilde U^{(0)}(1+\varepsilon)\right)=0$. Given that for all $\varepsilon \in [0,\varepsilon_2]$, $F^{(1)}_{s(\varepsilon)}=0$ has a unique solution in $B\left(\overline U^{(1)}_{s(\varepsilon)},r^{(1)}\right)$, showing that $\left\Vert \tilde U^{(0)}(1+\varepsilon) - \overline U^{(1)}_{s(\varepsilon)} \right\Vert_q < r^{(1)}$ will conclude the proof. Now,
\begin{eqnarray*}
\left\Vert \tilde U^{(0)}(1+\varepsilon) - \overline U^{(1)}_{s(\varepsilon)} \right\Vert_q &\leq& \left\Vert \tilde U^{(0)}(1+\varepsilon) - \overline U^{(0)}_{1+\varepsilon} \right\Vert_q + \left\Vert \overline U^{(0)}_{1+\varepsilon} - \overline U^{(0)}_{1} \right\Vert_q + \left\Vert \overline U^{(1)}_{0} - \overline U^{(1)}_{s(\varepsilon)} \right\Vert_q \\
&\leq& r^{(0)} + \varepsilon \left\Vert \Delta \overline U^{(0)} \right\Vert_q + s(\varepsilon) \left\Vert \Delta \overline U^{(1)} \right\Vert_q.
\end{eqnarray*}
Using that $s(0)=0$, that $s(\varepsilon)$ is continuous and that $r^{(0)}<r^{(1)}$, there exists $\varepsilon_3 \in [0, \varepsilon_2]$, such that for all $\varepsilon \in [0,\varepsilon_3]$,
\[
\left\Vert \tilde U^{(0)}(1+\varepsilon) - \overline U^{(1)}_{s(\varepsilon)} \right\Vert_q < r^{(1)}.
\]
By uniqueness, $\tilde U^{(0)}\left([1,1+\varepsilon_3]\right)\subset \cC^{(1)}$.
\end{proof}

The following result can be used to determine that a path of solution does not undergo any secondary bifurcations. 

\begin{corollary}
\label{corollary:Jinvertible}
Consider $\displaystyle{T_s}$ defined by (\ref{T}) and assume that the hypotheses \eqref{ineq4} of Theorem~\ref{th_pol} and \eqref{eq:th_reg} of Theorem~\ref{th_reg} are satisfied. Then for every $s \in [0,1]$, $\dim \left(Ker\left(Df\left(\tilde U(s)\right)\right)\right)=1$, that is $\displaystyle{\tilde U(s)}$ is a regular path \big(where $ \displaystyle{s\mapsto \tilde U(s)}$ is defined in Theorem~\ref{th_local}\big). 
\end{corollary}

\begin{proof}
According to (\ref{ineq3}), we have that
\[
\sup\limits_{V,V'\in B(0,r)} \left\vert \left[DT_s\left(\overline U_s + V\right)\left(V'\right) \right]_{n}\right\vert \leq Z_{n}(r,s) \mbox{  and  } \sup\limits_{V,V'\in B(0,r)} \left\vert \left[DT_s\left(\overline U_s + V\right)\left(V'\right) \right]_{d}\right\vert \leq Z_{d}(r,s).
\]
Since $ \displaystyle{\tilde U(s) \in B(\overline U_s,r)}$, we get that
\[
\sup\limits_{\tilde V\in B(0,1)} \left\vert \left[DT_s\left(\tilde U(s)\right)\left( \tilde V\right) \right]_{n}\right\vert \leq \frac{Z_{n}(r,s)}{r} \mbox{  and  } \sup\limits_{\tilde V\in B(0,1)} \left\vert \left[DT_s\left(\tilde U(s)\right)\left(\tilde V\right) \right]_{d}\right\vert \leq \frac{Z_{d}(r,s)}{r}.
\]
Hence 
\[
 \left\Vert DT_s\left(\tilde U(s)\right) \right\Vert_q \leq \frac{\left\Vert Z(r,s) \right\Vert_q}{r} \leq \frac{\left\Vert Z(r,s)+Y(r,s) \right\Vert_q}{r} < 1,
\]
because $Y$ and $Z$ have non negative entries and (\ref{ineq4}) holds. So we have that
\[
 \left\Vert I- JDF_s\left(\tilde U(s)\right) \right\Vert_q  < 1.
\]
We get that $\displaystyle{JDF_s\left(\tilde U(s)\right)}$ is invertible \big(its inverse is given by $\displaystyle{\sum_{k=0}^{\infty}\left(I-JDF_s(\tilde U(s))\right)^k}$\big), and so $ \displaystyle{DF_s\left(\tilde U(s)\right)}$ is injective. 
Let us show that  for every $s \in [0,1]$, $\dim\left(Ker\left(Df\left(\tilde U(s)\right)\right)\right) = 1$. Suppose by contradiction that $ \displaystyle{\dim\left(Ker\left(Df\left(\tilde U(s)\right)\right)\right) > 1}$. Let $\displaystyle{U,V \in \Omega_q}$ two linearly independent non-trivial vectors such that $\displaystyle{Df\left(\tilde U(s)\right)(U)=Df\left(\tilde U(s)\right)(V)=0}$.
Since $\displaystyle{DF_s\left(\tilde U(s)\right)}$ is injective,
\[
DF_s\left(\tilde U(s)\right)(U)=
\begin{pmatrix}
DE_s\left(\tilde U(s)\right)(U) \\ Df\left(\tilde U(s)\right)(U)
\end{pmatrix}
=\begin{pmatrix}
\dot U_s \cdot U\\ 0
\end{pmatrix}
\neq 0.
\]
Hence, $\displaystyle{\dot U_s \cdot U \neq 0} $ and $\displaystyle{\dot U_s \cdot V \neq 0} $, and then we can define $\displaystyle{W= \frac{U}{\dot U_s \cdot U}- \frac{V}{\dot U_s \cdot V} \neq 0} $. We conclude that $ \displaystyle{DF_s\left(\tilde U(s)\right)(W)=0}$. This is a contradiction. Hence $ \dim\left(Ker\left(Df\left(\tilde U(s)\right)\right)\right) \leq 1$. 
The derivative of the relation $ f\left(\tilde U(s)\right)=0$ with respect to $s$ is given by
\[
Df(\tilde U(s)) \left( \frac{d \tilde U}{ds}(s)\right)=0.
\]
We showed in the proof of Theorem~\ref{th_reg} since \eqref{eq:th_reg} holds, that $\displaystyle{ \frac{d \tilde U}{ds}(s) \neq 0}$. Hence, for every $s \in [0,1]$, $\dim\left(Ker\left(Df\left(\tilde U(s)\right)\right)\right)=1$, meaning by definition that $\tilde U(s)$ is a regular path.
\end{proof}

\section{Application to a 3-component reaction-diffusion PDEs} \label{sec:applic}

In this section, we present all quantities and estimates to construct explicitly the radii polynomials required to apply the theory of the general method of Section~\ref{method} to the problem of computing rigorously global smooth branches of steady states of the system of three reaction-diffusion PDEs given by \eqref{eq:mimura_system}. The first step is to consider a Galerkin projection of $f$ given in $\eqref{fourier}$. 

\subsection{Finite dimensional projection} \label{sec:finite_dim_projection}
Given a finite dimensional parameter $m$, denote $x^{[m]} \in \mathbb{R}^m$ to be $(x_0,\dots,x_{m-1})$ and $u^{[m]} \in \mathbb{R}^{3m}$ to be $(x_0,y_0,z_0,\dots , x_{m-1},y_{m-1},z_{m-1})$. The finite dimensional Galerkin projection of $f$ is 
\[
f^{[m]} \bydef 
\left\{
\begin{aligned}
\mathbb{R}\times \mathbb{R}^{3m} &\longrightarrow \mathbb{R}^{3m}\\
U^{[m]} \bydef \left(d,u^{[m]}\right) &\longmapsto \left(f_0^{[m]}\left(U^{[m]}\right),\dots,f_{m-1}^{[m]}\left(U^{[m]}\right)\right),
\end{aligned}
\right.
\]
where for $n \in \{0, \dots , m-1\}$,
\begin{eqnarray}
\label{f^m}
f_{n_x}^{[m]}\left(U^{[m]}\right) & = & (r_1-d(\pi n)^2)x_n + \frac{1}{\varepsilon}y_n - \frac{a_1}{2} \sum_{\scriptsize\substack{\vert k \vert < m \\ \vert n-k \vert < m}} x_{\vert n-k \vert} x_{\vert k \vert} - \frac{a_1}{2} \sum_{\scriptsize\substack{\vert k \vert < m \\ \vert n-k \vert < m}} x_{\vert n-k \vert} y_{\vert k \vert} \nonumber \\
& - & (\frac{b_1}{2}+\frac{1}{2\varepsilon N})\sum_{\scriptsize\substack{\vert k \vert < m \\ \vert n-k \vert < m}} x_{\vert n-k \vert}z_{\vert k \vert} - \frac{1}{2\varepsilon N}\sum_{\scriptsize\substack{\vert k \vert < m \\ \vert n-k \vert < m}} y_{\vert n-k \vert} z_{\vert k \vert}, \nonumber \\
f_{n_y}^{[m]}\left(U^{[m]}\right) & = & ((r_1 - \frac{1}{\varepsilon} - (d+\beta N) (\pi n)^2 )y_n - \frac{a_1}{2} \sum_{\scriptsize\substack{\vert k \vert < m \\ \vert n-k \vert < m}} y_{\vert n-k \vert} y_{\vert k \vert} - \frac{a_1}{2} \sum_{\scriptsize\substack{\vert k \vert < m \\ \vert n-k \vert < m}} x_{\vert n-k \vert} y_{\vert k \vert} \nonumber \\
& - & (\frac{b_1}{2}-\frac{1}{2\varepsilon N})\sum_{\scriptsize\substack{\vert k \vert < m \\ \vert n-k \vert < m}} y_{\vert n-k \vert} z_{\vert k \vert} + \frac{1}{2\varepsilon N}\sum_{\scriptsize\substack{\vert k \vert < m \\ \vert n-k \vert < m}} x_{\vert n-k \vert} z_{\vert k \vert}, \nonumber \\
f_{n_z}^{[m]}\left(U^{[m]}\right) & = & (r_2 - d (\pi n)^2 )z_n - \frac{a_2}{2} \sum_{\scriptsize\substack{\vert k \vert < m \\ \vert n-k \vert < m}} z_{\vert n-k \vert} z_{\vert k \vert}  \nonumber \\
&-& \frac{b_2}{2} \sum_{\scriptsize\substack{\vert k \vert < m \\ \vert n-k \vert < m}} x_{\vert n-k \vert} z_{\vert k \vert} - \frac{b_2}{2} \sum_{\scriptsize\substack{\vert k \vert < m \\ \vert n-k \vert < m}} y_{\vert n-k \vert} z_{\vert k \vert} = 0.
\end{eqnarray}

\subsection{Explicit construction of the contraction \boldmath $T_s$ \unboldmath}
\label{def_J}

To define the Newton-like operator $T_s$ given in (\ref{T}), remember that we have to build an injective linear operator $J$ which approximates the inverse of $DF_0\left(\overline U_0\right)$ (with $\overline U_0$ an approximate zero of $f$ given by a predictor-corrector step).
Since 
\[
DF_0\left(\overline U_0\right)=
\left(
\begin{array}{cccccc}
\multicolumn{6}{c}{ \left( \dot U_0 \right) ^T} \\ \hline
\multicolumn{1}{c|}{ } & & & & & \\
\multicolumn{1}{c|}{ } & & & & & \\ 
\multicolumn{1}{c|}{\frac{\partial f}{\partial d}{\left(\overline U_0\right)}} &\multicolumn{5}{c}{\qquad D_uf\left(\overline U_0\right)\qquad} \\  
\multicolumn{1}{c|}{ } & & & & & \\ 
\multicolumn{1}{c|}{ } & & & & & \\ 
\end{array}
\right),
\]
we take
\begin{equation}
\label{J}
J=
\begin{pmatrix}
J^{[m]} & & 0 & \\
 & J_m & & \\
0 & & J_{m+1} & \\
 & & & \ddots\\
\end{pmatrix},
\end{equation}
where $J^{[m]}$ is a numerical inverse of $DF_0^{[m]}\left(\overline U^{[m]}\right)$ that is
\[
J^{[m]} \approx
\left(
\begin{array}{cccccc}
\multicolumn{6}{c}{\left(\dot U_0^{[m]}\right)^T} \\ \hline
\multicolumn{1}{c|}{ } & & & & & \\
\multicolumn{1}{c|}{ } & & & & & \\ 
\multicolumn{1}{c|}{\frac{\partial f^{[m]}}{\partial d}{\left(\overline U^{[m]}_0\right)}} &\multicolumn{5}{c}{\qquad D_uf^{[m]}\left(\overline U^{[m]}_0\right)\qquad} \\  
\multicolumn{1}{c|}{ } & & & & & \\ 
\multicolumn{1}{c|}{ } & & & & & \\ 
\end{array}
\right)^{-1},
\]
and $J_n$ ($n\geq m$) is a $3\times3$ matrix defined as
\renewcommand{\arraystretch}{1.5}
\begin{eqnarray}
\label{Jn}
J_n & = &
\begin{pmatrix}
r_1-d(\pi n)^2 & \frac{1}{\varepsilon} & 0 \\
0 & r_1-\frac{1}{\varepsilon}-(d+\beta N)(\pi n)^2 & 0 \\
0 & 0 & r_2-d(\pi n)^2
\end{pmatrix}^{-1} \nonumber \\
& = & 
\begin{pmatrix}
\frac{1}{r_1-d(\pi n)^2} & -\frac{1}{\varepsilon(r_1-d(\pi n)^2)(r_1-\frac{1}{\varepsilon}-(d+\beta N)(\pi n)^2)} & 0 \\
0 & \frac{1}{r_1-\frac{1}{\varepsilon}-(d+\beta N)(\pi n)^2} & 0 \\
0 & 0 & \frac{1}{r_2-d(\pi n)^2}
\end{pmatrix}.
\end{eqnarray}
$J_n$ is the inverse of the linear part of $f_n(d,\cdot)=\left(f_{n_x}(d,\cdot),f_{n_y}(d,\cdot),f_{n_z}(d,\cdot)\right)$. The idea behind this choice is explained soon and its interest will appear concretely in Section~\ref{Zn}. To prove that $J$ is really injective, we compute $\left\Vert I^{[m]} - J^{(m)}DF_0(\overline U_0) \right\Vert_{\infty}$ using interval arithmetic and check that its value is less than one, and then prove that the matrices $J_n$ ($n\geq m$) are invertible \big(we only need to check using interval arithmetic that for all $ n\geq m$, $r_1-d(\pi n)^2 \neq 0$ and $r_2-d(\pi n)^2 \neq 0$, since $r_1-\frac{1}{\varepsilon}-(d+\beta N)(\pi n)^2 < 0$ with the values of the parameters we are considering\big). 

%
%

\subsection{Verifying that \boldmath $T_s\left(\Omega_q\right)\subset \Omega_q$ \unboldmath and bootstrap argument}
\label{sec:bootstrap}

Remark that because of the $(\pi n)^2$ terms, $f$ goes from the space $\Omega_q$ to the space $ \left\lbrace u \ \vert \ \left\Vert u \right\Vert_{q-2} < \infty\right\rbrace$. Indeed, since $\Omega_{q}$ is a Banach algebra (see Section~\ref{Estimes}), the non linear terms in $f$ do not affect the decay rate of $f(U)$. However, thanks to the choice of $J_n$, $J$ defined by \eqref{J} goes from $\Omega_{q-2}$ to $\Omega_q$ so that $T_s\left(\Omega_q\right)\subset \Omega_q$. Remark that if $q \in (1,2)$, we cannot directly conclude that the solution $U\in \Omega_q$ of $f=0$ is a strong solution of (\ref{eq:mimura_system}). However, since $f(U)=0$, then (\ref{fourier}) becomes
\[
  \left\{
      \begin{aligned}
        & d(\pi n)^2x_n = r_1x_n + \frac{1}{\varepsilon}y_n - a_1 [x^2]_n - a_1 [x\ast y]_n - (b_1+\frac{1}{\varepsilon N})[xz]_n - \frac{1}{\varepsilon N}[y\ast z]_n, \\
        & (d+\beta N) (\pi n)^2y_n = ((r_1 - \frac{1}{\varepsilon})y_n - a_1 [y^2]_n - a_1 [x\ast y]_n - (b_1-\frac{1}{\varepsilon N})[y\ast z]_n + \frac{1}{\varepsilon N}[x\ast z]_n, \\
        & d (\pi n)^2 z_n = r_2z_n - a_2 [z^2]_n - b_2 [x\ast z]_n - b_2 [y\ast z]_n,
      \end{aligned}
    \right.
\]
where each right-hand-side is in $\Omega_q$ (here again we use the fact that $\left(\Omega_{q},\ast\right)$ is a Banach algebra). One can then easily see by dividing on both sides by $(\pi n)^2$ that $U$ is in fact in $\Omega_{q+2}$. We can repeat this bootstrap argument to prove that any zero $U$ of $f$ lying in some $\Omega_q$ ($q>1$) is in fact in every $\Omega_{q_0}$ for $q_0>1$ and hence corresponds to $C^{\infty}$ steady states of (\ref{eq:mimura_system}). Furthermore, the estimates of Section~\ref{Estimes} can be used to get explicit bounds for derivatives of any order for those solution functions, even if we did the proof with a $q \in (1,2)$.

We are almost ready to compute explicitly the bounds $Y$ and $Z$ and the radii polynomials. But to do so, we need to compute $D^2F_s$ and to bound the convolution product that appear in it. That is what we do in the next two sections.

\subsection{Computation of \boldmath  $DF_s$ \unboldmath and \boldmath $D^2F_s$ \unboldmath }

In this section, we use the notation
\[
\left\{
\begin{aligned}
&\overline U_0=(\overline d,\overline x_0,\overline y_0,\overline z_0,\dots,\overline x_n,\overline y_n,\overline z_n,\dots),\\
&V=(d,x_0,y_0,z_0,\dots,x_n,y_n,z_n,\dots),\\
&V'=(d',x'_0,y'_0,z'_0,\dots,x'_n,y'_n,z'_n,\dots).
\end{aligned}
\right.
\]
Recall the definition of $F_s$ in (\ref{F_s}).
First, 
\[
\left[DF_s(\overline U_0)(V)\right]_d = \dot U_s \cdot V,
\]
and
\[
\left[D^2F_s(\overline U_0)(V)(V')\right]_d = 0.
\]
With the expression of $f_n$ in mind (\ref{fourier}), we set
\[
L_n(U)=
\begin{pmatrix}
(r_1-d(\pi n)^2)x_n + \frac{1}{\varepsilon}y_n  \\
((r_1 - \frac{1}{\varepsilon} - (d+\beta N) (\pi n)^2 )y_n \\
 (r_2 - d (\pi n)^2 )z_n 
\end{pmatrix}
\]
and
\[
Q_n(U)=
\begin{pmatrix}
- a_1 [x^2]_n - a_1 [x\ast y]_n - (b_1+\frac{1}{\varepsilon N})[x\ast z]_n - \frac{1}{\varepsilon N}[y\ast z]_n \\
- a_1 [y^2]_n - a_1 [x\ast y]_n - (b_1-\frac{1}{\varepsilon N})[y\ast z]_n + \frac{1}{\varepsilon N}[x\ast z]_n \\
- a_2 [z^2]_n - b_2 [x\ast z]_n - b_2 [y\ast z]_n 
\end{pmatrix},
\]
so that $f_n=L_n+Q_n$. Then we have
\begin{equation}
\label{df}
\left[DF_s(\overline U_0)(V)\right]_n=Df_n(\overline U_0)(V)=D_dL_n(\overline U_0)(d) + D_uL_n(\overline U_0)(v) + DQ_n(\overline U_0)(V)
\end{equation}
and
\begin{eqnarray}
\left[D^2F_s(\overline U_0)(V)(V')\right]_n &=& D^2f_n(\overline U_0)(V)(V') 
\label{d2f} \\
\nonumber
&=& D^2_{ud}L_n(\overline U_0)(d)(v') + D^2_{du}L_n(\overline U_0)(v)(d') + D^2Q_n(\overline U_0)(V)(V').
\end{eqnarray}
Also, $D^2Q_n(\overline U_0)(V)(V')=DQ_n(V)(V')$ because $DQ_n$ is linear ($Q_n$ is quadratic). More explicitly, all terms in \eqref{df} and \eqref{d2f} can be recovered from the fact that
\[
D_dL_n(\overline U_0)(d)=-(\pi n)^2 \overline u_n d,
\]
\[
D_uL_n(\overline U_0)(v)=
\begin{pmatrix}
(r_1- \overline d(\pi n)^2)x_n + \frac{1}{\varepsilon}y_n  \\
((r_1 - \frac{1}{\varepsilon} - ( \overline d+\beta N) (\pi n)^2 )y_n \\
 (r_2 -  \overline d (\pi n)^2 )z_n 
\end{pmatrix},
\]
$DQ_n(V)(V')=$
\[
\begin{small}
\hspace{-1cm}
\begin{pmatrix}
-2a_1[x\ast x']_n -a_1[y\ast x']_n -(b_1+\frac{1}{\varepsilon N})[z\ast x']_n  -a_1[x\ast y']_n -\frac{1}{\varepsilon N}[z\ast y']_n  -b_1[x\ast z']_n -\frac{1}{\varepsilon N}([x\ast z']_n +[y\ast z']_n)\\
-a_1[y\ast x']_n +\frac{1}{\varepsilon N}[z\ast x']_n  -2a_1[y\ast y']_n-a_1[x\ast y']_n+(\frac{1}{\varepsilon N}-b_1)[z\ast y']_n -b_1[y\ast z']_n +\frac{1}{\varepsilon N}([x\ast z']_n +[y\ast z']_n)\\
-b_2[z\ast x']_n  -b_2[z\ast y']_n  -2a_2[z\ast z']_n-b_2([x\ast z']_n +[y\ast z']_n)
\end{pmatrix},
\end{small}
\]
\[
D^2_{ud}L_n(\overline U_0)(d)(v')=- \left( \pi n \right)^2 dv'_n
\]
and
\[
D^2_{du}L_n(\overline U_0)(v)(d')=- \left( \pi n \right)^2 d'v_n.
\]

\subsection{Analytic estimates}
\label{Estimes}

In order to bound all terms in \eqref{df} and \eqref{d2f},  in particular quantities like $[x \ast y]_n  $, for $n \in \mathbb{N}$, we have to develop some analytic estimates. Similar estimates have been produced for the case $q \ge 2$ (e.g. \cite{MR2718657}), but not for the case $q \in (1,2)$. From these estimates, we get that $\left(\Omega_q, \ast\right)$ is a Banach algebra for each $q>1$. First notice that, for all $x,y\in \Omega_q$,
\[
\left\vert \left[ x\ast y \right]_n \right\vert \omega_n^q = \frac{1}{2}\omega_n^q \sum_{k\in\mathbb{Z}}x_{\vert k \vert}y_{\vert n-k \vert}
 \leq  \frac{1}{2} \sum_{k\in\mathbb{Z}} \frac{\omega_n^q}{\omega_k^q \omega_{n-k}^q} \left\Vert x \right\Vert_q \left\Vert y \right\Vert_q.
\]
Thus, what we need to show is that 
\[
\Psi_n^q  \bydef  \sum_{k\in\mathbb{Z}} \frac{\omega_n^q}{\omega_k^q \omega_{n-k}^q}
\]
is bounded for $n\in\mathbb{N}$. We start by rewriting $\Psi_n^q$.
If $n=0$,
\[
\sum_{k\in\mathbb{Z}} \frac{\omega_n^q}{\omega_k^q \omega_{n-k}^q} = \sum_{k=-\infty}^{-1} \frac{\omega_n^q}{\omega_k^q \omega_{n-k}^q} + 1 + \sum_{k=1}^{\infty} \frac{\omega_n^q}{\omega_k^q \omega_{n-k}^q} 
= 1 + 2\sum_{k=1}^{\infty} \frac{ n^q}{k^q \left(n+k\right)^q},
\]
and if $n>0$,
\begin{eqnarray}
\label{decompo}
\sum_{k\in\mathbb{Z}} \frac{\omega_n^q}{\omega_k^q \omega_{n-k}^q} &=& \sum_{k=-\infty}^{-1} \frac{\omega_n^q}{\omega_k^q \omega_{n-k}^q} + 1 +\sum_{k=1}^{n-1} \frac{\omega_n^q}{\omega_k^q \omega_{n-k}^q} + 1 + \sum_{k=n+1}^{\infty} \frac{\omega_n^q}{\omega_k^q \omega_{n-k}^q} \nonumber \\
&=& 2 + 2\sum_{k=1}^{\infty} \frac{ n^q}{k^q \left(n+k\right)^q} + \sum_{k=1}^{n-1} \frac{ n^q}{k^q \left(n-k\right)^q}.
\end{eqnarray}

In everything that follows, $K$ is a computational parameter (the larger $K$ is, the sharper the estimates will, but the greater the computational cost for the evaluation of the estimates will be) and $M$ is another computational parameter (which is taken equal to $2m-1$, see Section~\ref{Y}). First we define, for $n \geq 2$ 
\begin{equation} \label{eq:chi_n}
\chi_n(q)=\left(\frac{q}{2-q}+\frac{q(q-1)}{2(3-q)}+\frac{q(q-1)}{2\left\lfloor \frac{n}{2} \right\rfloor}+\frac{2-\left(2/3\right)^q}{\left\lfloor \frac{n}{2} \right\rfloor}- \frac{2-\left(2/3\right)^q}{q-1}\right)\frac{1}{\left\lfloor \frac{n}{2} \right\rfloor^{q-1}},
\end{equation}
and $q^*(M)$ the unique zero of $\chi_M$ in $(1,2)$. Note that $\chi_M$ is increasing on $(1,2)$, goes to $-\infty$ as $q$ goes to $1$ and to $\infty$ as $q$ goes to $2$, so $q^*(M)$ is well defined.
Then we define 
\[
\gamma_M^q(K)  \bydef 
\left\{
\begin{aligned}
& 2\sum_{k=1}^K\frac{1}{k^q} + \frac{2}{(q-1)K^{q-1}},\quad \mbox{if } 1<q<q^*(M)\\
& 2\sum_{k=1}^K\frac{1}{k^q}+\frac{2}{(q-1)K^{q-1}}+2\chi_M(q),\quad \mbox{if } q^*(M)\leq q <2\\
& 2\left(\frac{M}{M-1}\right)^q + \left( \frac{4\ln(M-2)}{M}+\frac{\pi^2-6}{3}\right)\left(\frac{2}{M}+\frac{1}{2}\right)^{q-2} ,\quad \mbox{if } q \geq 2,
\end{aligned}
\right.
\]
and finally
\begin{equation}
\label{def_alpha}
\alpha_n^q(K)  \bydef 
\left\{
\begin{aligned}
& 1 + 2\sum_{k=1}^K\frac{1}{k^q} + \frac{2}{(q-1)K^{q-1}},\quad \mbox{if } n=0\\
& 2 + 2\sum_{k=1}^K\frac{1}{k^q} + \frac{2}{(q-1)K^{q-1}} + \sum_{k=1}^{n-1} \frac{ n^q}{k^q \left(n-k\right)^q},\quad \mbox{if } 1\leq n <M\\
& 2 + 2\sum_{k=1}^K\frac{1}{k^q} + \frac{2}{(q-1)K^{q-1}} + \gamma^{(q)}_M ,\quad \mbox{if } n \geq M.
\end{aligned}
\right.
\end{equation}

\begin{prop}
\label{alpha}
Let $q>1$, $K$ and $M\geq 6$ computational parameters. For all $n\in\mathbb{N}$,
\[
\Psi_n^q \leq \alpha_n^q(K).
\] 
\end{prop}

This allows us to state the following result.
\begin{lem}
\label{n>m}
Let $x,y\in \Omega_q$, $q>1$, $K$ and $M\geq 6$ computational parameters. For all $n\geq M$,
\[
\left\vert \left[ x\ast y \right]_n  \right\vert\leq \frac{1}{2}\frac{\alpha_M^q(K)}{\omega_n^q}\left\Vert x \right\Vert_q \left\Vert y \right\Vert_q.
\]
\end{lem}
This bound, in addition of being very useful later in this paper, proves that $\left(\Omega_q,\ast\right)$ is a Banach algebra for $q>1$.

\begin{proof}{\em (of Proposition~\ref{alpha})}
The bound for $q\geq 2$ is due to \cite{MR2718657}. We prove the bound for $q<2$. The case $n<M$ is a direct consequence of the following inequality applied to (\ref{decompo})
\[
\sum_{k=1}^{\infty} \frac{ n^q}{k^q \left(n+k\right)^q} \leq \sum_{k=1}^{\infty} \frac{1}{k^q} \leq \sum_{k=1}^K\frac{1}{k^q} + \frac{1}{(q-1)K^{q-1}}.
\]
For the case $n\geq M$, let us consider the difference
\[
\Delta_n^q \bydef \Psi_n^q - \left(2+4\sum_{k=1}^{\infty}\frac{1}{k^q}\right).
\]
Using (\ref{decompo}) and the inequality below
\begin{eqnarray*}
\sum_{k=1}^{n-1}\frac{n^q}{k^q(n-k)^q} & = & \sum_{k=1}^{\left\lfloor \frac{n}{2} \right\rfloor }\frac{n^q}{k^q(n-k)^q}+\sum_{k=\left\lfloor \frac{n}{2} \right\rfloor+1}^{n-1}\frac{n^q}{k^q(n-k)^q} \\
& = & \sum_{k=1}^{\left\lfloor \frac{n}{2} \right\rfloor}\frac{n^q}{k^q(n-k)^q}+\sum_{l=1}^{n- \left\lfloor \frac{n}{2} \right\rfloor -1}\frac{n^q}{l^q(n-l)^q} \qquad \mbox{(we set }l=n-k) \\
& \leq & 2\sum_{k=1}^{\left\lfloor \frac{n}{2} \right\rfloor }\frac{n^q}{k^q(n-k)^q} \qquad \left(n-\left\lfloor \frac{n}{2} \right\rfloor -1 \leq \left\lfloor \frac{n}{2} \right\rfloor\right),
\end{eqnarray*}
we get that
\begin{eqnarray}
\label{Delta}
\Delta_n^q & \leq & 2\left(\sum_{k=1}^{\infty}\frac{1}{k^q}\left(\frac{n^q}{(n+k)^q}-1\right) + \sum_{k=1}^{\left\lfloor \frac{n}{2} \right\rfloor}\frac{1}{k^q}\left(\frac{n^q}{(n-k)^q}-1\right) -\sum_{k=\left\lfloor \frac{n}{2} \right\rfloor+1}^{\infty}\frac{1}{k^q} \right) \nonumber\\
& \leq & 2\left(\sum_{k=1}^{\left\lfloor \frac{n}{2} \right\rfloor}\frac{1}{k^q}\left(\frac{n^q}{(n-k)^q}-1\right) -\left(2-\left(\frac{2}{3}\right)^q\right)\sum_{k=\left\lfloor \frac{n}{2} \right\rfloor+1}^{\infty}\frac{1}{k^q} \right).
\end{eqnarray}
The first term of (\ref{Delta}) can be bounded from above as follows
\begin{eqnarray}
\label{term1}
\sum_{k=1}^{\left\lfloor \frac{n}{2} \right\rfloor}\frac{1}{k^q}\left(\frac{n^q}{(n-k)^q}-1\right) & = & \sum_{k=1}^{\left\lfloor \frac{n}{2} \right\rfloor }\frac{1}{k^q}\left( \left(1+\frac{k}{(n-k)}\right)^q -1\right) \nonumber \\
& = & \sum_{k=1}^{\left\lfloor \frac{n}{2} \right\rfloor }\frac{1}{k^q}\sum_{l=1}^{\infty}\frac{q(q-1)\dots (q-l+1)}{l!}\frac{k^l}{(n-k)^l} \nonumber \\
& \leq & \sum_{k=1}^{\left\lfloor \frac{n}{2} \right\rfloor }\frac{1}{k^q}\left(q\frac{k}{(n-k)}+\frac{q(q-1)}{2}\frac{k^2}{(n-k)^2} \right). 
\end{eqnarray}
The last inequality is due to the fact that for all $u \in (0,1)$, the series expansion 
\[
(1+u)^q=1+\sum_{k=1}^{\infty}\frac{q(q-1)\dots (q-k+1)}{k!}u^k
\]
is an alternating series for $k\geq 2$ (recall that $q\in (1,2)$).
According to (\ref{term1}), we have to bound $\displaystyle{\sum_{k=1}^{\left\lfloor \frac{n}{2} \right\rfloor }\frac{q}{k^q}\frac{k}{(n-k)}}$ and $\displaystyle{\sum_{k=1}^{\left\lfloor \frac{n}{2} \right\rfloor }\frac{q(q-1)}{2k^q}\frac{k^2}{(n-k)^2}}$.
First,
\[
\sum_{k=1}^{\left\lfloor \frac{n}{2} \right\rfloor }\frac{1}{k^{q-1}}  \leq  1 + \int_{1}^{\left\lfloor \frac{n}{2} \right\rfloor} \frac{dt}{t^{q-1}} 
 =  1+ \frac{1}{2-q}\left(\left\lfloor \frac{n}{2} \right\rfloor^{2-q}-1 \right) 
 \leq  \frac{1}{2-q}\left\lfloor \frac{n}{2} \right\rfloor^{2-q},
\]
and then,
\begin{equation}
\label{somme1}
\sum_{k=1}^{\left\lfloor \frac{n}{2} \right\rfloor }\frac{q}{k^q}\frac{k}{(n-k)}  \leq  \frac{q}{\left\lfloor \frac{n}{2} \right\rfloor}\sum_{k=1}^{\left\lfloor \frac{n}{2} \right\rfloor }\frac{1}{k^{q-1}} 
 \leq  \frac{q}{2-q}\frac{1}{\left\lfloor \frac{n}{2} \right\rfloor^{q-1}}. 
\end{equation}
Similarly 
\[
\sum_{k=1}^{\left\lfloor \frac{n}{2} \right\rfloor }k^{2-q} \leq \int_{1}^{\left\lfloor \frac{n}{2} \right\rfloor}t^{2-q}dt+\left\lfloor \frac{n}{2} \right\rfloor^{2-q}
= \frac{1}{3-q}\left(\left\lfloor \frac{n}{2} \right\rfloor^{3-q}-1 \right)+\left\lfloor \frac{n}{2} \right\rfloor^{2-q}
\leq \frac{1}{3-q}\left\lfloor \frac{n}{2} \right\rfloor^{3-q}+\left\lfloor \frac{n}{2} \right\rfloor^{2-q},
\]
and then
\begin{equation}
\label{somme2}
\sum_{k=1}^{\left\lfloor \frac{n}{2} \right\rfloor }\frac{q(q-1)}{2k^q}\frac{k^2}{(n-k)^2}  \leq  \frac{q(q-1)}{2\left\lfloor \frac{n}{2} \right\rfloor^2}\sum_{k=1}^{\left\lfloor \frac{n}{2} \right\rfloor }k^{2-q}
\leq \frac{q(q-1)}{2\left\lfloor \frac{n}{2} \right\rfloor^{q-1}}\left(\frac{1}{3-q}+\frac{1}{\left\lfloor \frac{n}{2} \right\rfloor} \right).
\end{equation}
According to (\ref{term1}), (\ref{somme1}) and (\ref{somme2}), we get 
\begin{equation}
\label{difference}
\sum_{k=1}^{\left\lfloor \frac{n}{2} \right\rfloor}\frac{1}{k^q}\left(\frac{n^q}{(n-k)^q}-1\right) \leq \left(\frac{q(q-1)}{2(3-q)}+\frac{q(q-1)}{2\left\lfloor \frac{n}{2} \right\rfloor}+\frac{q}{2-q}  \right)\frac{1}{\left\lfloor \frac{n}{2} \right\rfloor^{q-1}}.
\end{equation}
Then we bound the second term of (\ref{Delta}) from below
\begin{equation}
\label{reste}
\sum_{k=\left\lfloor \frac{n}{2} \right\rfloor +1}^{\infty }\frac{1}{k^q } = \sum_{k=\left\lfloor \frac{n}{2} \right\rfloor}^{\infty }\frac{1}{k^q }-\frac{1}{\left\lfloor \frac{n}{2} \right\rfloor^{q}} 
\geq \int_{\left\lfloor \frac{n}{2} \right\rfloor}^{\infty} \frac{dt}{t^q}-\frac{1}{\left\lfloor \frac{n}{2} \right\rfloor^{q}}
 =  \frac{1}{q-1}\frac{1}{\left\lfloor \frac{n}{2} \right\rfloor^{q-1}}-\frac{1}{\left\lfloor \frac{n}{2} \right\rfloor^{q}}.
\end{equation}
Using (\ref{Delta}), (\ref{difference}) and (\ref{reste}) , we get  
\begin{equation}
\label{lim}
\Delta_n^q \leq  2\left( \frac{q(q-1)}{2(3-q)}+\frac{q(q-1)}{2\left\lfloor \frac{n}{2} \right\rfloor}+\frac{q}{2-q}+\frac{1}{\left\lfloor \frac{n}{2} \right\rfloor}-\frac{1}{q-1} \right)\frac{1}{\left\lfloor \frac{n}{2} \right\rfloor^{q-1}} = 2\chi_n(q).
\end{equation}
Now, if $q<q^*(M)$ then $\chi_M(q)<0$ and hence for all $n \geq M$, $\chi_n(q)<0$. Thus
\[
\Psi_n^q \leq 2 + 4\sum_{k=1}^{\infty}\frac{1}{k^q} \leq 2 + 4\sum_{k=1}^{K}\frac{1}{k^q} + \frac{4}{(q-1)K^{q-1}}.
\]
If $q\geq q^*(M)$ then $\chi_M(q)\geq 0$ and we have two cases. The first case is that $\chi_n(q)$ decreases until becoming negative for some $n_0>M$ which implies that $\chi_n(q)<0$ for all $n \geq n_0$. The second case is that $\chi_n(q)$ stays positive for all $n\geq M$ which implies that $\chi_n(q)$ is decreasing for all $n\geq M$. In both cases we get that $\chi_n(q)\leq \chi_M(q)$ for all $n\geq M$ and thus
\[
\Psi_n^q \leq 2 + 4\sum_{k=1}^{\infty}\frac{1}{k^q} + 2\chi_M(q) \leq 2 + 4\sum_{k=1}^{K}\frac{1}{k^q} + \frac{4}{(q-1)K^{q-1}} + 2\chi_M(q).
\]
\end{proof}

Notice that since $\lim\limits_{n}\chi_n(q)=0$, (\ref{lim}) shows that $\limsup\limits_n \Delta_n^q \leq 0$. We could do the same kind of computation to bound $\Delta_n^q$ from below and show that in fact $\lim\limits_{n}\Delta_n^q=0$. So the bound for $q<q^*(M)$ is optimal, the only thing that can be improved is the way we approximate $\displaystyle{\sum_{k=1}^{\infty}\frac{1}{k^q}}$ (which depends of $K$). But this also shows that the bound for $q^*(M)\leq q<2$ may not be optimal, in fact it becomes quite bad when $q$ is close to $2$ since $\lim\limits_{q\rightarrow 2}\chi_M(q)=\infty$. However, if sharp estimates are needed for $q$ close to $2$, there is a numerical way to get almost optimal bounds which is detailed it in Appendix~\ref{appendix:sharper_estimates}. We now give other bounds that are sharper for $n<M$ by using computations. Let us define
\[
C_n^q(K)  \bydef   \sum\limits_{\substack {  n_1+n_2= n \\ \vert n_1 \vert,\vert n_2 \vert <K }} \frac{1}{\omega_{n_1}^q}\frac{1}{\omega_{n_2}^q},
\]
and
\[
\epsilon_n^q(K) \bydef  \frac{2}{(q-1)(K-1)^{q-1}}\left( \frac{1}{(K-n)^q}+\frac{1}{(K+n)^q} \right).
\]

\begin{lem}[Sharper estimates] 
\label{n<m}
Let $x,y\in \Omega_q$, $q>1$, $K$ and $M$ computational parameters, $M\leq K$. for all $ n<M$
\[
\left\vert \left[ x \ast y \right]_n \right\vert \leq  \frac{1}{2}\left(C_n^q(K)+\epsilon_n^q(K)\right)\left\Vert x \right\Vert_{q}\left\Vert y \right\Vert_{q},
\]
\end{lem}

\begin{proof}
\[
\left\vert \left[ x \ast y \right]_n \right\vert \leq \frac{1}{2} \left(\sum\limits_{\substack {  n_1+n_2= n }} \frac{1}{\omega_{n_1}^q}\frac{1}{\omega_{n_2}^q}\right)\left\Vert x \right\Vert_{q}\left\Vert y \right\Vert_{q}.  
\]
We can split the summation in two parts :
\[
\sum\limits_{\substack {  n_1+n_2= n }} \frac{1}{\omega_{n_1}^q}\frac{1}{\omega_{n_2}^q} = \sum\limits_{\substack {  n_1+n_2= n \\ \vert n_1 \vert,\vert n_2 \vert <K }} \frac{1}{\omega_{n_1}^q}\frac{1}{\omega_{n_2}^q}+\sum\limits_{\substack {  n_1+n_2= n \\ \max(\vert n_1 \vert, \vert n_2 \vert)  \geq K }} \frac{1}{\omega_{n_1}^q}\frac{1}{\omega_{n_2}^q}.
\]
The first one is exactly $C_n^q(K)$. We now bound the second one :
\begin{eqnarray*}
\sum\limits_{\substack {  n_1+n_2= n \\ \max(\vert n_1 \vert, \vert n_2 \vert) \geq K }} \frac{1}{\omega_{n_1}^q}\frac{1}{\omega_{n_2}^q} & \leq & 2 \sum\limits_{\substack {  n_1+n_2= n \\ \vert n_1 \vert \geq K }} \frac{1}{\omega_{n_1}^q}\frac{1}{\omega_{n_2}^q} \\
& \leq & 2\sum_{n_1=K}^{\infty} \frac{1}{\omega_{n_1}^q}\left(\frac{1}{\omega_{n-n_1}^q}+\frac{1}{\omega_{n+n_1}^q}\right) \\
& \leq & \frac{2}{(q-1)(K-1)^{q-1}}\left( \frac{1}{(K-n)^q}+\frac{1}{(K+n)^q} \right).
\end{eqnarray*}
\end{proof}

\begin{remark}
$q^*(M)$, the unique zero of  $\chi_M$ in $(1,2)$ defined in \eqref{eq:chi_n}, is increasing in $M$ and converges rather rapidly towards a bounded value. In particular, for $M\ge 100$ (which is always the case for the proofs presented in this work), one has that $q^*(M) \ge q^*(100)=1.4730$. Numerically, the limit when $M$ goes to $\infty$ is about $1.475$.
\end{remark}

\subsection{Computation of \boldmath $Y_n$ \unboldmath }
\label{Y}

According to (\ref{ineq3}), we focus here on 
\[
\left\vert T_s\left(\overline U_s\right) - \overline U_s \right\vert  =  \left\vert JF_s\left(\overline U_s\right) \right\vert 
\leq  \left\vert J \right\vert \left\vert F_s\left(\overline U_s\right) \right\vert. 
\]
Remember that absolute values and inequalities applied to vectors or matrices should be understood component wise. Observe that $E_s(\overline U_s)=0$ so that
\[
F_s\left(\overline U_s\right) =
\begin{pmatrix}
0\\
f\left(\overline U_s\right)
\end{pmatrix}. 
\]
Because of the shape of $J$ in (\ref{J}), we compute separately the bounds for $n<m$ and $n \geq m$. 

\subsubsection{\underline{Case \boldmath $n <m$ \unboldmath}}

Following the notation introduced earlier, $Y^{[m]}$ is the vector containing $Y_d$ and $Y_{n}$ for any $n<m$. We want to bound the quantity
\[
\left\vert J^{[m]} \right\vert \left\vert
\begin{pmatrix}
0\\
f^{[m]}\left(\overline U_s\right)
\end{pmatrix}
\right\vert.
\]
With the expression $\overline U_s= \overline U_0 + s\Delta \overline U$ and a Taylor expansion, we get that
\[
f^{[m]}\left(\overline U^{[m]}_s\right) = f^{[m]}\left(\overline U_0^{[m]}\right) + sDf^{[m]}\left(\overline U_0^{[m]}\right)\left(\Delta \overline U^{[m]}\right) + \frac{s^2}{2}D^2f^{[m]}\left(\overline U_0^{[m]}\right)\left(\Delta \overline U^{[m]}\right)^2,
\] 
which can be bounded, for $s\in [0,1]$, by
\[
\tilde f^{[m]}  \bydef  \left\vert f^{[m]}\left(\overline U_0^{[m]}\right)\right\vert +s\left\vert Df^{[m]}\left(\overline U_0^{[m]}\right) \left( \Delta\overline U^{[m]} \right) \right\vert + \frac{s^2}{2}\left\vert D^2f^{[m]}\left(\overline U_0^{[m]}\right)\left(\Delta\overline U^{[m]}\right)^2 \right\vert .
\] 
We see in Section~\ref{I} that we can get a uniform bound in $s$ by taking $s=1$ in the expression above. This uniform bound is simple to get but not the sharpest. We actually show in Appendix~\ref{appendix:sharper_s_bound} how to compute a sharper bound. The vectors $Df^{[m]}\left(\overline U_0^{[m]}\right) \left( \Delta\overline U^{[m]} \right)$ and $D^2f^{[m]}\left(\overline U_0^{[m]}\right)\left(\Delta\overline U^{[m]}\right)^2$ are computed with the expressions (\ref{df}) and (\ref{d2f}) by truncating the convolution products  as in \eqref{f^m}. This is a finite computation. We can set
\[
\label{Ym}
Y^{[m]} \bydef 
\left\vert J^{[m]} \right\vert
\begin{pmatrix}
0\\
\tilde f^{[m]}
\end{pmatrix}.
\]
 
\subsubsection{\underline{Case \boldmath $n \geq m$ \unboldmath}} \label{sec:Yn_n_ge_m}

For $n \geq m $,
\[
\left\vert \left[ T_s(\overline U_s) - \overline U_s  \right]_n \right\vert  =  \left\vert J_nf_n(\overline U_s) \right\vert = \left\vert J_n \left( f_n( \overline U_0)+sDf_n(\overline U_0) (\Delta\overline U ) + \frac{s^2}{2} D^2f_n(\overline U_0)(\Delta\overline U)^2 \right) \right\vert .
\]
We can then set
\[
Y_n \bydef \left\vert J_n \right\vert \left( \left\vert  f_n( \overline U_0) \right\vert +  s\left\vert Df_n(\overline U_0)( \Delta\overline U ) \right\vert +\frac{s}{2} \left\vert D^2f_n(\overline U_0)(\Delta\overline U)^2 \right\vert \right).
\]
The terms $Df_n\left(\overline U_0 \right) \left( \Delta\overline U \right)$ and $D^2f_n\left(\overline U_0\right)\left(\Delta\overline U\right)^2$ are computed with the expressions (\ref{df}) and (\ref{d2f}).
We can compute $Y_n $ in a finite number of operations, because for any $n \geq m $, $ \left[ \overline U_0 \right]_n=0 $ and $ \left[ \Delta\overline U \right]_n=0$. In particular, for any $n \geq 2m-1 $, $f_n(\overline U_s)=0$ and we can take $Y_n=0$, so only a finite number of $Y_n$ remains to be computed. Hence, setting $M=2m-1$, we have that for all $ n\geq M$, $Y_n=0$. In the next section, we see that for all $ n\geq M$, we can set $Z_n=\hat Z_M\frac{\omega_M^q}{\omega_n^q}$. In fact, the value of $M$ is determined by the degree of the non linearities of $f$. Here we have quadratic terms like 
\[
\left[\overline x \ast \overline y\right]_n=\sum_{k=n-m+1}^{m-1}\overline x_k \overline y_{n-k} = 0 \mbox{ for } n\geq M=2m-1,
\]
and with non linearities of degree $p$, we would have the same by taking $M=p(m-1)+1$. 

\subsection{Computation of \boldmath $Z_n$ \unboldmath}
\label{Z}

For $V,V'\in B(0,r) $ and $s\in [0,1] $, using a Taylor expansion, we get 
\begin{eqnarray}
\label{DT}
 DT_s( \overline U_s + V)(V') & = & \left(I-JDF_s(\overline U_0+s\Delta\overline U + V)\right)(V') \nonumber \\ 
& = & \left(I-J \left( DF_s(\overline U_0)+ D^2F_s(\overline U_0 )(s\Delta\overline U+V) \right) \right)(V').  
\end{eqnarray}
As for $Y$, we compute separately the bounds for $n<m$ and $n \geq m$.

\subsubsection{\underline{Case \boldmath $n <m$ \unboldmath}}

For all $ V'\in B(0,r) $, the shape of $J$ in (\ref{J}) allows us to write 
\begin{eqnarray}
\label{I-DF}
\left[\left(I-JDF_s(\overline U_0 )\right)(V')\right]^{[m]} & = & V'^{[m]}-\left[JDF_s(\overline U_0 )(V')\right]^{[m]} \nonumber \\ 
& = & V'^{[m]}-J^{[m]}\left[DF_s(\overline U_0 )(V')\right]^{[m]} \nonumber \\
& = & V'^{[m]}-J^{[m]}\left(DF_s^{[m]}(\overline U_0 )\left(V'^{[m]}\right)+R^{[m]}(\overline U_0,V') \right) \nonumber \\
& = & \left( I^{[m]}-J^{(m)}DF_s^{[m]}(\overline U_0^{[m]} )\right)\left(V'^{[m]}\right)-J^{(m)}R^{[m]}(\overline U_0,V')\nonumber \\ 
& = & \left( I^{[m]}-J^{[m]}\left( DF_0^{[m]}(\overline U_0^{[m]}) + s
\begin{pmatrix}
\left(\Delta \dot U \right)^T \\
0
\end{pmatrix}
 \right) \right)\left(V'^{[m]}\right) \nonumber \\
 & & -J^{(m)}R^{[m]}(\overline U_0,V') ,
\end{eqnarray}
where
\[
R^{[m]}_d(\overline U_0,V')=\sum_{k=m}^{\infty} \frac{\partial E_s}{\partial x_k}(\overline U_0)x'_k+\sum_{k=m}^{\infty} \frac{\partial E_s}{\partial y_k}(\overline U_0)y'_k+\sum_{k=m}^{\infty} \frac{\partial E_s}{\partial z_k}(\overline U_0)z'_k=0,
\]
and for all $ n\in \{0,\dots,m-1\}$,
\begin{eqnarray*}
R^{[m]}_{n_x}(\overline U_0,V') & = & \sum_{k=m}^{\infty} \frac{\partial f_{n_x}}{\partial x_k}(\overline U_0)x'_k+\sum_{k=m}^{\infty} \frac{\partial f_{n_x}}{\partial y_k}(\overline U_0)y'_k+\sum_{k=m}^{\infty} \frac{\partial f_{n_x}}{\partial z_k}(\overline U_0)z'_k \\
& = & -a_1 \sum_{k=m}^{m+n-1} \overline x_{k-n}x'_k-\frac{a_1}{2}\sum_{k=m}^{m+n-1} \overline y_{k-n}x'_k-\frac{1}{2}(b_1+\frac{1}{\varepsilon N})\sum_{k=m}^{m+n-1} \overline z_{k-n}x'_k \\
& &-\frac{a_1}{2}\sum_{k=m}^{m+n-1} \overline x_{k-n}y'_k-\frac{1}{2\varepsilon N}\sum_{k=m}^{m+n-1} \overline z_{k-n}y'_k \\
& &-\frac{1}{2}(b_1+\frac{1}{\varepsilon N})\sum_{k=m}^{m+n-1} \overline x_{k-n}z'_k -\frac{1}{2\varepsilon N}\sum_{k=m}^{m+n-1} \overline y_{k-n}z'_k. 
\end{eqnarray*}
Similarly,
\begin{eqnarray*}
R^{[m]}_{n_y}(\overline U_0,V')
& = & -a_1 \sum_{k=m}^{m+n-1} \overline y_{k-n}y'_k-\frac{a_1}{2}\sum_{k=m}^{m+n-1} \overline x_{k-n}y'_k-\frac{1}{2}(b_1-\frac{1}{\varepsilon N})\sum_{k=m}^{m+n-1} \overline z_{k-n}y'_k \\
& &-\frac{a_1}{2}\sum_{k=m}^{m+n-1} \overline y_{k-n}x'_k+\frac{1}{2\varepsilon N}\sum_{k=m}^{m+n-1} \overline z_{k-n}x'_k \\
& &-\frac{1}{2}(b_1-\frac{1}{\varepsilon N})\sum_{k=m}^{m+n-1} \overline y_{k-n}z'_k +\frac{1}{2\varepsilon N}\sum_{k=m}^{m+n-1} \overline x_{k-n}z'_k, 
\end{eqnarray*}
and
\begin{eqnarray*}
R^{[m]}_{n_z}(\overline U_0,V')
& = & -\frac{b_2}{2}\sum_{k=m}^{m+n-1} \overline z_{k-n}x'_k -\frac{b_2}{2}\sum_{k=m}^{m+n-1} \overline z_{k-n}y'_k \\
& &-a_2\sum_{k=m}^{m+n-1} \overline z_{k-n}z'_k -\frac{b_2}{2}\sum_{k=m}^{m+n-1} \overline x_{k-n}z'_k-\frac{b_2}{2}\sum_{k=m}^{m+n-1} \overline y_{k-n}z'_k. 
\end{eqnarray*}
$R^{[m]}(\overline U_0 ,V')$ can be bounded uniformly for $V' \in B(0,r) $ by $\tilde R^{[m]}(\overline U_0)r$ where 
\[
\tilde R^{[m]}_d(\overline U_0)  \bydef  0,
\]
\[
  \tilde R^{[m]}_{n_x}(\overline U_0) 
  \bydef  \frac{1}{2}\sum_{k=m}^{m+n-1}\left((3a_1+b_1+\frac{1}{\varepsilon N})\vert \overline x_{k-n} \vert +(a_1+ \frac{1}{\varepsilon N})\vert \overline y_{k-n} \vert +(b_1+\frac{2}{\varepsilon N})\vert \overline z_{k-n} \vert \right)\frac{1}{\omega_k^q},
\]
\[
  \tilde R^{[m]}_{n_y}(\overline U_0) 
  \bydef  \frac{1}{2}\sum_{k=m}^{m+n-1}\left((a_1+\frac{1}{\varepsilon N})\vert \overline x_{k-n} \vert +(3a_1+b_1+\frac{1}{\varepsilon N})\vert \overline y_{k-n} \vert +(b_1+\frac{2}{\varepsilon N})\vert \overline z_{k-n} \vert \right)\frac{1}{\omega_k^q}, 
\]
\[
 \tilde R^{[m]}_{n_z}(\overline U_0) 
  \bydef   \frac{1}{2}\sum_{k=m}^{m+n-1}\left(b_2\vert \overline x_{k-n} \vert +b_2\vert \overline y_{k-n} \vert +(2b_2+2a_2)\vert \overline z_{k-n} \vert \right)\frac{1}{\omega_k^q}. 
\]
Notice that $\tilde R^{[m]}(\overline U_0) $ can be computed in a finite number of operations. According to (\ref{DT}) and (\ref{I-DF}), we have that for all $ V,V' \in B(0,r)$,
\begin{eqnarray*}
\left\vert DT_s( \overline U_s + V)(V') \right\vert &\leq & \left\vert \left( I^{[m]}-J^{(m)}DF_0^{[m]}(\overline U_0^{[m]} )\right)\left(V'^{[m]}\right) \right\vert + s\left\vert J^{[m]}
\begin{pmatrix}
\left(\Delta \dot U \right)^T \\
0
\end{pmatrix} 
V'^{[m]} \right\vert \\
& + & \left\vert J^{(m)}R^{[m]}(\overline U_0,V')\right\vert + \left\vert J \right\vert \left\vert D^2F_s(\overline U_0 )(V) (V') + D^2F_s(\overline U_0 )(s\Delta U) (V')\right\vert.
\end{eqnarray*}
Using expression (\ref{d2f}) and Lemma~\ref{n<m}, we get that for all $ V,V'\in B(0,r) $ and $n<2m-1$,
\begin{equation}
\label{majd2f}
\left\vert \left[D^2F_s(\overline U_0)(V)(V')\right]_n \right\vert \leq \frac{1}{2} \left(C_n^q(K)+\epsilon_n^q(K)\right)
\begin{pmatrix}
\lambda_1\\
\lambda_1\\
\lambda_2
\end{pmatrix}r^2 + \frac{2(\pi n)^2 r^2}{\rho \omega_n^q}
\begin{pmatrix}
1\\
1\\
1
\end{pmatrix},
\end{equation}
with $\displaystyle{\lambda_1=4a_1+2b_1+\frac{4}{\varepsilon N}}$ and $\lambda_2=4b_2+2a_2$. Let us set
\[
\widetilde{DQ}^q_n(K) \bydef 
\frac{1}{2} \left(C_n^q(K)+\epsilon_n^q(K)\right)
\begin{pmatrix}
\lambda_1\\
\lambda_1\\
\lambda_2
\end{pmatrix}.
\]

For $\left\vert \left[D^2F_s(\overline U_0)(s\Delta\overline U)(V')\right]_n \right\vert$, we know explicitly  $\Delta\overline U$ which allows us to compute sharper bounds. Still using (\ref{d2f}), we get that for all $ V'\in B(0,r) $, $s \in [0,1]$ and $n<2m-1$,
\begin{equation}
\label{majd2fbis}
\left\vert \left[D^2F_s(\overline U_0)(s\Delta\overline U)(V')\right]_n \right\vert \leq s\begin{pmatrix} \left[\theta_1(\Delta\overline u)\ast w^q\right]_n \\  \left[\theta_2(\Delta\overline u)\ast w^q\right]_n \\  \left[\theta_3(\Delta\overline u)\ast w^q\right]_n \end{pmatrix}r 
+ s(\pi n)^2 \left( \frac{\vert \Delta\overline d\vert}{\omega_n^q} \begin{pmatrix} 1\\ 1\\1\end{pmatrix}r
+ \frac{ \left[\left\vert \Delta\overline u \right\vert\right]_n }{\rho}r \right),
\end{equation}
where 
\[
\theta_1(u) \bydef (3a_1+b_1+\frac{1}{\varepsilon N}) \vert x \vert + (a_1+\frac{1}{\varepsilon N}) \vert y \vert + (b_1+\frac{2}{\varepsilon N}) \vert z \vert,
\]
\[
\theta_2(u) \bydef (a_1+\frac{1}{\varepsilon N}) \vert x \vert + (3a_1+b_1+\frac{1}{\varepsilon N}) \vert y \vert + (b_1+\frac{2}{\varepsilon N}) \vert z \vert,
\]
\[
\theta_3(u) \bydef b_2 \vert x \vert + b_2 \vert y \vert + (2a_2+2b_2) \vert z \vert
\]
and
\[
w^q \bydef \left(\frac{1}{\omega_0^q},\dots,\frac{1}{\omega_n^q},\dots\right).
\]
Let us also set
\[
\Theta_n^q(u) \bydef \begin{pmatrix} \left[\theta_1(u)\ast w^q\right]_n \\  \left[\theta_2(u)\ast w^q\right]_n \\  \left[\theta_3(u)\ast w^q\right]_n \end{pmatrix}.
\]
Observe that since $\Delta\overline u$ has only a finite number of non-zero coefficients, $\Theta_n^q(\Delta\overline u)$ can be computed in a finite number of operations.

Using all the bounds obtained in this section and the definition of $Z$ in (\ref{ineq1}), we can define $Z_d$ and the $m$ first $Z_{n}$ by
\begin{eqnarray*}
Z^{[m]} &=& \left\vert I^{[m]} - J^{[m]}DF_0^{[m]}(\overline U_0 ^{[m]}) \right\vert \left(W_1^q\right)^{[m]}r + s\left\vert J^{[m]}
\begin{pmatrix}
\left(\Delta \dot U \right)^T \\
0
\end{pmatrix} 
\right\vert \left(W_1^q\right)^{[m]}r + \vert J^{[m]} \vert \tilde R^{[m]}(\overline U_0)r \\ 
& + & \vert J^{[m]} \vert \left( \left(\widetilde{DQ}^q\right)^{[m]}(K)r^2 + \frac{2}{\rho} \left(W_2^q\right)^{[m]} r^2 +  s\left(\Theta^q\right)^{[m]}(\Delta\overline u) r + s\vert \Delta\overline d\vert \left(W_2^q\right)^{[m]} r + \frac{s}{\rho}\Lambda^{[m]}(\Delta \overline u) r\right) ,
\end{eqnarray*}
where
\[
\left(W_1^q\right)^{[m]}=\left(\frac{1}{\rho},\frac{1}{\omega_0^q},\frac{1}{\omega_0^q},\frac{1}{\omega_0^q},\dots ,\frac{1}{\omega_{m-1}^q},\frac{1}{\omega_{m-1}^q},\frac{1}{\omega_{m-1}^q}\right)^T,
\]
\[
\left(W_2^q\right)^{[m]}=\left(0,\frac{\pi ^2 0^2}{\omega_0^q},\frac{\pi ^2 0^2}{\omega_0^q},\frac{\pi ^2 0^2}{\omega_0^q},\dots ,\frac{\pi ^2 (m-1)^2}{\omega_{m-1}^q},\frac{\pi ^2 (m-1)^2}{\omega_{m-1}^q},\frac{\pi ^2 (m-1)^2}{\omega_{m-1}^q}\right)^T,
\]
\[
\left(\widetilde{DQ}^q\right)^{[m]}(K)=
\begin{pmatrix} 
0 \\ \widetilde{DQ}_0^q(K) \\ \vdots \\ \widetilde{DQ}_{m-1}^q(K)
\end{pmatrix}, \
\\\left(\Theta^q\right)^{[m]}(u)=
\begin{pmatrix}
0 \\ \Theta_0^q(u) \\ \vdots \\ \Theta_{m-1}^q(u)
\end{pmatrix} \mbox{ and }
\Lambda^{[m]}(u)=
\begin{pmatrix}
 0 \\ (\pi0)^2 \vert u_0 \vert \\ \vdots \\ (\pi(m-1))^2 \vert u_{m-1} \vert
\end{pmatrix}.
\]

\subsubsection{\underline{Case \boldmath $m\leq n < M$ \unboldmath}}
\label{Zn}

Let $m\leq n < 2m-1$. According to (\ref{df}) we have that
\[
\left[DF_s(\overline U_0)(V')\right]_n= \underbrace{D_dL_n(\overline U_0)(d')}_{=0} + D_uL_n(\overline U_0)(v') + DQ_n(\overline U_0)(V')
\]
and by definition of $J_n$ in (\ref{Jn}), $J_n\left(D_uL_n(\overline U_0)(v)\right)=v_n$, simplifying (\ref{DT}) into
\begin{eqnarray*}
\left[DT_s( \overline U_s + V)(V')\right]_n &=& \left[\left(I-J \left( DF_s(\overline  U_0) \right) \right)(V')\right]_n - \left[JD^2F_s( \overline U_0)(s\Delta\overline U+V)(V')\right]_n\\
&=& - J_nDQ_n(\overline U_0)(V') - \left[J\left(D^2F_s(\overline U_0)(s\Delta\overline U +V,V')\right)\right]_n .
\end{eqnarray*}
Now using (\ref{d2f}) we get that
\[
\left[DT_s( \overline U_s + V)(V')\right]_n = - J_n\left(DQ_n(\overline U_0 + s\Delta\overline U +V)(V') - \left( \pi n \right)^2 \left((s\Delta\overline d + d)v'_n+d'(s\underbrace{\Delta\overline u_n}_{=0} + v_n)\right) \right).
\]

Using the same bounds as for (\ref{majd2f}) and (\ref{majd2fbis}) we can set
\[
Z_n = \left\vert J_n \right\vert \left(\Theta_n^q(\overline u)r+s\Theta_n^q(\Delta\overline u)r +\widetilde{DQ}_n^q(K) r^2
+ \frac{(\pi n)^2}{\omega_n^q} \left( s\vert\Delta\overline d \vert r + \frac{2 r^2}{\rho} \right)
\begin{pmatrix}1\\1\\1\end{pmatrix}\right).
\]

\subsubsection{\underline{Case \boldmath $n \geq m$ \unboldmath}} \label{sec:Zn_n_ge_M}

We still have 
\[
\left[DT_s( \overline U_s + V)(V')\right]_n = - J_n\left(DQ_n(\overline U_0 + s\Delta\overline U +V)(V') - \left( \pi n \right)^2 \left((s\Delta\overline d + d)v'_n+d' v_n\right) \right),
\]
but here we bound the convolution products in $DQ_n$ using Lemma~\ref{n>m} to get, for all $ n \geq M$,
\[
\left\vert \left[DT_s( \overline U_s + V)(V')\right]_n \right\vert \leq \left\vert J_n \right\vert \left(\frac{1}{2}\frac{\alpha_M^q(K)}{\omega_n^q} \left(\left\Vert\overline u \right\Vert_q r+s\left\Vert\Delta\overline u\right\Vert r + r^2\right) \begin{pmatrix}\lambda_1\\\lambda_1\\\lambda_2\end{pmatrix}
+ \frac{(\pi n)^2}{\omega_n^q} \left( s\vert\Delta\overline d \vert r + \frac{2 r^2}{\rho} \right)
\begin{pmatrix}1\\1\\1\end{pmatrix}\right).
\]
Then we set 
\[
\hat Z_M = \left\vert J_M \right\vert \left(\frac{1}{2}\frac{\alpha_M^q(K)}{\omega_M^q} \left(\left\Vert\overline u \right\Vert_q r+s\left\Vert\Delta\overline u\right\Vert r + r^2\right) \begin{pmatrix}\lambda_1\\\lambda_1\\\lambda_2\end{pmatrix}
+ \frac{(\pi M)^2}{\omega_M^q} \left( s\vert\Delta\overline d \vert r + \frac{2 r^2}{\rho} \right)
\begin{pmatrix}1\\1\\1\end{pmatrix}\right),
\]
and since the terms of $\left\vert J_n \right\vert$ and $(\pi n)^2\left\vert J_n \right\vert$ are decreasing for $n\geq \max\left(\sqrt{\frac{r_1}{\pi^2d}},\sqrt{\frac{r_2}{\pi^2d}}\right)$ ($\approx 11$ with the values of the parameter taken here, $m$ is always taken such that $2m-1>11$), we can set, for all $ n\geq M$,
\begin{equation}
\label{ZM}
Z_n=\hat Z_M\frac{\omega_M^q}{\omega_n^q}.
\end{equation}

Notice that (\ref{ZM}) is what allows us to check the hypotheses of Theorem~\ref{th_local} with finite computations so it is really crucial for the proof, and we are able to do this thanks to the fact that the terms of $\left\vert J_n \right\vert$ are decreasing, which happens because the magnitude of the eigenvalues of the linear part of $f$ are growing in $(\pi n)^2$. In fact, we would have (\ref{ZM}) for every system whose equations can be written as the sum of a linear operator with eigenvalues of increasing magnitude and a non linear polynomial term (in particular for reaction-diffusion systems).

\subsection{Explicit computation of the radii polynomials}
\label{I}

Now according to Section~\ref{sec:def_pol} and using the bounds $Y$ and $Z$ we got in the two previous sections, we define the radii polynomials. Notice that $Y$ does not depend on $r$ and that $Z$ has linear and quadratic terms in $r$, so the radii polynomials are all of degree two.

\subsubsection{\underline{Case \boldmath $n<m$ \unboldmath}}
Let us set
\[
a^{[m]}(s)=\vert J^{[m]} \vert \left(\left(\widetilde{DQ}^q\right)^{[m]}(K) + \frac{2}{\rho} \left(W_2^q\right)^{[m]} \right),
\]
\begin{eqnarray}
\label{bm}
b^{[m]}(s) &=& \left\vert I^{[m]} - J^{[m]}DF_0^{[m]}(\overline U_0^{[m]}) \right\vert \left(W_1^q\right)^{[m]} + \vert J^{[m]} \vert \tilde  R^{[m]}(\overline U_0) + s\left\vert J^{[m]}
\begin{pmatrix}
\left(\Delta \dot U \right)^T \\
0
\end{pmatrix} 
\right\vert \left(W_1^q\right)^{[m]}  \nonumber \\
& +&  s\vert J^{[m]} \vert \left( \left(\Theta^q\right)^{[m]}(\Delta\overline u)  + \vert \Delta\overline d\vert \left(W_2^q\right)^{[m]}  + \frac{1}{\rho} \Lambda^{[m]}(\Delta \overline u) \right)- \left(W_1^q\right)^{[m]},
\end{eqnarray}
and
\[
c^{[m]}(s)=Y^{[m]}(s).
\]
Then we define
\[
P_d(r,s)=a^{[m]}_d(s)r^2 + b^{[m]}_d(s)r + c^{[m]}_d(s) 
\]
and for all $n<m$,
\[
P_n(r,s)=a^{[m]}_n(s)r^2 + b^{[m]}_n(s)r + c^{[m]}_n(s).
\]

\subsubsection{\underline{Case \boldmath $m\leq n< M$ \unboldmath}} \label{sec:rad_polyCase_m_le_n_le_ M}

For $m\leq n< 2m-1$,  define
\begin{equation}
\label{a}
a_n(s)=\left\vert J_n \right\vert \left( \widetilde{DQ}_n^q(K)+
\frac{2(\pi n)^2}{\rho \omega_n^q} 
\begin{pmatrix}
1\\
1\\
1
\end{pmatrix}\right),
\end{equation}
\begin{equation}
\label{b}
b_n(s)=\left\vert J_n \right\vert \left(\Theta_n^q(\overline u)+s\Theta_n^q(\Delta\overline u)\right)
+ \left(s(\pi n)^2\vert\Delta\overline d \vert  \left\vert J_n \right\vert
\begin{pmatrix}
1\\
1\\
1
\end{pmatrix}
- 
\begin{pmatrix}
1\\
1\\
1
\end{pmatrix}\right)\frac{1}{\omega_n^q},
\end{equation}
and
\[
c_n(s)=Y_n(s).
\]

Then for each $m\leq n<M$, we define
\[
 P_n(r,s)=a_n(s)r^2 + b_n(s)r + c_n(s).
\]

\subsubsection{\underline{Case \boldmath $n=M$ \unboldmath}}

We have $Y_M=0$ (from the choice of $M=2m-1$ as explained in Section~\ref{Y}) so $c_M=0$. Setting
\[
a_M(s)=\left\vert J_M \right\vert \left( \frac{1}{2}\frac{\alpha_M^q}{\omega_M^q} 
\begin{pmatrix}
c_1\\
c_1\\
c_2
\end{pmatrix} +
\frac{2(\pi M)^2}{\rho \omega_M^q} 
\begin{pmatrix}
1\\
1\\
1
\end{pmatrix}\right)
\]
and
\begin{equation}
\label{bqueue}
b_M(s)=\left(\frac{1}{2}\alpha_M^q \left( \left\Vert\overline u\right\Vert_q + s\left\Vert\Delta\overline u\right\Vert_q\right) \left\vert J_M \right\vert
\begin{pmatrix}
c_1\\
c_1\\
c_2
\end{pmatrix} 
+ s(\pi M)^2 \vert\Delta\overline d \vert \left\vert J_M \right\vert
\begin{pmatrix}
1\\
1\\
1
\end{pmatrix}
- 
\begin{pmatrix}
1\\
1\\
1
\end{pmatrix}\right)\frac{1}{\omega_M^q},
\end{equation}
we can define
\[
P_M(r,s)=a_M(s)r^2 + b_M(s)r=a_M(s)r\left(r+\frac{b_M(s)}{a_M(s)}\right).
\]

\subsubsection{Procedure to find (if possible) \boldmath $r>0$ \unboldmath satisfying (\ref{ineq4})}

To verify hypothesis (\ref{ineq4}) of Theorem~\ref{th_local}, we use Lemma~\ref{nb_fini}. More explicitly we look for $r>0$ such that, for all $ s\in [0,1]$, $P_d(r,s)<0$ and $P_n(r,s)<0$ for all $n\leq M$. Notice that the coefficients of the radii polynomials are increasing with $s$ so that it is equivalent to find $r>0$ such that 
\begin{equation}
\label{P(r)}
P_d(r,1)<0 \quad \mbox{and} \quad P_n(r,1)<0, ~~ \mbox{for all~} n\leq M.
\end{equation}
To find such $r$, we set
\begin{equation} \label{eq:I_n}
I_d \bydef \lbrace r>0 \ | \ P_d(r,1)<0 \rbrace \quad \mbox{and} \quad I_n \bydef \lbrace r>0 \ | \ P_n(r,1)<0 \rbrace, ~~ \mbox{for all~} n\leq M.
\end{equation}
Then we determine numerically an approximation of 
\begin{equation} \label{eq:I}
I \bydef I_d\cap \left(\bigcap_{n=0}^M I_n \right).
\end{equation}
If $I\neq \emptyset$, we choose $r \in I$, and check (\ref{P(r)}) rigorously, by computing the coefficients of the radii polynomials with $s=1$ using interval arithmetic. We see in Section~\ref{opt} why it is reasonable to hope that such $r$ exists, provided the parameters $\Delta_s$, $m$ and $q$ are chosen carefully. 

\section{Optimization of the parameters} \label{sec:parameter_opt}

\subsection{Optimal choice of the parameters  \boldmath  $m$ \unboldmath and \boldmath $\Delta_s$ \unboldmath}
\label{opt}

Recall that the parameter $m$ controls the dimension (which equals $3m$) of the finite dimensional Galerkin projection given by \eqref{f^m}, and from \eqref{eq:predictor}, the parameter $\Delta_s$ is used to define a predictor $\hat U_0 =  \overline U_0 + \Delta_s \dot U_0$ whose value serves as an initial point for Newton's method to get a corrector $\overline U_1$. The value of  $\Delta_s$ represents roughly the arc length of curve we are covering in one predictor-corrector step, hence the name {\em pseudo-arclength} continuation. Since $\Delta \overline U = \overline U_1 - \overline U_0$, if $\Delta_s$ is not so large, then its value should be close to the length $\| \Delta \overline U \|$ of the segment $[\overline U_0 , \overline U_1]$. Hence, studying $\Delta_s$ is roughly the same as studying the length of $\Delta \overline U$. 

The optimal strategy aims at maximizing the pseudo-arclength parameter $\Delta_s$ to prove existence of long pieces of solution curve in one predictor-corrector step (as explained in Section~\ref{sec:pred-cor}) while taking $m$ as small as possible in order to minimize the computational cost. However, the fact that we are looking for an $r$ verifying the hypotheses of Theorem~\ref{th_local} (which is equivalent to find $r>0$ for which every radii polynomial is negative) leads to some constraints.

Let us now remark that for a quadratic polynomial of the form $P(r)=ar^2+br+c$, if $a>0$, $b<0$ and $c>0$ is small enough \big(precisely $c<\frac{b^2}{4a}$\big), then there is an interval $[r_{min},r_{max}] \subset (0,\infty)$ such that for all $ r \in [r_{min},r_{max}]$, $P(r)<0$. Based on this fact, let us study in details the coefficients of each radii polynomial to see how to  choose $m$ and $\Delta_s$ optimally. Since we showed in Section~\ref{I} that we could bound the polynomials letting $s=1$, we always set from now $s=1$.

\subsubsection{\underline{Case \boldmath $n<m$ \unboldmath}}

Recall the definition of the coefficients $a^{[m]}$, $b^{[m]}$ and $c^{[m]}$ of  the radii polynomials for the case $n<m$. First notice that each component of $a^{[m]}$ and $c^{[m]}$ is positive. In the definition of $b^{[m]}$, the first two terms are very small (the first by definition of $J$, and the second provided $m$ is not too small, which will always be the case), and the next two can be made as small as needed (and so $b^{[m]}$ will be negative) by taking $\Delta_s$ small enough. $c^{[m]}$ can also be made very small by taking $\Delta_s$ small (see Section~\ref{Ym}). Hence we can expect each set $I_n$ (for $n < m$) defined in Section~\ref{I} to be non empty if $\Delta_s$ is small enough (the same is true for $I_d$).

\subsubsection{\underline{Case \boldmath $m \leq n < M$ \unboldmath}} 

Recall the definition in Section~\ref{sec:rad_polyCase_m_le_n_le_ M} of the coefficients $a_n$, $b_n$ and $c_n$ of  the radii polynomial $P_n(r)$ for the case $m \leq n < M$. According to the previous section, $\Delta_s$ should not be too large and hence here again $c_n$ should be small. Also, the predominant term of $b_n$ in (\ref{b}) is $\left\vert J_n \right\vert \Theta_n^q(\overline u)$ which decreases to $0$ as $n$ grows, so taking $m$ large enough should allow us to have $b_n$ negative for $n\geq m$ and thus, recalling \eqref{eq:I_n}, we can expect that $I_n \neq \emptyset$ (for $m \leq n < M$).

\subsubsection{\underline{Case \boldmath $n=M$ \unboldmath}}

The situation in the case $n=M=2m-1$ is the same as above except that the expression of $b_M$ in (\ref{bqueue}) is different, with $m$ large enough we can expect $I_M$ to be non empty. 

\subsubsection{Algorithm to choose \boldmath $\Delta_s$ \unboldmath  and \boldmath $m$ \unboldmath optimally} \label{sec:algorithm}
 
Let us now present an algorithm to chose $\Delta_s$  and $m$ optimally. 
Given $\overline U_0$, $\dot U_0$, $\Delta_s $ and $m$,
\begin{enumerate}
\item \label{alg:step1} Compute $\overline U_1$, $\dot U_1$ and $I$ given by \eqref{eq:I};
\item If for some $n<m$ $I_n = \emptyset$, take $\Delta_s$ smaller and go back to Step~\ref{alg:step1};
\item If for some $m\leq n \leq M$ $I_n = \emptyset$, take $m$ larger and go back to Step~\ref{alg:step1};
\item If $I=\emptyset$, take $\Delta_s$ smaller and go back to Step~\ref{alg:step1};
\item Start a new predictor-corrector step from $\overline U_1$ and $\dot U_1$ with $\Delta_s$ larger and $m$ smaller.
\end{enumerate}

The fact that we try to increase $\Delta_s$ and decrease $m$ after each successful step is not optimal. We indeed observed numerically that the process often failed. In fact we noticed that the value of $\max I$ has to reach some threshold before $\Delta_s$ could be increased successfully. Similarly $d$ has to reach some other threshold before $m$ could be decreased successfully. In practice we only try to increase $\Delta_s $ or decrease $m$ if those thresholds are reached. 

Note that we want to change the value of $m$ along the process while conserving the smoothness property of the global curve. The important fact is to have exactly the same function $F_s$ given by \eqref{F_s} at the end of one piece of curve and at the start of the next, that is $F^{(0)}_1=F^{(1)}_0$ with the notations of Section~\ref{sec:recollement},  and this even when we change the value of $m$ between the two. Let us denote $m^{(0)}$ (resp. $m^{(1)}$) the value of $m$ for the first curve (resp. the second). In $F_s$, only $E_s$ changes with $m$. 
Hence, we only need to check that $E^{(0)}_1=E^{(1)}_0$. Remember that $\overline U_s$ and $\dot U_s$ are constructed from finite dimensional vectors and completed with zeros. In fact, $E^{(0)}_1(U) = \left(U^{\left[m^{(0)}\right]}-\overline U_1^{\left[m^{(0)}\right]}\right)\cdot \dot U_1^{\left[m^{(0)}\right]}$ and $E^{(1)}_0(U) = \left(U^{\left[m^{(1)}\right]}-\overline U_1^{\left[m^{(1)}\right]}\right)\cdot \dot U_1^{\left[m^{(1)}\right]}$. If we increase $m$, that is $m^{(1)}>m^{(0)}$, the last $m^{(1)}-m^{(0)}$ frequencies of $\dot U_1^{\left[m^{(1)}\right]}$  \Big (that is $ \left[\dot U_1\right]_n$, for $n\in \{m^{(0)}+1,\dots,m^{(1)}\}$\Big) are zeros and hence $E^{(0)}_1$ is equal to $E^{(1)}_0$. Decreasing $m$, that is $m^{(1)}<m^{(0)}$, requires a bit of care since the last frequencies of $\dot U_1^{\left[m^{(0)}\right]}$ are not necessarily zero. Hence we have to make an intermediate step. Let us denote $\overline U_0^{\left[m^{(0)}\right]}$ the last point we have and $\dot U_0^{\left[m^{(0)}\right]}$ a tangent vector at this point. From there, we make a new predictor-corrector step with $m$ equal to $m^{(0)}$, get a point $\overline U_1^{\left[m^{(0)}\right]}$ and then compute $\dot U_1^{\left[m^{(0)}\right]}$. Before doing the proof of this portion of curve, we set the $m^{(0)}-m^{(1)}$ frequencies of $\overline U_1^{\left[m^{(0)}\right]}$ and $\dot U_1^{\left[m^{(0)}\right]}$ to zero. Then we are able to decrease $m$ at the next step while having exactly the same $F_s$ at the connecting point. 

\subsection{Optimal choice of the decay rate parameter \boldmath  $q$ \unboldmath } \label{sec:q}

In light of Lemma~\ref{lemma:equivalence}, computing steady state of \eqref{eq:mimura_system} with Neumann boundary conditions is equivalent to find $U \in \Omega_{q}$ such that $f(U)=0$, and that for any fixed $q >1$. Hence, it seems legitimate to investigate which decay rate $q$ is optimal. The weight $\omega_n^q$ defined in \eqref{e:weights} depends on $q$ and it influences the value of the norm $\left\Vert\cdot\right\Vert_q$ and the theoretical bounds of Lemma~\ref{n>m}. Since the value of $q$ has a major impact on the radii polynomials, it influences strongly the values of the parameters $\Delta_s$ and $m$ used for the proof, and therefore it influences  the computational time required to perform the proofs. As one can see in Table~\ref{fig:computational_cost}, depending on the value of $q$, the computational costs required to prove the red branches of Figure~\ref{A1} and Figure~\ref{B1} can change drastically. In Figure~\ref{fig:Ds_m_as_functions_of_q}, one can see how $q$ influences the values of $\Delta_s$ and $m$ while performing the algorithm of Section~\ref{sec:algorithm} along the first bifurcation branch presented in red in Figure \ref{A1}, and along the first part of the second bifurcation branch presented in red in Figure \ref{B1}. 
\begin{table}[h]
\begin{center}
{\small
\begin{tabular}{|c|c|c|c|}
\hline $q$ & Red branch of Figure~\ref{A1} & Red branch of Figure~\ref{B1}  \\ 
\hline $1.2$ & $515$ & $528$ \\ 
\hline $1.5$ & $321$ & $393$ \\ 
\hline $2$ & $276$ & $408$ \\ 
\hline $3$ & $460$ & $2256$ \\ 
\hline 
\end{tabular}
} 
\caption{\small Computational cost (in seconds), as a function of the decay rate parameter $q$, required to compute the radii polynomials and find an $r>0$ at which they are all simultaneously negative. These computations were done without interval arithmetic.}
\label{fig:computational_cost}
\end{center}
\end{table}
\begin{figure}[h]
\begin{center}
\begin{minipage}[c]{0.4 \linewidth}\centering
\subfigure[Values of $\Delta_s$ along the first branch]{\includegraphics[width=6cm]{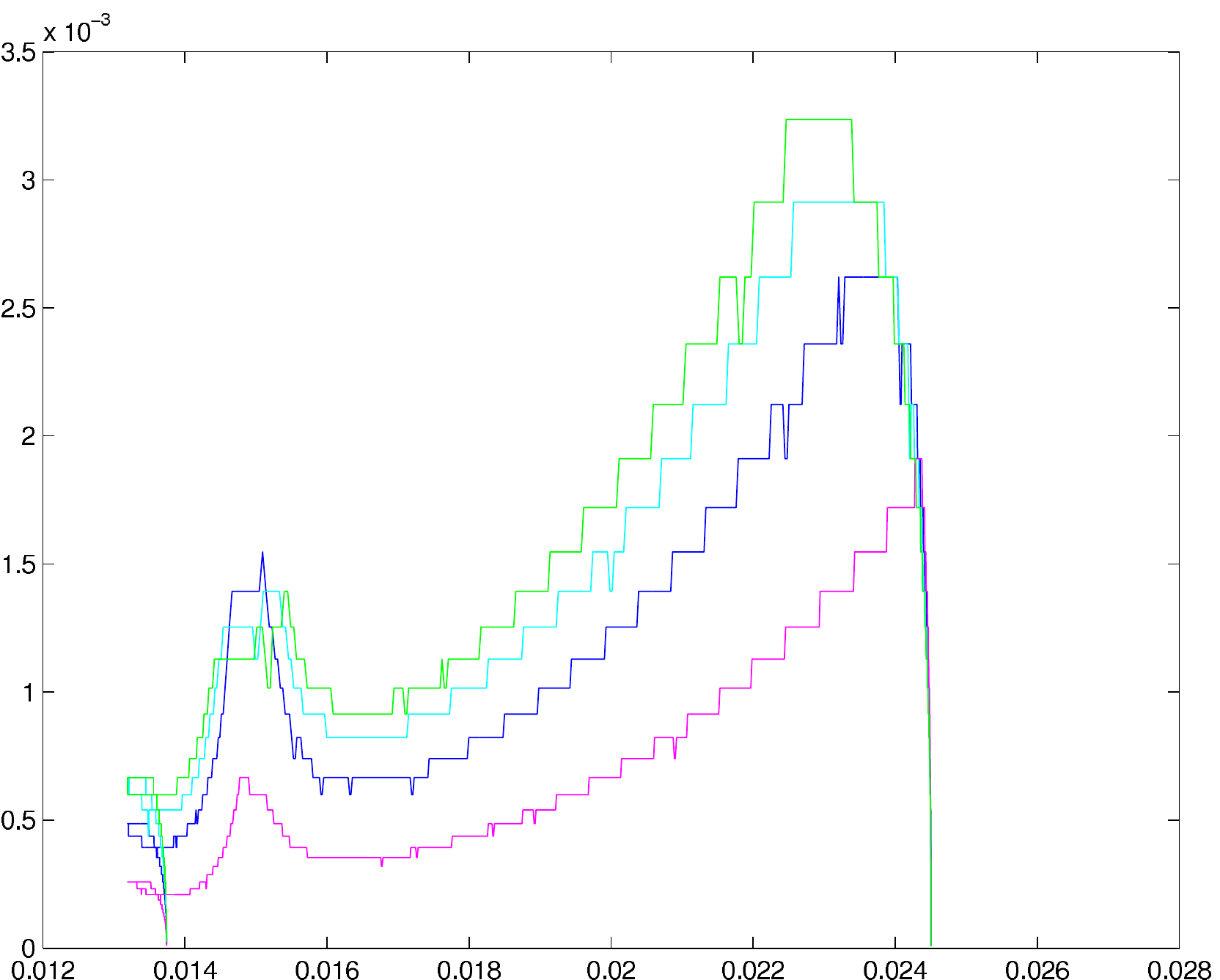}} \\
\subfigure[Values of $m$ along the first branch.]{\includegraphics[width=6cm]{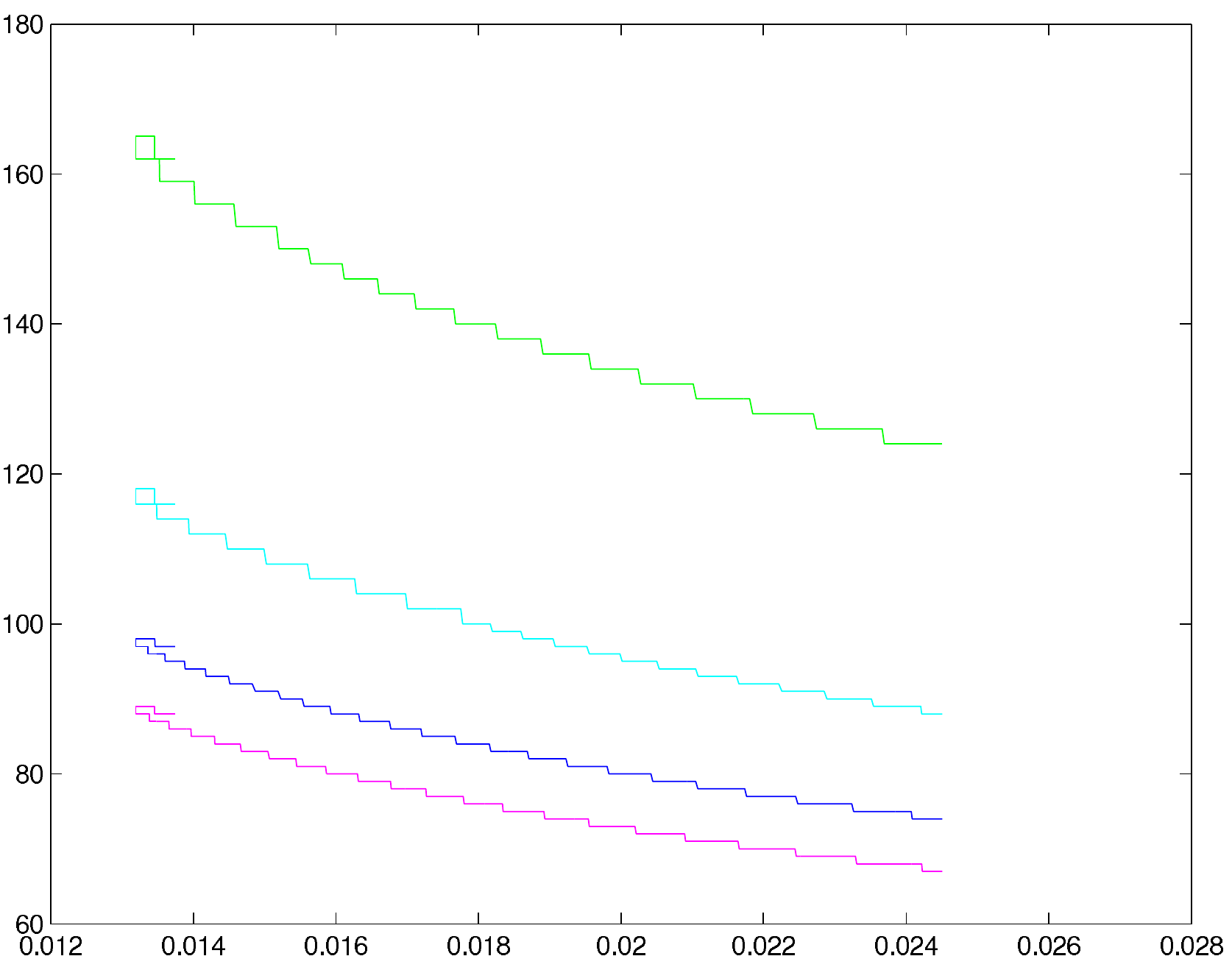}} 
\end{minipage}
\begin{minipage}[c]{0.4 \linewidth}\centering
\subfigure[Values of $\Delta_s$ along the second branch.]{\includegraphics[width=6cm]{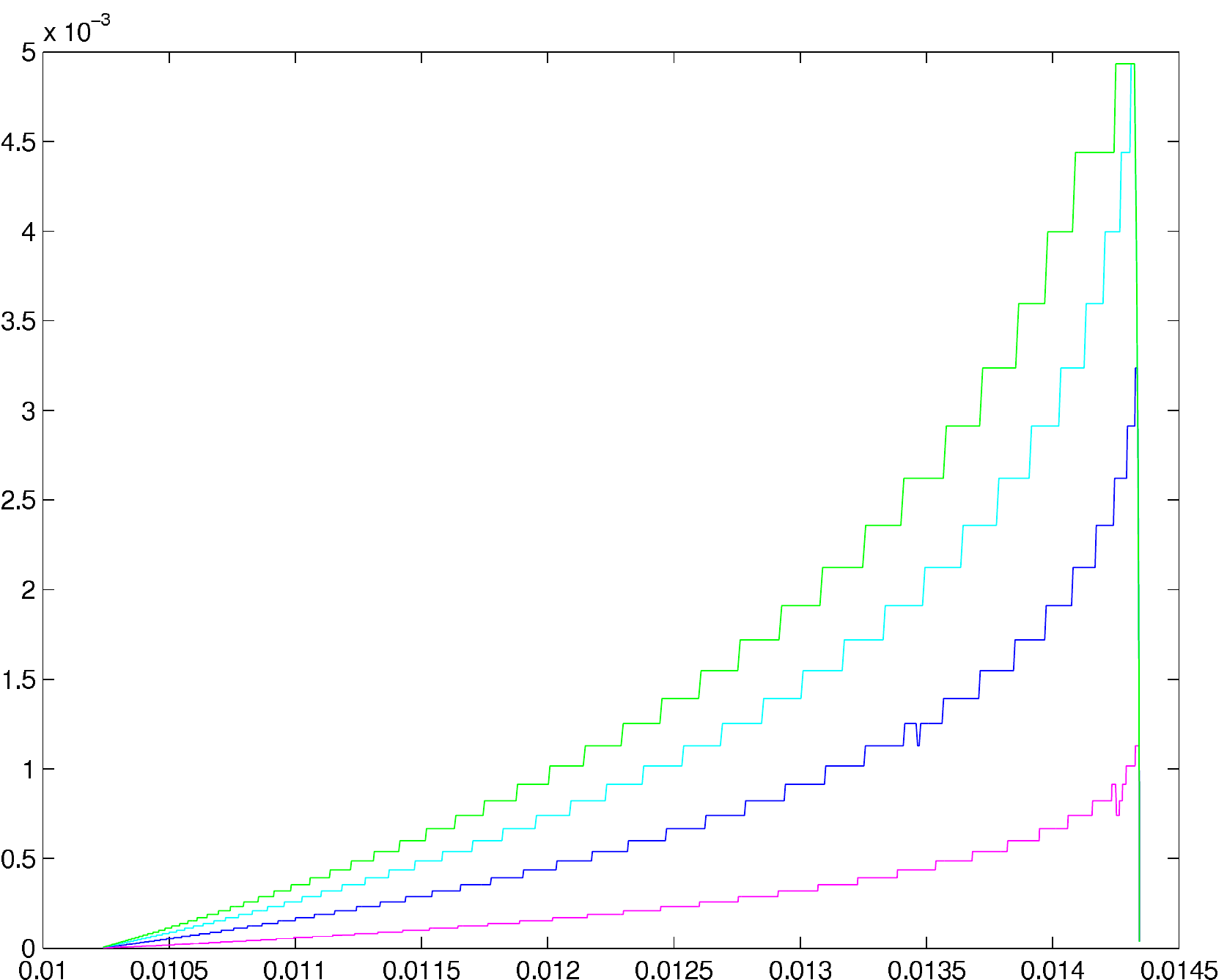}} \\
\subfigure[\label{compB1} Values of $m$ along the second branch]{\includegraphics[width=6cm]{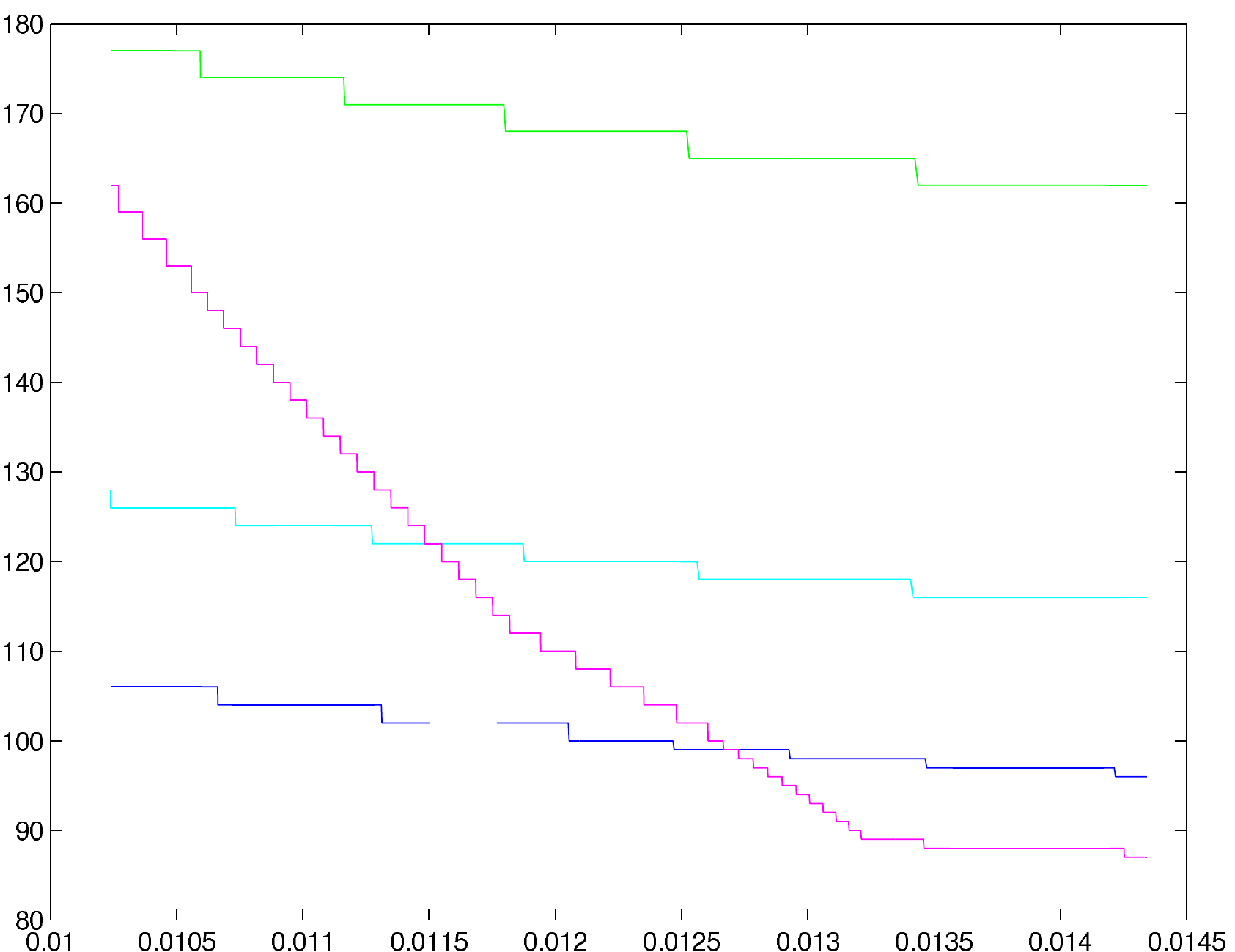}} 
\end{minipage}
\vspace{-.3cm}
\caption{\small Influence of the decay rate parameter $q$ on the parameters $\Delta_s$ and $m$, accordingly to the Algorithm of Section~\ref{sec:algorithm}. {\bf On the left:} results for the branch in red of Figure~\ref{A1}. {\bf On the right:} results for the branch in red of Figure~\ref{B1}. Green corresponds to $q=1.2$, cyan corresponds to $q=1.5$, blue corresponds to $q=2$ and magenta corresponds to $q=3$.}
\label{fig:Ds_m_as_functions_of_q}
\end{center}
\end{figure}
What we see with these comparisons is that, on the red branch of Figure~\ref{A1}, taking $q$ smaller allows using a greater $\Delta_s$ but at the expense of a greater $m$. Hence, a better compromise seems to be somewhere between $q=1.5$ and $q=2$. But it seems that when $d$ becomes small, it could also be better to take $q$ smaller in order to use a smaller $m$ (see the case $q=3$ on Figure~\ref{compB1}). Let us see how this evolves when $d$ becomes even smaller. Note that we were able to get the rigorous data (for different values of $q$) of Table~\ref{fig:computational_cost} and of Figure~\ref{fig:Ds_m_as_functions_of_q} because the rigorous computations of the concerned branches did not take too long. However for some other curves, the proof takes much longer, mainly because $d$ gets smaller and hence $m$ must be taken larger for the coefficients $b_M$ of the last radii polynomial (\ref{bqueue}) to be negative. On the other hand, it only takes one point along the branch to determine, given $q$, which value of $m$ should be taken in order to make $b_M$ negative. Also, the cost of the proof decreases linearly with $\Delta_s$ but increases quadratically with $m$, so it seems fair to chose $q$ according to $m$ only and independently of $\Delta_s$ for those branches where $m$ has to be taken large. In Figure~\ref{complet}, we present some non rigorous results concerning the second part of the second branch. Although non rigorous, these results helps understanding the role played by $q$ for the computational cost of the method. We can see in Figure~\ref{q_opt1} that as the diffusion parameter $d$ decreases, the optimal decay rate $q$ becomes smaller as well. That can be explained by the fact that the coefficient $b_M$ of the last radii polynomial is non negative if $M=2m-1$ is taken too small. A necessary condition for $b_M$ to be negative is roughly the following
\[
\frac{\alpha_M^q \left\Vert u \right\Vert_q}{d(\pi M)^2} < \mbox{some constant},
\]
which shows why $M$ has to be taken larger when $d$ becomes smaller. But to understand the impact of $q$, we have to detail how the values of $\alpha_M^q$ and $\left\Vert u \right\Vert_q$ evolve with $q$. For this let us recall the definitions of $\left\Vert \cdot \right\Vert_q$ in (\ref{norme}) and of $\alpha_M^q$ in (\ref{def_alpha}) and let us look at the shape of the solution $u$. For instance, we see on Figure~\ref{spatial1} that on the first branch the solutions are almost sinusoids of period $2$ and hence $\left\Vert u \right\Vert_q \approx \max\left(\left\vert u_0 \right\vert_{\infty} \omega_0^q,\left\vert u_1 \right\vert_{\infty} \omega_1^q\right)=\max\left(\left\vert u_0 \right\vert_{\infty} ,\left\vert u_1 \right\vert_{\infty} \right)$. Hence, taking $q$ smaller does not decrease the value of $\left\Vert \cdot \right\Vert_q$, but it increases slightly $\alpha_M^q$, and then finally leads to a greater $m$. However, 
when $d$ becomes smaller, the high frequencies in the solutions are more and more important. Thus $\left\Vert u \right\Vert_q$ involves $\left\vert u_n \right\vert_{\infty} \omega_n^q$ for some $n>1$ and hence taking $q$ smaller decreases significantly the value of $\left\Vert \cdot \right\Vert_q$ and therefore the value of $\alpha_M^q \left\Vert u \right\Vert_q/d\pi^2$. This allows taking a smaller $m$, provided that $q$ is not too small. Indeed, $\alpha_M^q\underset{q\rightarrow 1}{\sim} \frac{4}{q-1}$ increases more than $\left\Vert u \right\Vert_q$ decreases when $q$ gets close to $1$, which explains why the optimal value of $q$ seems not to go below $1.3$.  One can see in Figure~\ref{q_opt1} that when $d$ is really small, the value of $m$ is far much smaller with $q=1.2$ than with $q=2$. This is a significant improvement, especially in terms of the computational cost which evolves in $m^3$. Actually, we did some computation, and the cost of the proof to do part of a branch for small $d$ is about $10$ times faster with $q=1.3$ than with $q=2$. This brings us to the conclusion that the new estimates introduced in Section~\ref{Estimes} for $q<2$ can be quite useful. Actually, they allowed proving parts of some branches in Figure~\ref{complet} which we would not have been able to do with $q=2$. 

\begin{figure}[h]
\begin{center}
\hspace{-.5cm}
\begin{minipage}[c]{0.4 \linewidth}\centering
\subfigure[\label{complet}]{\includegraphics[width=8.5cm]{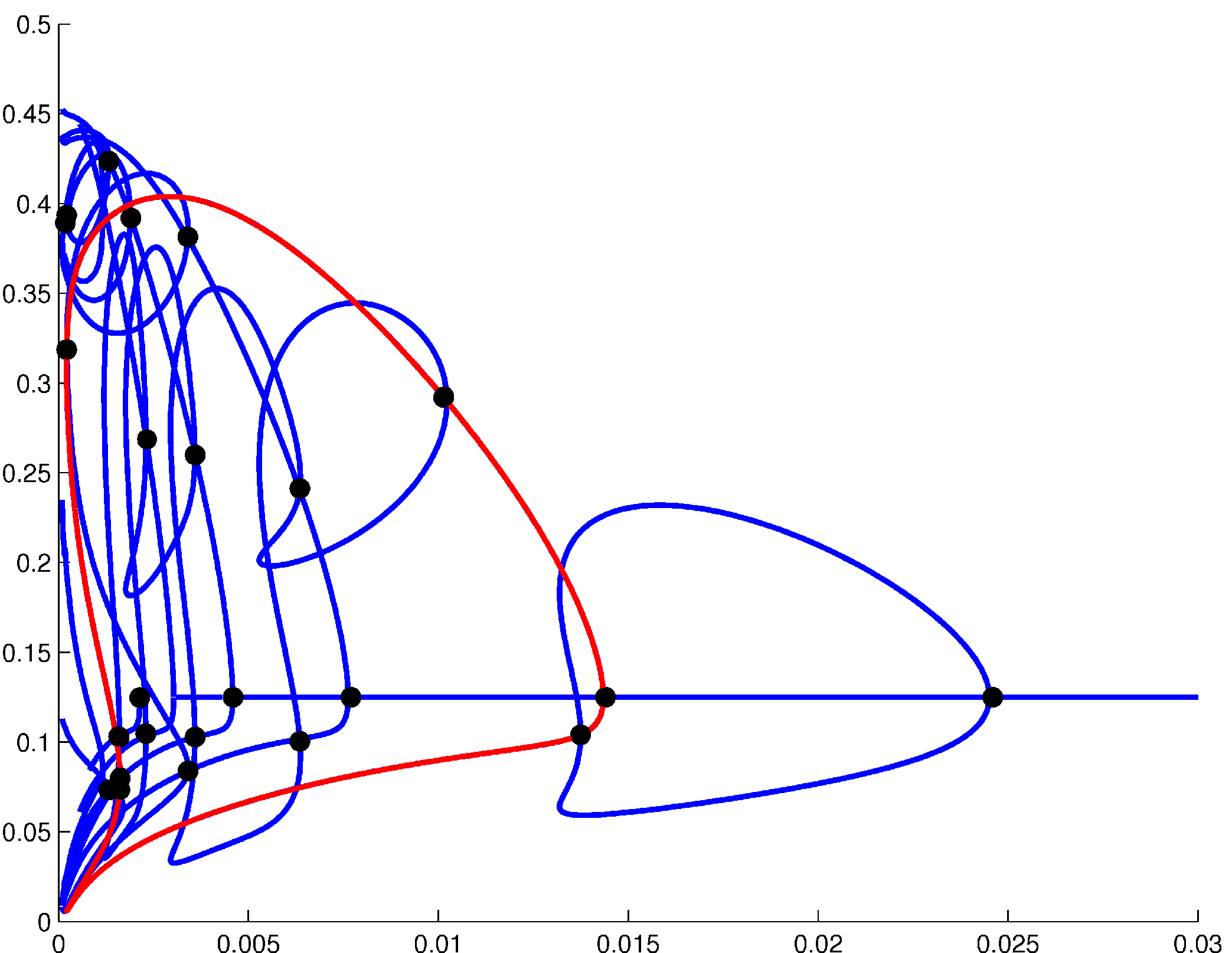}} 
\end{minipage}
\hspace{2.5cm}
\begin{minipage}[c]{0.4 \linewidth}\centering
\subfigure[]{\includegraphics[width=6.5cm]{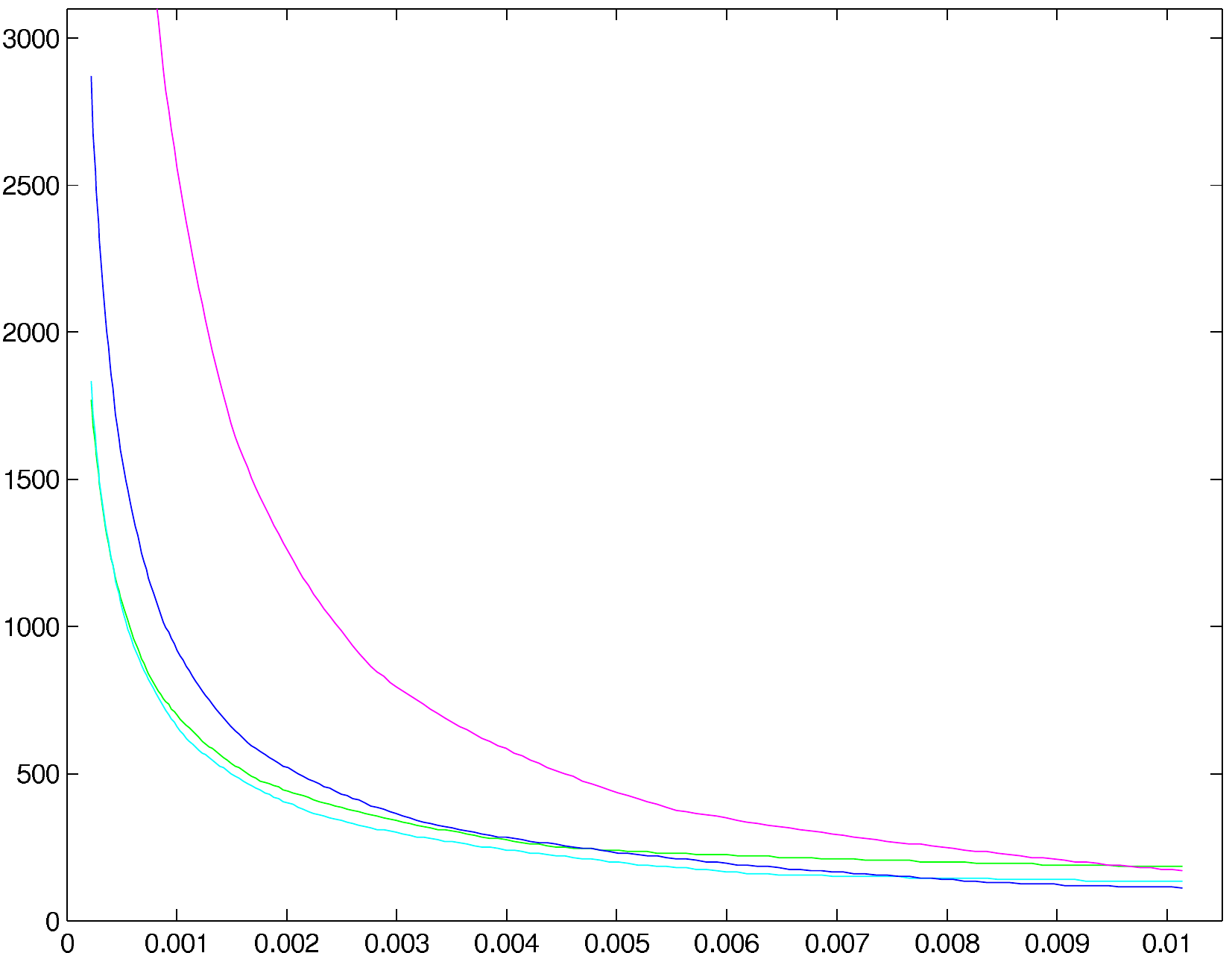}} 
\subfigure[\label{q_opt1}]{\includegraphics[width=6.5cm]{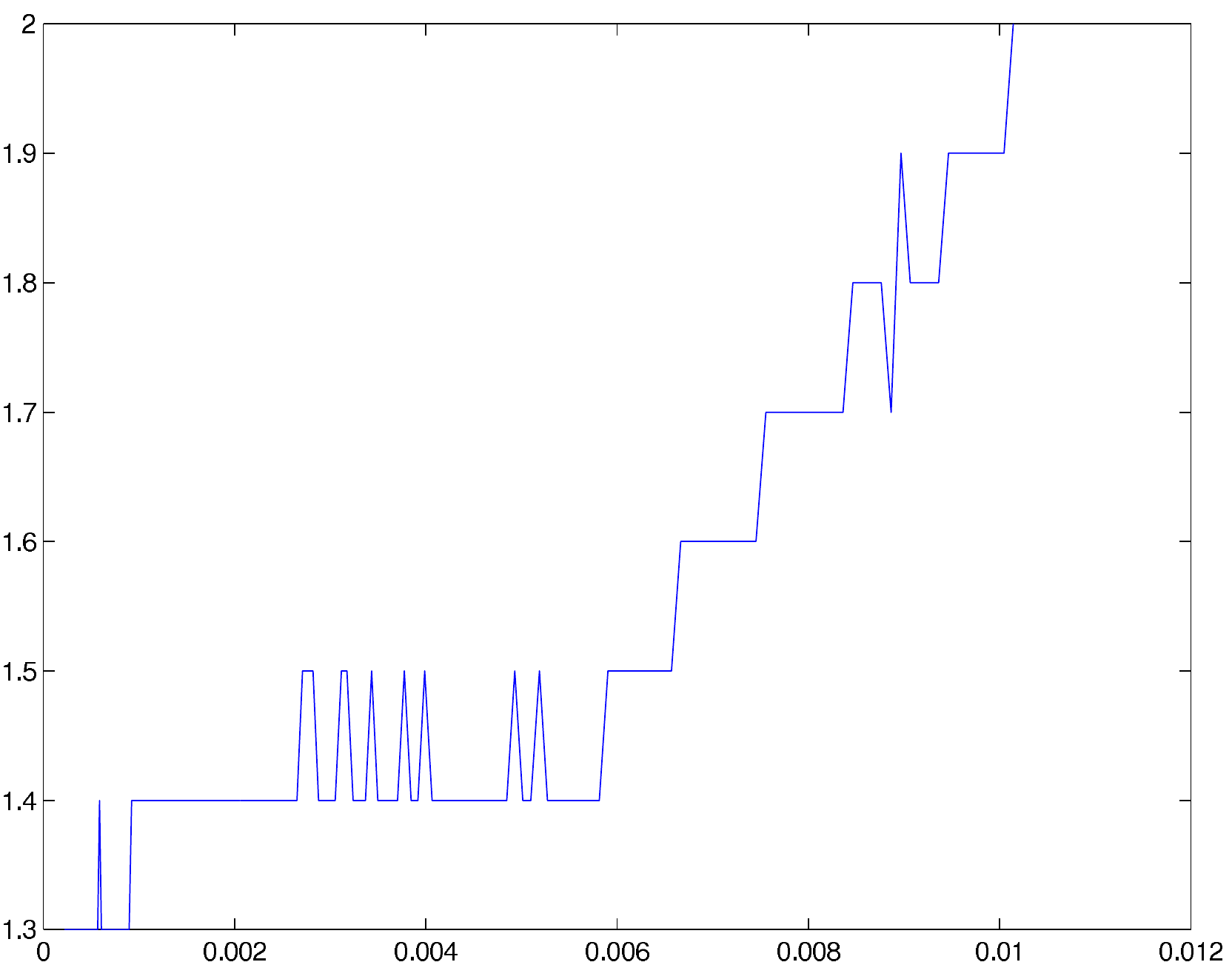}} 
\end{minipage}
\vspace{-.3cm}
\caption{\small (a) A more complete (non rigorous) bifurcation diagram of steady states of \eqref{eq:mimura_system}. (b) Values of $m$ along the branch for different $q$. Green corresponds to $q=1.2$, cyan corresponds to $q=1.5$, blue corresponds to $q=2$ and magenta corresponds to $q=3$. (c) Best values of $q$ to use along the branch as a function of $d$ in order to get $m$ as small as possible.}
\label{espace}
\end{center}
\end{figure}

\begin{figure}[h]
\hspace{-.1cm}
\begin{minipage}[c]{0.4 \linewidth}\centering
\subfigure[The chosen steady states]{\includegraphics[width=8cm]{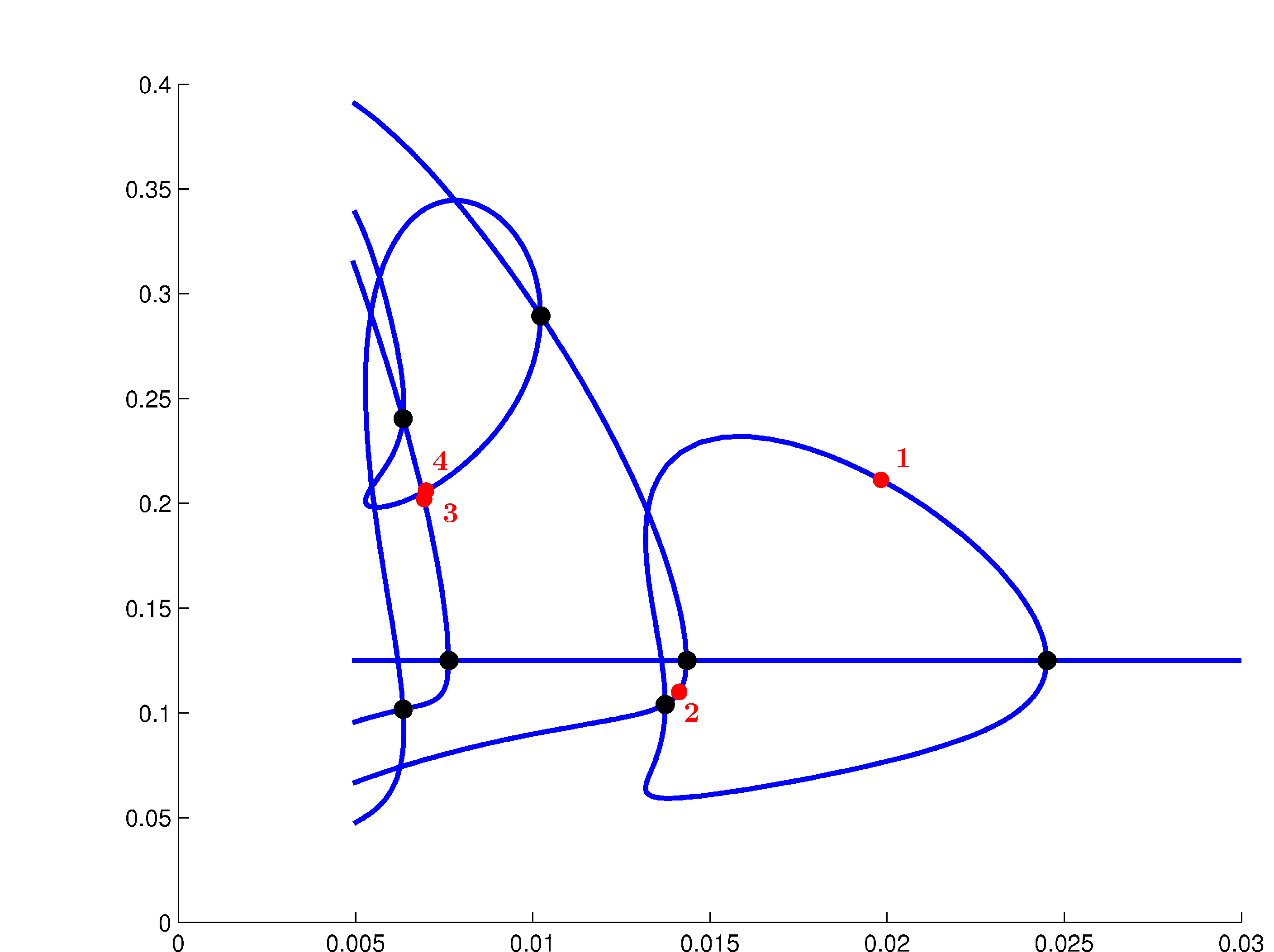}} 
\end{minipage}
\hspace{-.05cm}
\begin{minipage}[c]{0.4 \linewidth}\centering
\subfigure[\label{spatial1} Solution 1]{\includegraphics[width=3cm]{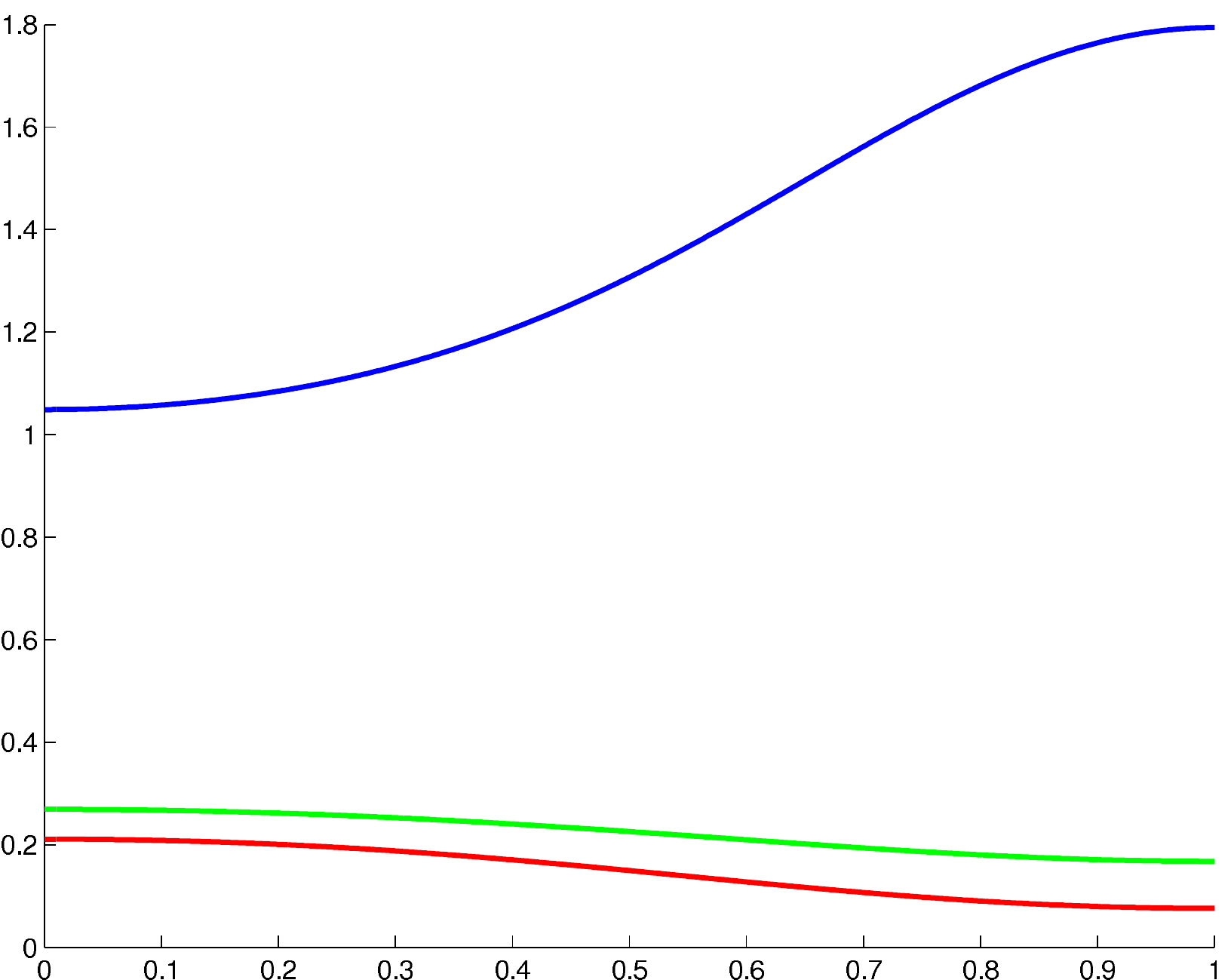}} \\
\subfigure[\label{spatial2} Solution 2]{\includegraphics[width=3cm]{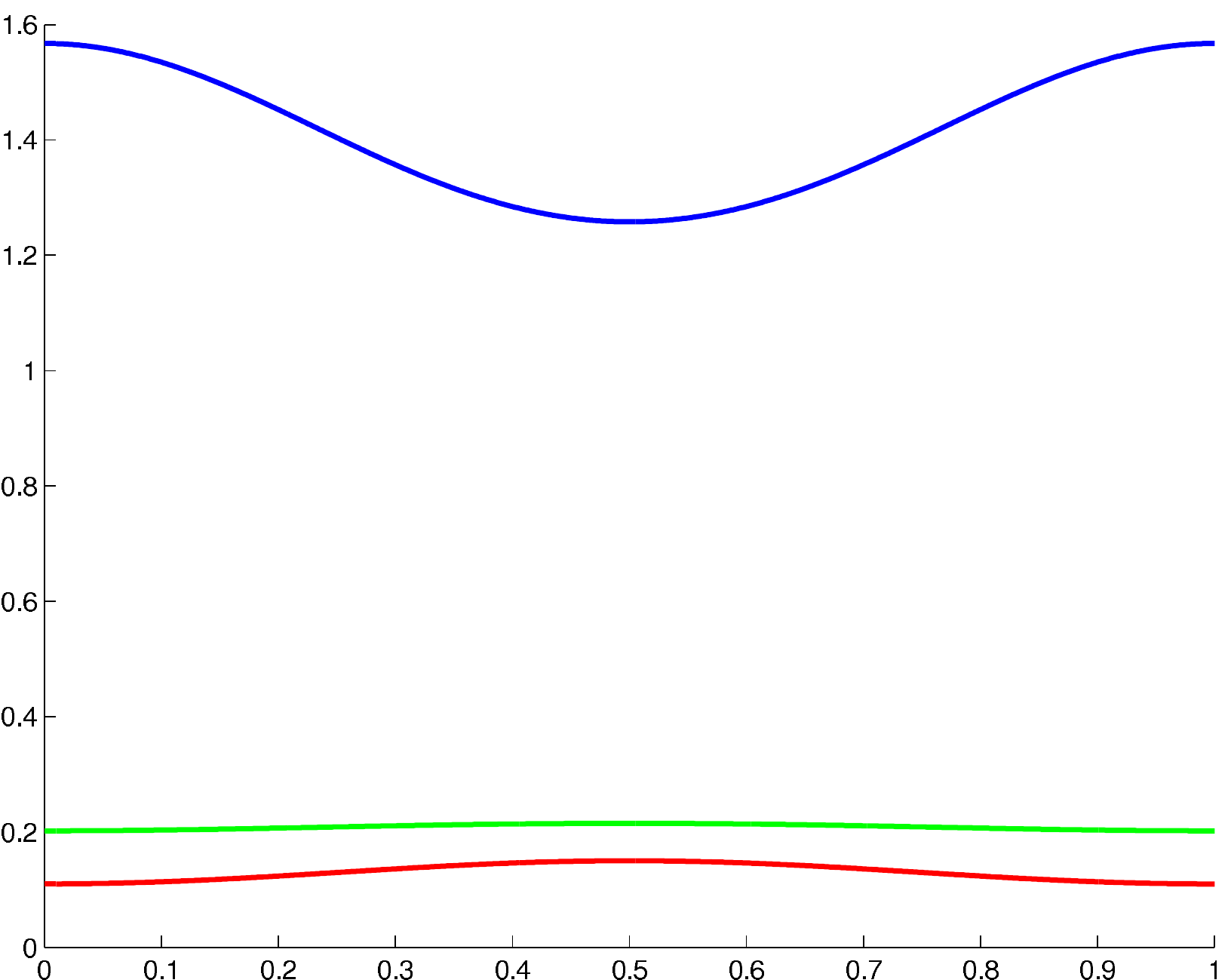}} 
\end{minipage}
\hspace{-3.5cm}
\begin{minipage}[c]{0.4 \linewidth}\centering
\subfigure[\label{spatial3} Solution 3]{\includegraphics[width=3cm]{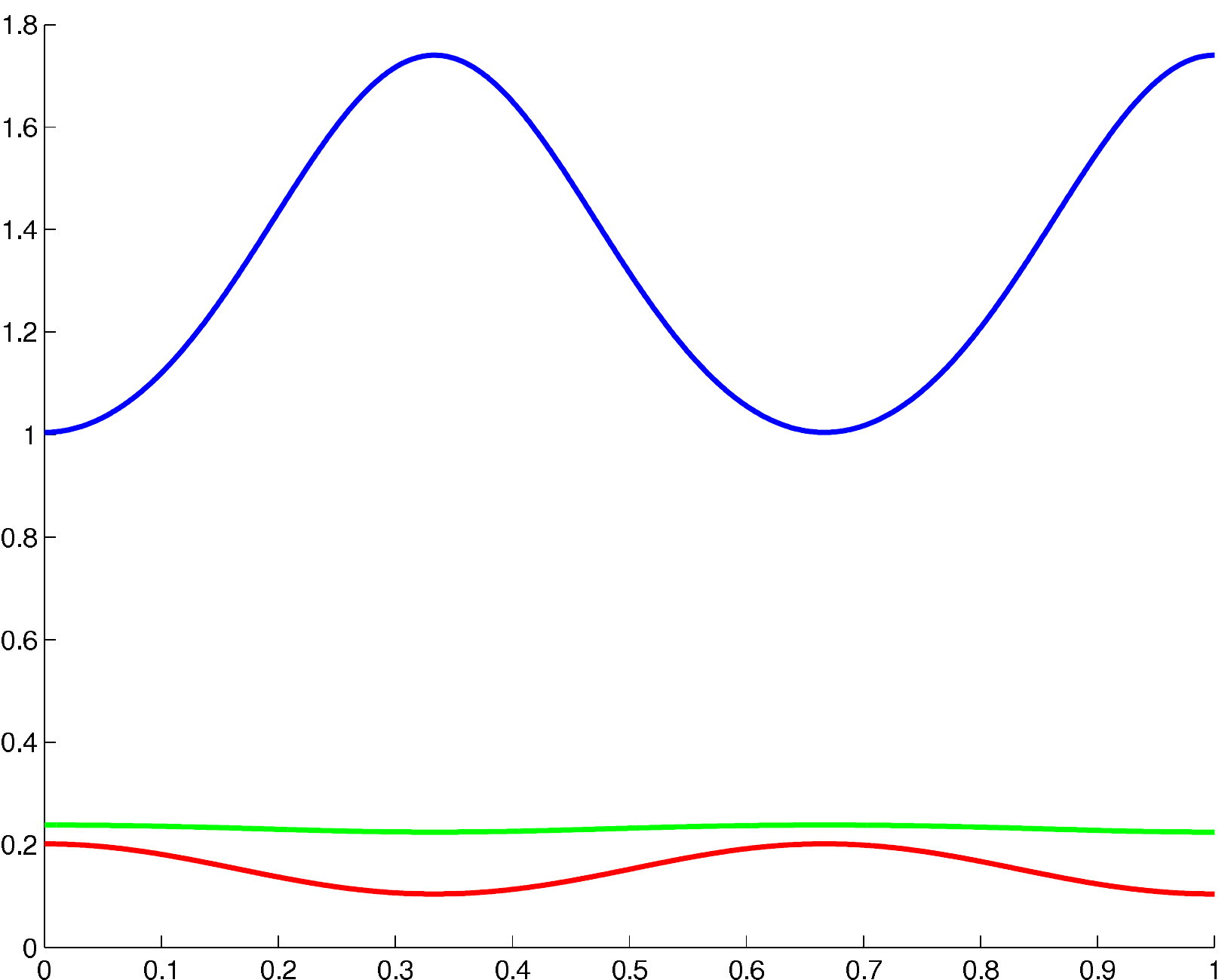}} \\
\subfigure[\label{spatial4} Solution 4]{\includegraphics[width=3cm]{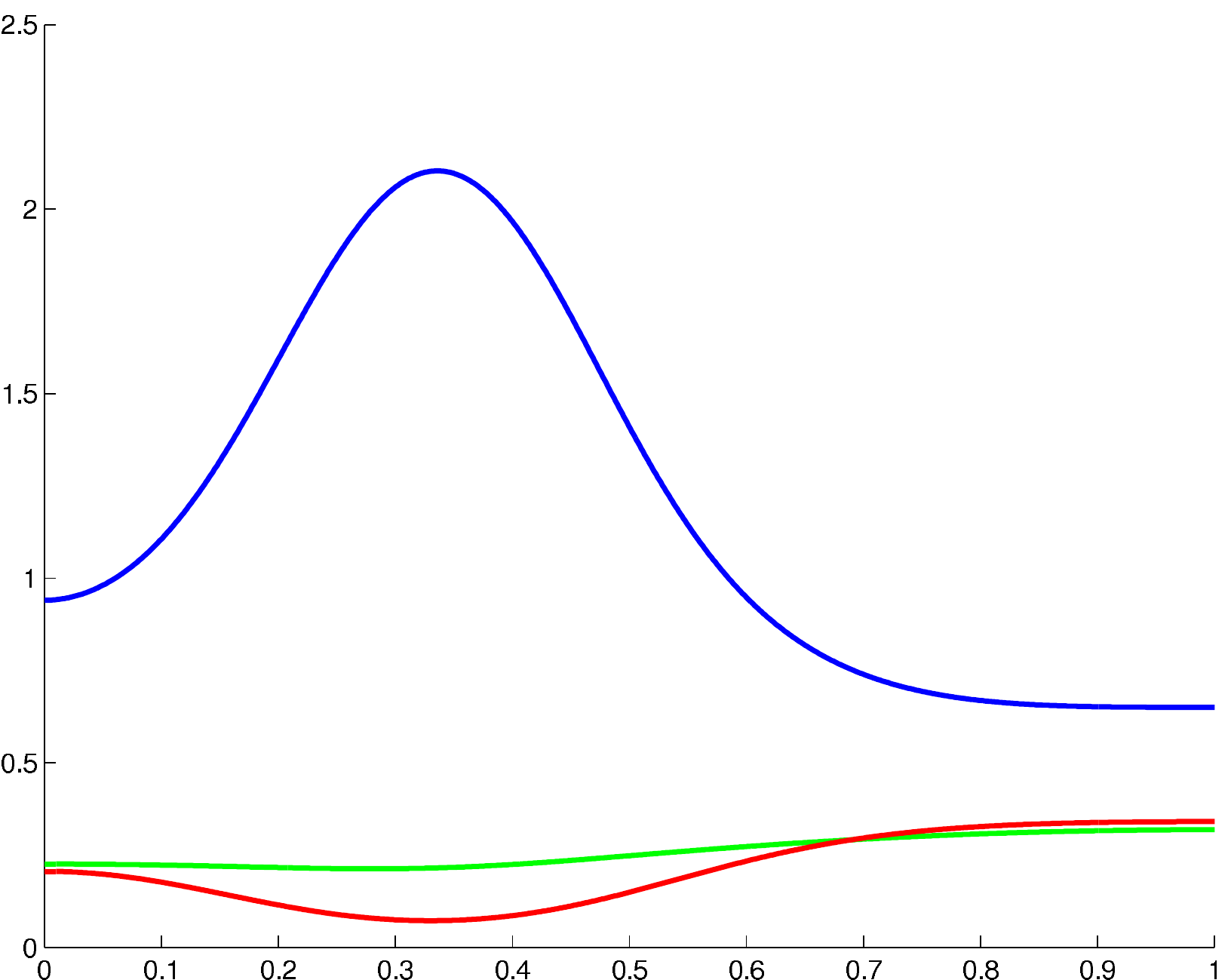}} 
\end{minipage}
\vspace{-.4cm}
\caption{\small Spatial representation of some steady states of Theorem~\ref{thm:bif_diagram} with $x$ blue, $y$ green and $z$ red.}
\label{espace1}
\end{figure}

\begin{figure}[h]
\begin{center}
\begin{minipage}[c]{0.4 \linewidth}\centering
\subfigure[\label{A1}]{\includegraphics[width=6cm]{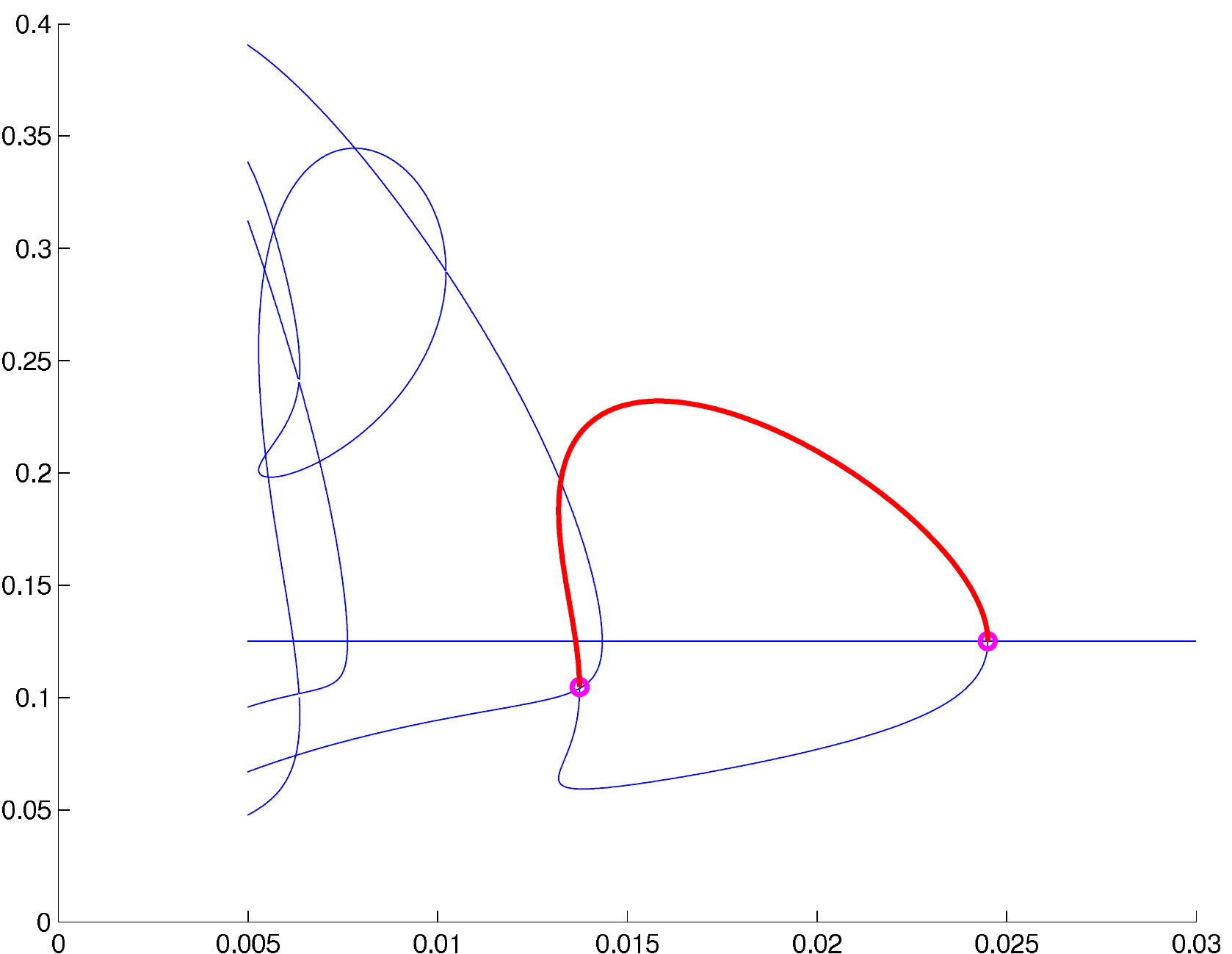}} \\
\vspace{-.4cm}
\subfigure[]{\includegraphics[width=5cm]{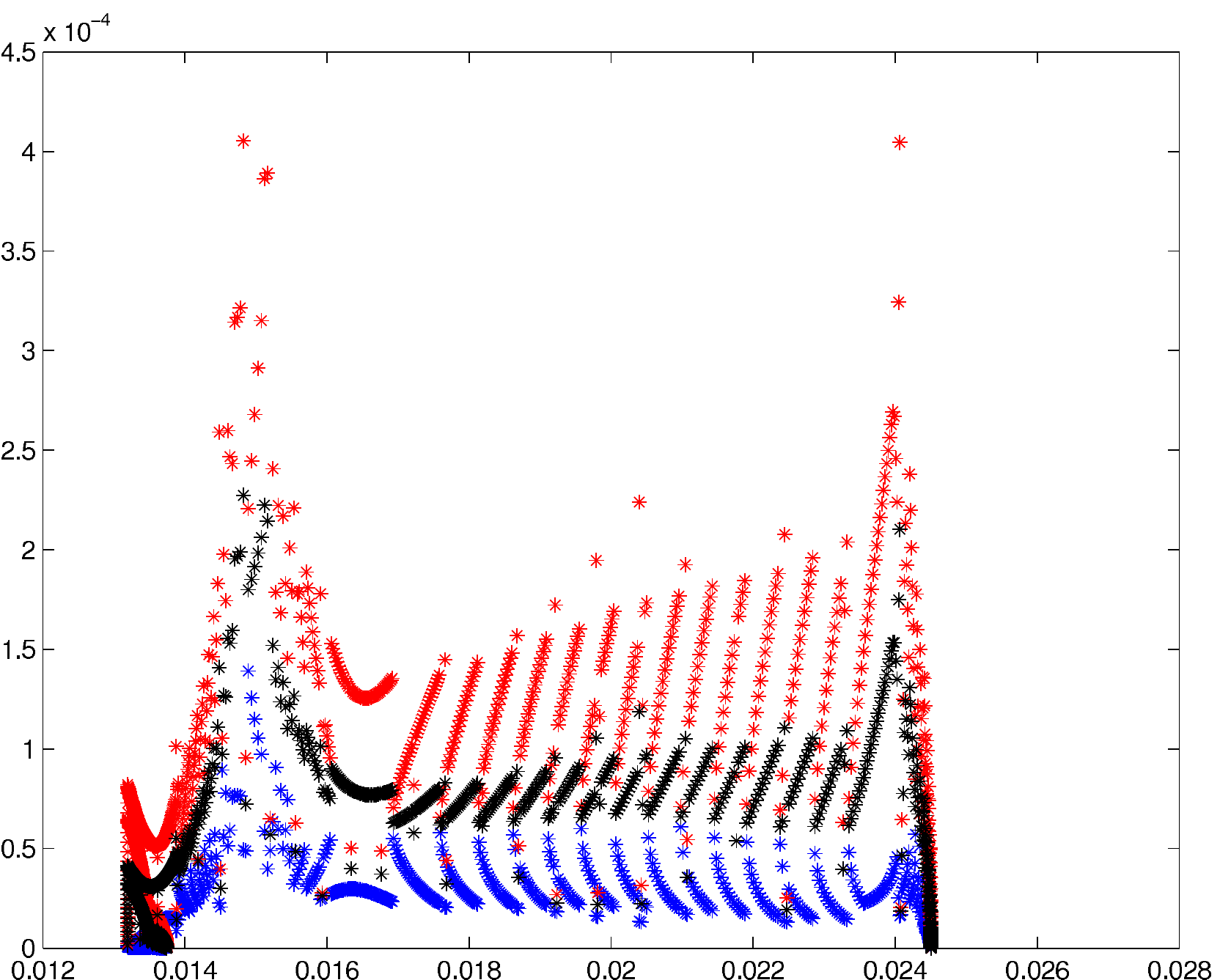}} 
\end{minipage}
\begin{minipage}[c]{0.4 \linewidth}\centering
\subfigure[\label{B1}]{\includegraphics[width=6cm]{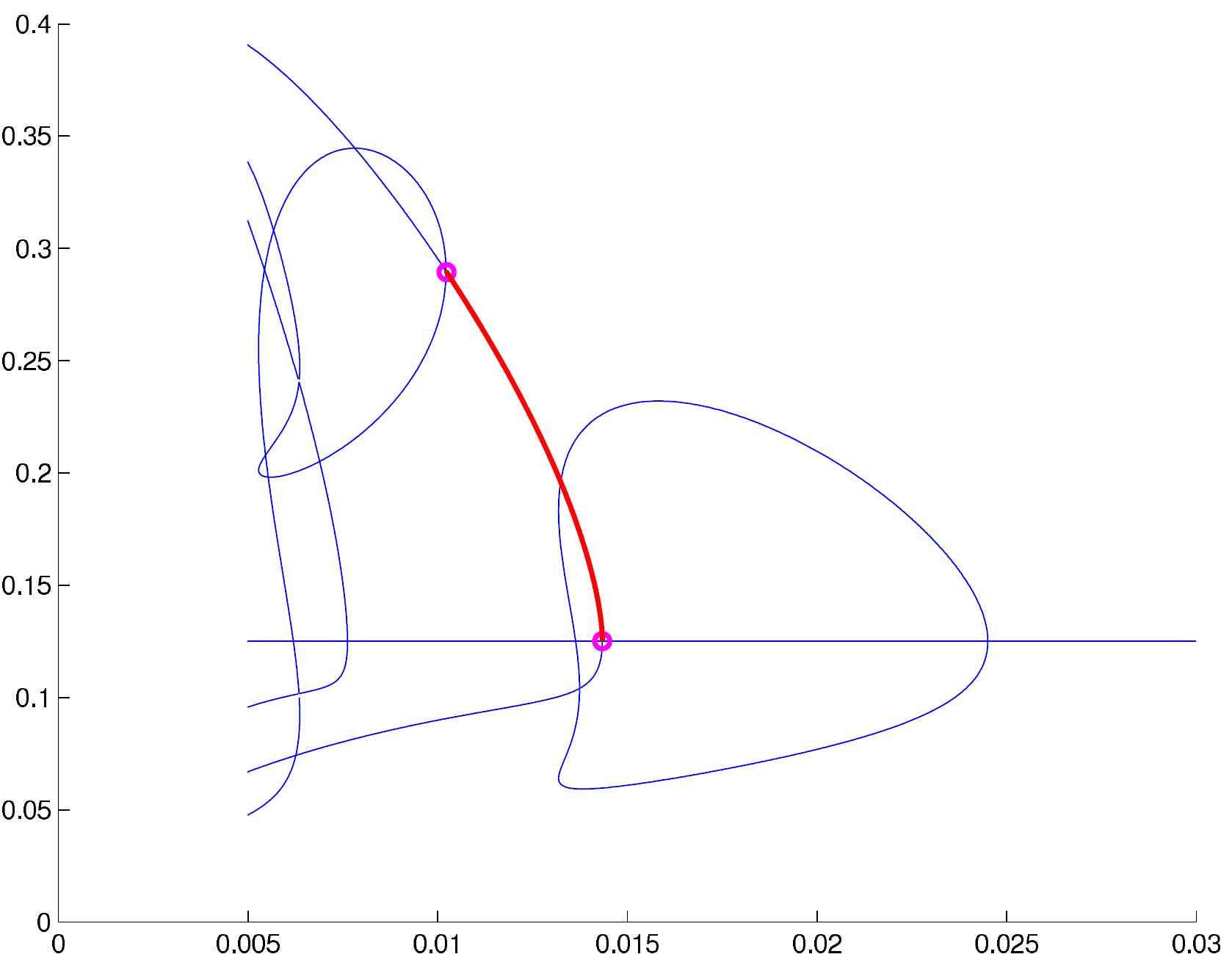}} \\
\vspace{-.4cm}
\subfigure[]{\includegraphics[width=5cm]{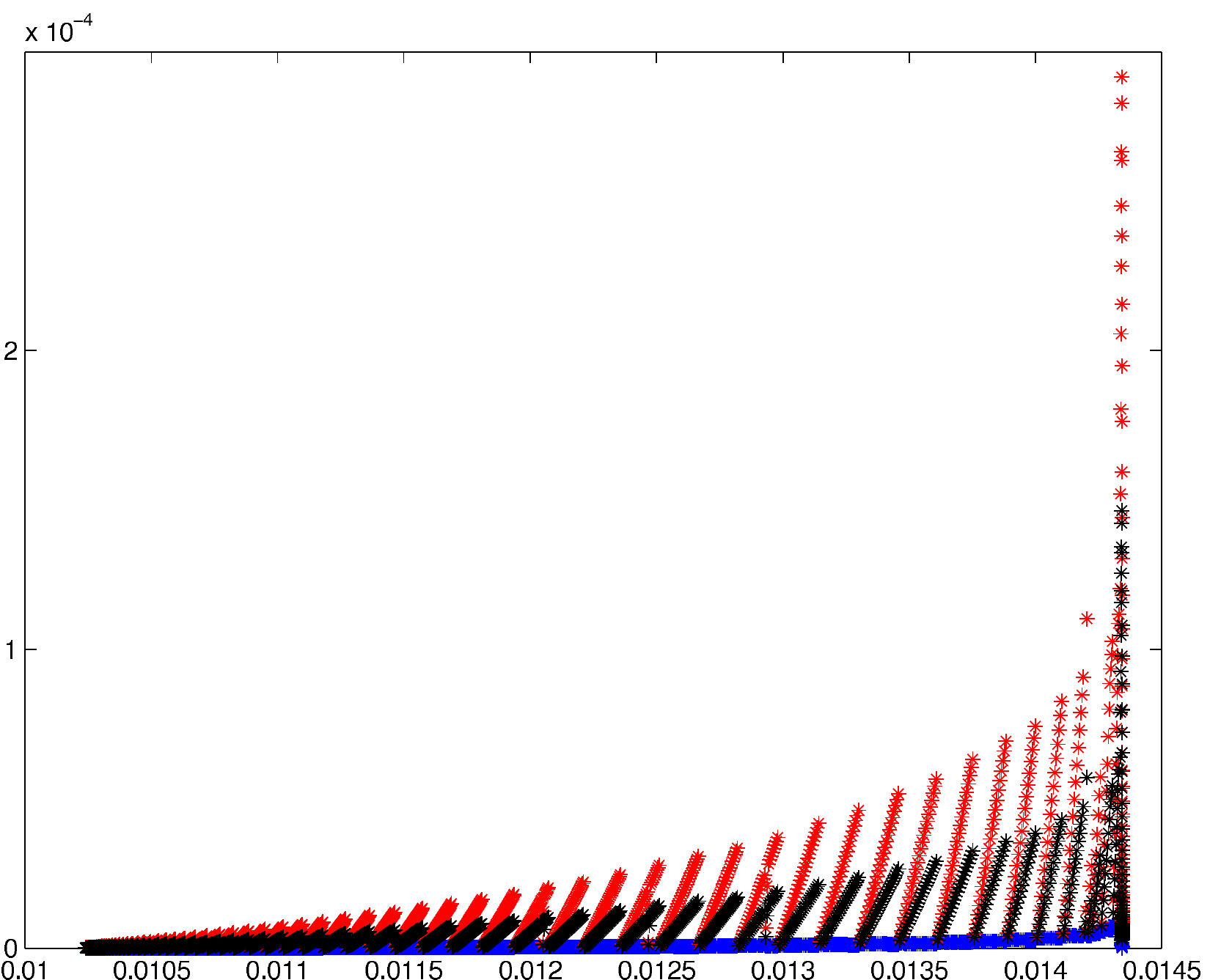}} 
\end{minipage}
\vspace{-.4cm}
\caption{\small Proofs of the branches in red with $q=2$. {\bf On the left:} The radius of the {\em tubes} in which we proved the existence of a unique portion of curve are between $2.1e^{-4}$ and $4.7e^{-8}$. {\bf On the right:} The radius of the {\em tubes} in which we proved the existence of a unique portion of curve are between $2.2e^{-4}$ and $1.7e^{-8}$. For parts (b) and (d), the values of $\min(I)$ are in blue and the values of $\max(I)$ are in red along the branch. In black the values of $r$ used for the rigorous verification using interval arithmetic.}
\end{center}
\end{figure}

\section{Proofs of two main results of Section~\ref{sec:introduction} } \label{sec:proofs}

\subsection{Proof of Theorem~\ref{thm:bif_diagram}}

Recall \eqref{eq:def_ball} and fix $\rho=10$. Set $K=10^4$ in the bounds of Lemma~\ref{n>m} and set $K=3M$ in those of Lemma~\ref{n<m}. Set the multiplicative coefficient used to change $\Delta_s$ equal to $\frac{10}{9}$ and the one used to change $m$ equal to $1.02$. That means that in the algorithm of Section~\ref{sec:algorithm}, taking $\Delta_s$ larger means setting $\Delta_s=\frac{10}{9} \Delta_s$, taking $\Delta_s$ smaller means setting $\Delta_s=\frac{9}{10} \Delta_s$, taking $m$ larger means setting $m= \lfloor1.02 m \rfloor$ and taking $m$ smaller means setting $m = \lfloor m/1.02 \rfloor$. We proved that in a small tube (whose size is given by $r$) of each portion of curve represented in Figure~\ref{japonais} between two bifurcation points, there exists a unique smooth curve representing exactly one steady states of (\ref{eq:mimura_system}). Note that the method cannot prove that those curves really reach the bifurcation points, but we can virtually go as close as we want to those bifurcation points, the only limit being the computational cost. On the other hand, applying the result of Corollary~\ref{corollary:Jinvertible}, we can conclude that along the rigorously computed smooth branches, there are no secondary bifurcation of steady states.
In Table~\ref{fig:computational_cost}, Figure~\ref{fig:Ds_m_as_functions_of_q} and Figure~\ref{espace1}, one has some example of the parameters used, the running time required to compute the branches and some spatial representations of solutions on different branches. \qed
%


\subsection{Proof of Corollary~\ref{corollary:co-existence}}

Consider one portion of branch we have proven. We have the existence of a smooth function $\tilde U$ (defined on $[0,1]$) parametrising our portion of curve, and we know that it lies in a {\em tube}  of radius $r$ around our numerical solution. So we know the initial point and the final point with a precision of $r$ in $\left\Vert\cdot\right\Vert_q$. In particular, if $\overline d_{in}$ (respectively $\overline d_f$) is the numerical value of $d$ from which the branch starts (respectively at which the branch finishes) then there exists a $\tilde d_{in}$ in $\displaystyle{\left[\overline d_{in}-\frac{r}{\rho},\overline d_{in}+\frac{r}{\rho}\right]}$ and a $\tilde d_{f}$ in $\displaystyle{\left[\overline d_{f}-\frac{r}{\rho},\overline d_{f}+\frac{r}{\rho}\right]}$ such that $\tilde U_d(0)=\tilde d_{in}$ and $\tilde U_d(1)=\tilde d_{f}$. Thus, if the branch we are considering is such that
\[
\overline d_{in}-\frac{r}{\rho}>0.006  \quad \mbox{and} \quad \overline d_{f}+\frac{r}{\rho}<0.006,
\]
or reversely 
\[
\overline d_{in}+\frac{r}{\rho}<0.006  \quad \mbox{and} \quad \overline d_{f}-\frac{r}{\rho}>0.006
\]
then the intermediate value theorem yields the existence of an $s$ in $(0,1)$ such that $\tilde U_d(s)=0.006$. Notice that if $d=0.006$ is too close to $\overline d_{in}$ or $\overline d_{f}$ for those hypothesis to be satisfied, you just have to consider several consecutive portions of curve rather than a single one. On the diagram we proved, there are eleven such portions (see Figure~\ref{fig10}). To complete our proof, we need to check that each solution is different from one another, and we are able to do this because the radius $r$ we get in our proofs are small enough. Remember that what we are representing on our bifurcation diagram is the value of $z$ at $0$ and that
\[
\tilde z(0)=\frac{\tilde z_0}{2}+\sum_{n=1}^{\infty}\tilde z_n.
\]
Besides we know that $\left\Vert \tilde U(s) - \overline U_s \right\Vert_q<r$ so 
\[
\left\vert \tilde z(0) - \left(\frac{\overline z_0}{2}+\sum_{n=1}^{m-1}\overline z_n\right) \right\vert < \frac{r}{2}+\sum_{n=1}^{\infty}\frac{r}{n^q}.
\]
In fact, the quantity $\displaystyle{\frac{\overline z_0}{2}+\sum_{n=1}^{m-1}\overline z_n}$ is the one we represented on the diagram, and so the error we did by using this representation is bounded by
\[
\varepsilon_r \bydef r\left(\frac{1}{2}+\sum_{n=1}^{\infty}\frac{1}{n^q}\right).
\]
With the value of $q$ we used and the $r$ we got, those $\varepsilon_r$ were always smaller than $10^{-4}$ and hence the error we made are in fact lying in thickness of the line on the diagram, so we are sure that we have eleven different solutions. \qed


\appendix
\section{Appendix}

\subsection{Sharper estimates for \boldmath $q^*(M)\leq q<2$ \unboldmath } \label{appendix:sharper_estimates}

What we have already proven in the proof of Proposition~\ref{alpha} in (\ref{lim}) is that, for all $ q \in [q^*(M),2)$ and $n\geq M$
\[
\Psi_n^q \leq 2 + 4\sum_{k=1}^{\infty}\frac{1}{k^q} + \chi_n(q),
\]
and since we needed a uniform bound, we took $\displaystyle{2 + 4\sum_{k=1}^{\infty}\frac{1}{k^q} + \chi_M(q)}$. Since $\lim\limits_{n}\chi_n(q)=0$, we would like to take $m$ and hence $M=2m-1$ larger to improve the estimate, but the point of getting a sharp estimate is to allow us to take $m$ smaller, so this does not really make sense. However, there is a way to rigorously compute a bound that is almost optimal, without increasing $m$. Let $\varepsilon>0$, $\exists n_0$, $\forall n\geq n_0$, $\chi_n(q)<\varepsilon$ and since $\chi_n(q)$ is deceasing such $n_0$ can be computed numerically and rigorously using interval arithmetic. Then we have that, for all $ n\geq M$
\begin{small}
\[
\Psi_n^q \leq \max\left(2 + 4\sum_{k=1}^{K}\frac{1}{k^q} + \frac{4}{(q-1)K^{q-1}}+\varepsilon ,\max\limits_{M\leq n <n_0}\left(2 + 2\sum_{k=1}^K\frac{1}{k^q} + \frac{2}{(q-1)K^{q-1}} + \sum_{k=1}^{n-1} \frac{ n^q}{k^q \left(n-k\right)^q}\right)\right),
\]
\end{small}
the max being computed rigorously using interval arithmetic. We can have a bound as sharp as we want by taking $\varepsilon$ small, but at the expense of some computational cost which is in $\mathcal{O}(n_0^2)$. Indeed, when $q$ goes to $2$, $\chi_n(q)$ is equivalent to $\displaystyle{\frac{q}{2-q}\frac{1}{\left\lfloor \frac{n}{2} \right\rfloor^{q-1}}}$ and hence $\displaystyle{n_0 \approx \left(\frac{2q}{2-q}\frac{1}{\varepsilon}\right)^{\frac{1}{q-1}}}$.
 
\subsection{Sharper uniform bounds for \boldmath $s\in[0,1]$ \unboldmath for \boldmath $Y$ \unboldmath} \label{appendix:sharper_s_bound}
\label{sharp}
To prove that $\tilde T$ is a uniform contraction, we have in Section~\ref{Y} to bound uniformly in $s$ terms of the form $g(s) \bydef \left \vert \alpha s^2 + \beta s + \gamma \right\vert$ for $s$ in $[0,1]$. We did that the simplest way, using triangular inequality to say that 
\begin{eqnarray*}
g(s) & \leq &  \vert \alpha \vert s^2 + \vert \beta \vert s + \vert \gamma \vert \\
& \leq &  \vert \alpha \vert + \vert \beta \vert + \vert \gamma \vert,\quad \forall s\in [0,1].
\end{eqnarray*}
However, given the expression of $g$, we can get a better bound which depends on whether the apex of $g$ is between $0$ and $1$ or not. More explicitly,
\[
\forall s\in[0,1],\ g(s)\leq 
\left\{
\begin{aligned}
& \max\left(\vert g(0)\vert ,\vert g(1)\vert\right)=\max\left(\vert \gamma \vert, \vert \alpha + \beta + \gamma \vert \right),\quad \mbox{if }-\frac{\beta}{2\alpha} \notin [0,1],\\
& \max\left(\vert g(0)\vert , \left\vert g\left(\frac{-\beta}{2\alpha}\right)\right\vert,\vert g(1)\vert\right)=\max\left(\vert \gamma \vert ,\left\vert -\frac{\beta^2}{4\alpha}+\gamma\right\vert \vert \alpha + \beta + \gamma \vert\right),\quad \mbox{else.}
\end{aligned}
\right.
\]


\begin{thebibliography}{10}

\bibitem{MR1636690}
Yukio Kan-on and Masayasu Mimura.
\newblock Singular perturbation approach to a {$3$}-component
  reaction-diffusion system arising in population dynamics.
\newblock {\em SIAM J. Math. Anal.}, 29(6):1519--1536 (electronic), 1998.

\bibitem{MR2154033}
Rui Peng and Mingxin Wang.
\newblock Pattern formation in the {B}russelator system.
\newblock {\em J. Math. Anal. Appl.}, 309(1):151--166, 2005.

\bibitem{MR0288640}
Michael~G. Crandall and Paul~H. Rabinowitz.
\newblock Bifurcation from simple eigenvalues.
\newblock {\em J. Functional Analysis}, 8:321--340, 1971.

\bibitem{MR0301587}
Paul~H. Rabinowitz.
\newblock Some global results for nonlinear eigenvalue problems.
\newblock {\em J. Functional Analysis}, 7:487--513, 1971.

\bibitem{MR2251792}
Masato Iida, Masayasu Mimura, and Hirokazu Ninomiya.
\newblock Diffusion, cross-diffusion and competitive interaction.
\newblock {\em J. Math. Biol.}, 53(4):617--641, 2006.

\bibitem{turing}
A.M. Turing.
\newblock The chemical basis of morphogenesis.
\newblock {\em Phil. Trans. R. Soc. Lond. B}, 237:37--72, 1952.

\bibitem{MR1242579}
Herbert Amann.
\newblock Nonhomogeneous linear and quasilinear elliptic and parabolic boundary
  value problems.
\newblock In {\em Function spaces, differential operators and nonlinear
  analysis ({F}riedrichroda, 1992)}, volume 133 of {\em Teubner-Texte Math.},
  pages 9--126. Teubner, Stuttgart, 1993.

\bibitem{MR1974423}
Y.~S. Choi, Roger Lui, and Yoshio Yamada.
\newblock Existence of global solutions for the
  {S}higesada-{K}awasaki-{T}eramoto model with weak cross-diffusion.
\newblock {\em Discrete Contin. Dyn. Syst.}, 9(5):1193--1200, 2003.

\bibitem{MR1616969}
Yuan Lou, Wei-Ming Ni, and Yaping Wu.
\newblock On the global existence of a cross-diffusion system.
\newblock {\em Discrete Contin. Dynam. Systems}, 4(2):193--203, 1998.

\bibitem{MR2437576}
Hirofumi Izuhara and Masayasu Mimura.
\newblock Reaction-diffusion system approximation to the cross-diffusion
  competition system.
\newblock {\em Hiroshima Math. J.}, 38(2):315--347, 2008.

\bibitem{MR2630003}
Jan~Bouwe van~den Berg, Jean-Philippe Lessard, and Konstantin Mischaikow.
\newblock Global smooth solution curves using rigorous branch following.
\newblock {\em Math. Comp.}, 79(271):1565--1584, 2010.

\bibitem{MR2338393}
Sarah Day, Jean-Philippe Lessard, and Konstantin Mischaikow.
\newblock Validated continuation for equilibria of {PDE}s.
\newblock {\em SIAM J. Numer. Anal.}, 45(4):1398--1424 (electronic), 2007.

\bibitem{MR910499}
H.~B. Keller.
\newblock {\em Lectures on numerical methods in bifurcation problems},
  volume~79 of {\em Tata Institute of Fundamental Research Lectures on
  Mathematics and Physics}.
\newblock Published for the Tata Institute of Fundamental Research, Bombay,
  1987.
\newblock With notes by A. K. Nandakumaran and Mythily Ramaswamy.

\bibitem{MR660633}
Shui~Nee Chow and Jack~K. Hale.
\newblock {\em Methods of bifurcation theory}, volume 251 of {\em Grundlehren
  der Mathematischen Wissenschaften [Fundamental Principles of Mathematical
  Science]}.
\newblock Springer-Verlag, New York, 1982.

\bibitem{MR2718657}
Marcio Gameiro and Jean-Philippe Lessard.
\newblock Analytic estimates and rigorous continuation for equilibria of
  higher-dimensional {PDE}s.
\newblock {\em J. Differential Equations}, 249(9):2237--2268, 2010.

\end{thebibliography}
\end{document}